\documentclass[11pt]{amsart} 
\usepackage{xspace,amssymb,amsmath,amscd,amsthm,epsfig,etoolbox, hyperref, ytableau, float, mathrsfs, mdframed, enumerate, tikz, mathtools}
\hypersetup{colorlinks = true, citecolor = blue, linkcolor = blue}
\usepackage[textwidth = 20mm]{todonotes}
\usepackage[capitalise, nameinlink]{cleveref}
\setuptodonotes{color = pink}
\reversemarginpar
\newlength{\defbaselineskip}
\theoremstyle{plain}
\newtheorem{theorem}{Theorem}

\newtheorem{prop}[theorem]{Proposition}
\newtheorem{corollary}[theorem]{Corollary}

\theoremstyle{definition}

\newtheorem{example}[theorem]{Example}
\newtheorem{remark}[theorem]{Remark}

\newcommand{\boks}[1]{\ytableausetup{boxsize = #1 cm}}

\newcommand{\y}[1]{\ydiagram{#1}}
\newcommand{\yt}[1]{\ytableaushort{#1}}

\newcommand{\Z}{\mathbb{Z}}
\newcommand{\Q}{\mathbb{Q}}

\newcommand{\xx}{\mathbf{x}}
\newcommand{\mcM}{\mathcal{M}}
\newcommand{\vdt}{\Vdash}
\newcommand{\PLambda}{\textsf{P}\mathsf{\Lambda}}
\renewcommand{\top}{\mathrm{top}}

\renewcommand{\bot}{\mathrm{left}}
\newcommand{\vn}{\varnothing}
\DeclareMathOperator{\area}{area}
\DeclareMathOperator{\dg}{dg}

\DeclareMathOperator{\sgn}{sgn}

\DeclareMathOperator{\ICRHT}{ICRHT} 
\DeclareMathOperator{\ICRPT}{ICRPT} 
\DeclareMathOperator{\PTBT}{PTBT} 
\DeclareMathOperator{\HTBT}{HTBT} 
\DeclareMathOperator{\ETBT}{ETBT} 
\DeclareMathOperator{\TRHT}{TRHT} 
\DeclareMathOperator{\TPRT}{TPRT} 
\DeclareMathOperator{\wt}{wt}

\let\bbmatrix\bordermatrix
\patchcmd{\bbmatrix}{8.75}{8.75}{}{}
\patchcmd{\bbmatrix}{\left(}{\left[}{}{}
\patchcmd{\bbmatrix}{\right)}{\right]}{}{}

\newlength{\cellsize}
\newcommand\tableau[1]{
\vcenter{
\let\\=\cr
\baselineskip=-16000pt
\lineskiplimit=16000pt
\lineskip=0pt
\halign{&\tableaucell{##}\cr#1\crcr}}}

\newcommand{\tableaucell}[1]{{%
\def \arg{#1}\def \void{}%
\ifx \void \arg
\vbox to \cellsize{\vfil \hrule width \cellsize height 0pt}%
\else
\unitlength=\cellsize
\begin{picture}(1,1)
\put(0,0){\makebox(1,1){$#1$}}
\put(0,0){\line(1,0){1}}
\put(0,1){\line(1,0){1}}
\put(0,0){\line(0,1){1}}
\put(1,0){\line(0,1){1}}
\end{picture}%
\fi}}

\begin{document}

\title{Transition matrices and Pieri-type rules for polysymmetric functions}
 
\author{Aditya Khanna
\and Nicholas A. Loehr}
\thanks{This work was supported by a grant from the Simons
 Foundation/SFARI (\#633564 to Nicholas Loehr).} 
\address{Dept. of Mathematics, Virginia Tech, Blacksburg, VA 24061-0123}
\email{adityakhanna@vt.edu, nloehr@vt.edu}
 
\begin{abstract}
Asvin G and Andrew O'Desky recently introduced
the graded algebra $\PLambda$ of polysymmetric functions as a 
generalization of the algebra $\Lambda$ of symmetric functions.
This article develops combinatorial formulas
for some multiplication rules and transition matrix entries for $\PLambda$
that are analogous to well-known classical formulas for $\Lambda$.
In more detail, we consider pure tensor bases $\{s^{\otimes}_{\tau}\}$,
$\{p^{\otimes}_{\tau}\}$, and $\{m^{\otimes}_{\tau}\}$ for $\PLambda$ that 
arise as tensor products of the classical Schur basis, power-sum basis, and
monomial basis for $\Lambda$. We find expansions in these bases
of the non-pure bases $\{P_{\delta}\}$, $\{H_{\delta}\}$,
$\{E^+_{\delta}\}$, and $\{E_{\delta}\}$ studied by Asvin G and O'Desky.
The answers involve tableau-like structures generalizing semistandard
tableaux, rim-hook tableaux, and the brick tabloids of
E\u{g}ecio\u{g}lu and Remmel. These objects arise by iteration of new Pieri-type
rules that give expansions of products such as $s^{\otimes}_{\sigma}H_{\delta}$,
$p^{\otimes}_{\sigma}E_{\delta}$, etc.
\end{abstract}

\maketitle

 

\section{Introduction}
\label{sec:intro} 

The ring $\Lambda$ of symmetric functions is an object of great interest in
modern algebraic combinatorics. Recently, Asvin G and Andrew O'Desky 
introduced a generalization $\PLambda$ called the ring of polysymmetric 
functions~\cite{polysymm}. Our goal in this paper is to extend some
of the rich combinatorial theory for symmetric functions to the new
setting of polysymmetric functions. In particular, we develop combinatorial 
formulas for some multiplication rules and transition matrix entries for 
$\PLambda$ that are analogous to well-known classical formulas for $\Lambda$.

\subsection{Review of Symmetric Functions}
\label{subsec:rev-symm}

We assume the reader has some prior familiarity with symmetric functions;
background material may be found in texts such 
as~\cite{loehr-comb,macd,stanvol2}. We briefly recall some fundamental
notation and terminology. An \emph{integer partition of $n$} is a weakly
decreasing sequence $\lambda=(\lambda_1,\lambda_2,\ldots,\lambda_\ell)$
of positive integers with sum $n$. We call $\lambda_i$ the $i$th
\emph{part} of $\lambda$, and let $\ell(\lambda)=\ell$ be the 
number of nonzero parts of $\lambda$. 
We write $|\lambda|=n$ or $\area(\lambda)=n$ or $\lambda\vdash n$ to mean
that $\lambda$ is an integer partition of $n$. We write 
$\lambda=(1^{m_1}2^{m_2}3^{m_3}\cdots)$ to indicate that $\lambda$
is a partition with $m_1$ parts equal to $1$, $m_2$ parts equal to $2$,
and so on.  We denote the number of times $i$ appears in $\lambda$ 
by $m_i(\lambda)$.  A \emph{symmetric function over 
$\Q$} is a formal power series of bounded degree in countably many variables 
with coefficients in $\Q$, say $f=f(\xx)=f(x_1,x_2,\ldots,x_m,\ldots)$, 
that remains unchanged under any permutation of the variables $x_i$. Letting 
each variable $x_i$ have degree $1$, the set $\Lambda^n$ of homogeneous 
symmetric functions of degree $n$ is a vector space of dimension $p(n)$, 
the number of integer partitions of $n$.  The set of all symmetric functions 
is a graded $\Q$-algebra $\Lambda=\bigoplus\limits_{n\geq 0} \Lambda^n$.

Bases of the vector space $\Lambda^n$ are naturally indexed by integer 
partitions of $n$. The \emph{monomial symmetric function}
$m_{\lambda}(\xx)$ is the formal sum of all distinct monomials obtained by
permuting the subscripts in $x_1^{\lambda_1}x_2^{\lambda_2}\cdots 
x_{\ell}^{\lambda_\ell}$. The \emph{complete symmetric function}
$h_k(\xx)$ is the sum of all monomials $x_{i_1}x_{i_2}\cdots x_{i_k}$
where $1\leq i_1\leq i_2\leq\cdots\leq i_k$. The \emph{elementary symmetric
function} $e_k(\xx)$ is the sum of all monomials $x_{i_1}x_{i_2}\cdots x_{i_k}$
where $1\leq i_1<i_2<\cdots<i_k$. The \emph{power-sum symmetric function}
$p_k(\xx)$ is $x_1^k+x_2^k+\cdots+x_m^k+\cdots$. For any list of positive
integers $\alpha=(\alpha_1,\alpha_2,\ldots,\alpha_s)$, we define
\[ h_{\alpha}(\xx)=\prod_{i=1}^s h_{\alpha_i}(x),\quad
   e_{\alpha}(\xx)=\prod_{i=1}^s e_{\alpha_i}(x),\quad
   p_{\alpha}(\xx)=\prod_{i=1}^s p_{\alpha_i}(x). \]
The \emph{Schur symmetric function} $s_{\lambda}(\xx)$ can be defined
as $s_{\lambda}(\xx)=\sum_{\mu} K_{\lambda,\mu}m_{\mu}(\xx)$, 
where the \emph{Kostka number} $K_{\lambda,\mu}$ is the number of 
semistandard Young tableaux of shape $\lambda$ and content $\mu$.
It is known that each of the sets $\{m_{\lambda}:\lambda\vdash n\}$,
$\{h_{\lambda}:\lambda\vdash n\}$, $\{e_{\lambda}:\lambda\vdash n\}$,
$\{p_{\lambda}:\lambda\vdash n\}$, and
$\{s_{\lambda}:\lambda\vdash n\}$ is a basis of $\Lambda^n$.
It follows that each of the sets $\{h_k: k\in\Z_{>0}\}$,
 $\{e_k:k\in\Z_{>0}\}$, and $\{p_k: k\in\Z_{>0}\}$ is algebraically
independent over $\Q$. This leads to an abstract description of $\Lambda$
as a polynomial ring $\Lambda=\Q[h_k:k>0]$ in formal indeterminates $h_k$
where $\deg(h_k)=k$. Similarly, we can think of $\Lambda$ as a polynomial
ring in the $e_k$ or the $p_k$, where $\deg(e_k)=k=\deg(p_k)$.
  
Transition matrices between bases of $\Lambda^n$ often exhibit interesting
combinatorics~\cite{remmel-trans,eg-rem}. Given indexed bases 
$\{f_{\lambda}: \lambda\vdash n\}$ and $\{g_{\lambda}:\lambda\vdash n\}$
of $\Lambda^n$, the \emph{transition matrix} $\mcM(f,g)$ is the unique
matrix (with rows and columns indexed by partitions of $n$) such that
\begin{equation}\label{eq:trans-mat1}
 f_{\mu}=\sum_{\lambda\vdash n} \mcM(f,g)_{\lambda,\mu}g_{\lambda}.
\end{equation}
For example, the definition of Schur functions (given above) states that
$\mcM(s,m)_{\lambda,\mu}$ is the Kostka number $K_{\mu,\lambda}$.
It is known that $\mcM(h,s)_{\lambda,\mu}=K_{\lambda,\mu}$, so that
$\mcM(s,m)$ is the transpose of $\mcM(h,s)$. It is routine to check
that matrix inversion switches the roles of the input basis and the
output basis: $\mcM(g,f)=\mcM(f,g)^{-1}$.  If $\{k_{\lambda}\}$ is
another basis of $\Lambda^n$, then $\mcM(f,k)$ is the matrix product
$\mcM(g,k)\mcM(f,g)$.

\subsection{Polysymmetric Functions}
\label{subsec:polysymm}

For each positive integer $d$, let $\Lambda_{(d)}$ be a copy of the
ring $\Lambda$ of symmetric functions where all degrees are multiplied 
by $d$. The $\Q$-algebra of \emph{polysymmetric functions} may be defined
abstractly as the tensor product
\[ \PLambda=\Lambda_{(1)}\otimes\Lambda_{(2)}\otimes\cdots\otimes
\Lambda_{(d)}\otimes\cdots. \]
To get a more concrete description, we view $\Lambda_{(d)}$ as the ring of 
symmetric functions in a variable set $\xx_{d*}=\{x_{d,1},x_{d,2},\ldots\}$,
where $\deg(x_{d,i})=d$ for all $i\geq 1$. Then $\PLambda$ appears as a
particular subalgebra of the $\Q$-algebra $\Q[[\xx_{**}]]$ of formal series of 
bounded degree in all the variables $x_{d,i}$ for $d,i\in\Z_{>0}$.
A formal series $f=f(\xx_{**})$ belongs to $\PLambda$ iff for each fixed $d$,
$f$ is unchanged by any permutation of the variables in $\xx_{d*}$.
An isomorphism between the abstract and concrete versions of $\PLambda$
is defined by sending the pure tensor $f_1\otimes f_2\otimes f_3\otimes\cdots$ 
to the formal series 
$f_1(\xx_{1*})f_2(\xx_{2*})f_3(\xx_{3*})\cdots$. Like $\Lambda$,
$\PLambda$ is a graded algebra: $\PLambda=\bigoplus_{n\geq 0} \PLambda^n$,
where $\PLambda^n$ is the vector space of homogeneous polysymmetric functions
of degree $n$.

Bases for $\PLambda^n$ are naturally indexed by \textit{(splitting) types}, 
which we discuss next.
A \emph{block} is an ordered pair of positive integers $(d,m)$, which we 
usually write as $d^m$. We say $d^m$ has \emph{degree} $d$, \emph{multiplicity}
$m$, and \emph{weight} $dm$.  We order blocks by writing $a^b\geq d^e$ to mean
either $a>d$, or $a=d$ and $b\geq e$. A \emph{type} of weight $n$ 
is a weakly decreasing sequence of blocks 
$\tau=(d_1^{m_1},d_2^{m_2},\ldots,d_s^{m_s})$ such that
$d_1m_1+d_2m_2+\cdots+d_sm_s=n$. We write $|\tau|=n$ or $\tau\vdt n$ to 
mean that $\tau$ is a type of weight $n$. We call $s$ 
the \textit{length of $\tau$} and write $s=\ell(\tau)$. 
For fixed $d$, let $\tau|_d$
(sometimes abbreviated as $\tau_d$) 
be the partition formed by taking the multiplicities of the blocks of $\tau$
of degree $d$. For example, $\tau=(3^43^43^22^32^22^12^11^41^31^31^1)$
is a type of weight $55$ with $\tau|_3=(4,4,2)$, $\tau|_2=(3,2,1,1)$, 
and $\tau|_1=(4,3,3,1)$.  We may abbreviate any type $\tau$ by writing 
$\tau=(1^{\tau|_1}2^{\tau|_2}3^{\tau|_3}\cdots)$.
The \emph{sign} of type $\tau$ is $\sgn(\tau)=\prod_{i=1}^k (-1)^{m_i}$.
The power of $-1$ in $\sgn(\tau)$ is 
$\sum_{i=1}^k m_i=\sum_{i=1}^k \area(\tau|_i)$.

\begin{remark}
Types of weight $n$ encode the possible ways a polynomial $p(x)\in\Q[x]$
of degree $n$ can split into irreducible factors. For example,
$p=(x^2+1)^3(x^2-2)^3(x^2-3)(x-1)^2(x-2)^2$ has associated type
$\tau=(2^32^32^11^21^2)$.
\end{remark}

Suppose $\{f_{\lambda}\}$ is any fixed basis for $\Lambda$,
where $\lambda$ ranges over integer partitions, 
and $f_{\lambda}\in\Lambda^n$ whenever $\lambda\vdash n$. By the general
theory of tensor products, it follows that the set of tensor products
$f_{\lambda_{(1)}}\otimes f_{\lambda_{(2)}}\otimes f_{\lambda_{(3)}}
 \otimes\cdots$, 
where all but finitely many $f_{\lambda_{(d)}}$ are equal to $1$,
is a basis for the vector space $\PLambda$. We can identify
the list $(\lambda_{(1)},\lambda_{(2)},\ldots)$ with the
type $\tau=(1^{\lambda_{(1)}}2^{\lambda_{(2)}}\cdots)$. Define
\[ f_{\tau}^{\otimes}=f_{\tau|_1}\otimes f_{\tau|_2}\otimes\cdots
                    =\prod_{d\geq 1} f_{\tau|_d}(\xx_{d*}). \]
This is a homogeneous element of $\PLambda$ of degree
$\sum_{d\geq 1} d\area(\tau|_d)=|\tau|$.
Letting $\tau$ range over all types, 
we get a basis $\{f_{\tau}^{\otimes}\}$ of $\PLambda$.
For each $n\geq 0$, $\{f_{\tau}^{\otimes}:\tau\vdt n\}$ 
is a basis of $\PLambda^n$. We call these bases of $\PLambda$
and $\PLambda^n$ the \emph{pure tensor bases} associated with 
the given basis $\{f_{\lambda}\}$ of $\Lambda$. 

\begin{example}
Given $\tau=(4^44^34^12^22^22^12^11^31^31^11^1) 
           =(1^{3,3,1,1}2^{2,2,1,1}4^{4,3,1})$,
\[ m_{\tau}^{\otimes}=m_{3311}\otimes m_{2211}\otimes 1\otimes m_{431}\otimes 1
\otimes 1\otimes\cdots=m_{3311}(\xx_{1*})m_{2211}(\xx_{2*})m_{431}(\xx_{4*}). \]
Hereafter, we often omit trailing $1$s in the tensor product presentation
of a polysymmetric function.
\end{example}

\subsection{The Bases $H$, $E^+$, $E$, and $P$}
\label{subsec:new-bases}

The authors of~\cite{polysymm} introduced four bases of $\PLambda$,
denoted by $\{H_{\tau}\}$, $\{E^+_{\tau}\}$, $\{E_{\tau}\}$, and $\{P_{\tau}\}$,
that are not pure tensor bases. These are polysymmetric analogues
of the symmetric functions $h_{\mu}$, $e_{\mu}$, and $p_{\mu}$,
defined as follows. Order the subscripts of variables in $\xx_{**}$ 
lexicographically: $(i,j)\leq (k,\ell)$ means $i<k$, or $i=k$ and $j\leq\ell$.
For each positive integer $d$, define
\begin{equation}\label{eq:Hd}
 H_d=\sum_{\substack{(i_1,j_1)\leq (i_2,j_2)\leq\cdots\leq (i_s,j_s)
  \\ i_1+i_2+\cdots+i_s=d}} x_{i_1,j_1}x_{i_2,j_2}\cdots x_{i_s,j_s},
\end{equation}
which is the sum of all distinct monomials of degree $d$. Define
\begin{equation}\label{eq:E+d}
 E_d^+=\sum_{\substack{(i_1,j_1)<(i_2,j_2)<\cdots<(i_s,j_s)
  \\ i_1+i_2+\cdots+i_s=d}} x_{i_1,j_1}x_{i_2,j_2}\cdots x_{i_s,j_s},
\end{equation}
which is the sum of monomials of degree $d$ where no variable $x_{ij}$ appears
more than once within any given monomial. 
Such monomials are called \textit{square-free}.
Define
\begin{equation}\label{eq:Ed}
 E_d=\sum_{\substack{(i_1,j_1)<(i_2,j_2)<\cdots<(i_s,j_s)
  \\ i_1+i_2+\cdots+i_s=d}} (-1)^s x_{i_1,j_1}x_{i_2,j_2}\cdots x_{i_s,j_s},
\end{equation}
which is a signed variation of $E_d^+$. Define
\begin{equation}\label{eq:Pd}
 P_d=\sum_{k|d} k\sum_{j\geq 1} x_{k,j}^{d/k},
\end{equation}
where ``$\sum_{k|d}$'' indicates a sum over positive divisors $k$ of $d$.
It is routine to check
that $H_d$, $E_d^+$, $E_d$, and $P_d$ all belong to $\PLambda^d$.

For any block $d^m$, define $H_{d^m}=H_d(\xx_{**}^m)$, which means
that every variable $x_{ij}$ appearing in every monomial of $H_d$
gets replaced by $x_{ij}^m$. Similarly, define $E^+_{d^m}=E^+_d(\xx_{**}^m)$,
$E_{d^m}=E_d(\xx_{**}^m)$, and $P_{d^m}=P_d(\xx_{**}^m)$. These objects
are all in $\PLambda^{dm}$.  Finally, for any ordered
sequence of blocks $\delta=(d_1^{m_1},d_2^{m_2},\ldots,d_t^{m_t})$, define
\[ H_{\delta}=\prod_{i=1}^t H_{d_i^{m_i}},\ 
   E_{\delta}^+=\prod_{i=1}^t E^+_{d_i^{m_i}},\ 
   E_{\delta}=\prod_{i=1}^t E_{d_i^{m_i}},\ 
   P_{\delta}=\prod_{i=1}^t P_{d_i^{m_i}}. \]
In particular, this defines $H_{\tau}$ (etc.) when $\tau$ is a type.
It is shown in~\cite{polysymm} that each of the sets
$\{H_{\tau}:\tau\vdt n\}$, $\{E^+_{\tau}:\tau\vdt n\}$,
$\{E_{\tau}:\tau\vdt n\}$, and $\{P_{\tau}:\tau\vdt n\}$ is 
a linear basis of $\PLambda^n$.
As in the case of $\Lambda$, this leads to an alternate algebraic
characterization of $\PLambda$ as an abstract polynomial ring.
Starting with formal indeterminates $H_{d^m}$ for each block $d^m$,
we can think of $\PLambda$ as $\Q[H_{d^m}: d,m>0]$, where $H_{d^m}$
has degree $dm$. Similarly, $\PLambda=\Q[E^+_{d^m}: d,m>0]
 =\Q[E_{d^m}: d,m>0]=\Q[P_{d^m}: d,m>0]$.

\subsection{Transition Matrices for $\PLambda$}
\label{subsec:trans-mat}

Our main goal in this paper is to develop the combinatorics of certain
transition matrices between bases of $\PLambda^n$. We use notation analogous
to the symmetric case. Given bases $\{F_{\tau}:\tau\vdt n\}$ and
$\{G_{\tau}:\tau\vdt n\}$ of $\PLambda^n$,
the \emph{transition matrix} $\mcM(F,G)$ is the unique matrix (with
rows and columns indexed by types of weight $n$) such that
\begin{equation}\label{eq:trans-mat2}
 F_{\sigma}=\sum_{\tau\vdt n} \mcM(F,G)_{\tau,\sigma}G_{\tau}.
\end{equation}

In the special case of pure tensor bases, we can immediately find
transition matrices for $\PLambda$ if we know the corresponding
transition matrices for $\Lambda$. 

\begin{prop}\label{prop:transmat-tensor}
Let $\{f_{\lambda}\}$ and $\{g_{\lambda}\}$ be bases of $\Lambda$
such that $f_{\lambda},g_{\lambda}\in\Lambda^n$ whenever $\lambda\vdash n$.
Let $F_{\tau}=f_{\tau}^{\otimes}$ and $G_{\tau}=g_{\tau}^{\otimes}$ be
the corresponding pure tensor bases. For all types $\sigma,\tau$,
\[ \mcM(F,G)_{\tau,\sigma}=\prod_{d\geq 1} \mcM(f,g)_{\tau|_d,\sigma|_d}. \]
\end{prop}
\begin{proof}
For $\sigma\vdt n$, we compute 
\begin{align*}
F_{\sigma}=f_{\sigma}^{\otimes} &= \prod_{d\geq 1} f_{\sigma|_d}(\xx_{d*})
  =\prod_{d\geq 1} \sum_{\lambda_{(d)}\vdash \area(\sigma|_d)}
    \mcM(f,g)_{\lambda_{(d)},\sigma|_d} g_{\lambda_{(d)}}(\xx_{d*})
\\ &= \sum_{\lambda_{(1)}\vdash \area(\sigma|_1)}\cdots
      \sum_{\lambda_{(d)}\vdash \area(\sigma|_d)}\cdots
\prod_{d\geq 1} \mcM(f,g)_{\lambda_{(d)},\sigma|_d} g_{\lambda_{(d)}}(\xx_{d*})
= \sum_{\tau\vdt n} \prod_{d\geq 1} \mcM(f,g)_{\tau|_d,\sigma|_d} 
g_{\tau}^{\otimes}.
\end{align*}
where in the last step we set $\tau=(1^{\lambda_{(1)}}2^{\lambda_{(2)}}\cdots
 d^{\lambda_{(d)}}\cdots)$. So the coefficient of $G_{\tau}$ in $F_{\sigma}$
is $\prod_{d\geq 1} \mcM(f,g)_{\tau|_d,\sigma|_d}$, as needed.
\end{proof}

\subsection{Main Results}
\label{subsec:results}

Transition matrices involving the bases $H$, $E$, $E^+$, and $P$ are more 
subtle. In this paper, we find formulas for entries in
the following transition matrices:
\begin{itemize}
\item $\mcM(P,s^{\otimes})$, $\mcM(H,s^{\otimes})$, 
$\mcM(E^+,s^{\otimes})$, $\mcM(E,s^{\otimes})$ (\cref{sec:s-expand}).
\item $\mcM(P,p^{\otimes})$, $\mcM(H,p^{\otimes})$, 
 $\mcM(E^+,p^{\otimes})$, $\mcM(E,p^{\otimes})$ (\cref{sec:p-expand}).
\item $\mcM(P,m^{\otimes})$, $\mcM(H,m^{\otimes})$, 
$\mcM(E^+,m^{\otimes})$, $\mcM(E,m^{\otimes})$ (\cref{sec:m-expand}).
\end{itemize}

Our $s^{\otimes}$-expansions involve tableau-like structures that 
arise by iteration of certain rules analogous to the Pieri rules 
(giving the Schur expansions of $s_{\mu}h_k$ and $s_{\mu}e_k$) and the
Murnaghan--Nakayama rule (giving the Schur expansion of $s_{\mu}p_k$).
Letting $\delta=(d_1^{m_1},d_2^{m_2},\ldots,d_t^{m_t})$ be any ordered
sequence of blocks, we prove Pieri-type rules for the $s^{\otimes}$-expansions
of $s^{\otimes}_{\sigma}P_{\delta}$, 
   $s^{\otimes}_{\sigma}H_{\delta}$, 
   $s^{\otimes}_{\sigma}E^+_{\delta}$,
and $s^{\otimes}_{\sigma}E_{\delta}$.
Our $p^{\otimes}$-expansions have a more algebraic flavor and reveal
some identities for $\PLambda$ analogous to corresponding power-sum
identities for $\Lambda$. 
Our $m^{\otimes}$-expansions complement some comparable results 
in~\cite{polysymm}. We give combinatorial descriptions of transition
matrix entries using objects generalizing the brick tabloids
studied by E\u{g}ecio\u{g}lu and Remmel~\cite{eg-rem}.
We also prove Pieri-like rules for the $p^{\otimes}$-expansions
of $p_{\sigma}^{\otimes}F_{\delta}$ and the $m^{\otimes}$-expansions
of $m_{\sigma}^{\otimes}F_{\delta}$ where $F$ is $P$, $H$, $E^+$, or $E$.

\section{Expansions in the $s^{\otimes}$ Basis}
\label{sec:s-expand}

Recall that $\{s_{\lambda}\}$ is the Schur basis of $\Lambda$, and
$\{s_{\tau}^{\otimes}\}$ is the associated pure tensor basis of $\PLambda$.
This section provides combinatorial formulas for the coefficients
in the $s^{\otimes}$-expansions of $s^{\otimes}_{\sigma}F$ where
$F$ is $P_{d^m}$, $H_{d^m}$, $E^+_{d^m}$, $E_{d^m}$, or any product
of such factors. As special cases, we find the transition matrices
$\mcM(P,s^{\otimes})$, $\mcM(H,s^{\otimes})$, $\mcM(E^+,s^{\otimes})$, 
and $\mcM(E,s^{\otimes})$.

\subsection{Rule for $s^{\otimes}_{\sigma}P_{d^m}$.}
\label{subsec:s-times-P1}

Before stating the rule for the $s^{\otimes}$-expansion of 
$s^{\otimes}_{\sigma}P_{d^m}$, we review the analogous classical
rule for the Schur expansion of $s_{\mu}p_k$. Given an integer partition 
$\mu=(\mu_1\geq\mu_2\geq\cdots\geq\mu_s)$, the \emph{diagram of $\mu$}
is the set $\dg(\mu)=\{(i,j)\in\Z^2: 1\leq i\leq s,\ 1\leq j\leq \mu_i\}$.
We visualize the diagram of $\mu$ by drawing $s$ rows of left-justified
unit boxes with $\mu_i$ boxes in the $i$th row from the top. 
The \emph{conjugate partition} $\mu'$ is the partition whose diagram
is obtained from $\dg(\mu)$ by interchanging rows and columns.
Given $\mu$ and another integer partition $\nu$ such that
$\dg(\mu)\subseteq\dg(\nu)$, the \emph{skew shape} $\nu/\mu$
is the set difference $\dg(\nu)\setminus\dg(\mu)$. We visualize 
a skew shape as the collection of boxes in the diagram for $\nu$
that are outside the diagram for $\mu$. A skew shape $\nu/\mu$
is a \emph{$k$-ribbon} (or a \textit{$k$-rim-hook} or a \textit{$k$-border strip}) if it consists of $k$ boxes that can be
labeled $b_1,\ldots,b_k$ so that, for $1<i\leq k$, $b_i$ is one unit left 
of $b_{i-1}$ or one unit below $b_{i-1}$.  Equivalently, this 
means that $\nu/\mu$ is a connected strip of $k$ boxes on the southeast 
border of $\dg(\nu)$ that contains no $2\times 2$ square.
The \emph{sign} of a $k$-ribbon $\nu/\mu$ that has boxes in $r$ different rows
is $\sgn(\nu/\mu)=(-1)^{r-1}$. The next result is often called the
Murnaghan--Nakayama Rule, the Pieri Rule for Power-Sums, or the Slinky Rule.

\begin{prop}~\cite[Theorem 10.46]{loehr-comb}.\label{prop:slinky}
For any integer partition $\mu$ and positive integer $k$,
\[ s_{\mu}p_k=\sum_{\nu:\ \nu/\mu\mbox{ is a $k$-ribbon}} 
\sgn(\nu/\mu)s_{\nu}.\]
\end{prop}

\begin{example}
\boks{0.4}
\ytableausetup{aligntableaux = top}
We compute $s_{(3,2)}p_4=s_{(7,2)}-s_{(5,4)}-s_{(3,3,3)}+s_{(3,2,2,1,1)}
 -s_{(3,2,1,1,1,1)}$ using the following diagrams,
where the boxes in the $4$-ribbon $\nu/\mu$ are shaded in gray.
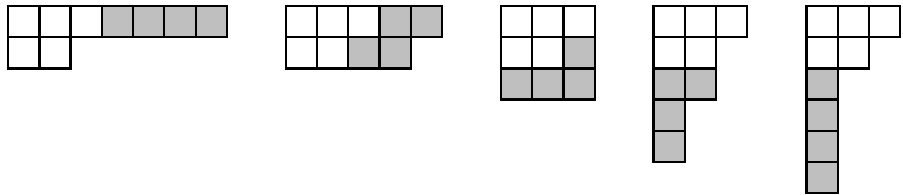
\[ \y{3,2}*[*(lightgray)]{3+4} \qquad 
   \y{3,2}*[*(lightgray)]{3+2,2+2}\qquad
    \y{3,2}*[*(lightgray)]{0,2+1,3}\qquad
   \y{3,2}*[*(lightgray)]{0,0,2,1,1}\qquad
 \y{3,2}*[*(lightgray)]{0,0,1,1,1,1}\]
\end{example}

Turning to the polysymmetric case,
let $\sigma=(1^{\sigma|_1}2^{\sigma|_2}\cdots i^{\sigma|_i}\cdots)$ be
a fixed type. The \emph{tensor diagram of $\sigma$} is the formal symbol
\[ \dg(\sigma)=\dg(\sigma|_1)\otimes\dg(\sigma|_2)\otimes\cdots
\otimes\dg(\sigma|_i)\otimes\cdots. \]
We draw $\dg(\sigma)$ as a succession of partition diagrams joined
by tensor signs; we draw $\vn$ in any position $i$ where $\sigma|_i$
is the empty partition. For example, the diagram of 
$\sigma=(1^{3,3,1}2^{2,1,1,1}4^{3,2,2,1})$ is
\\ \ytableausetup{aligntableaux = top}
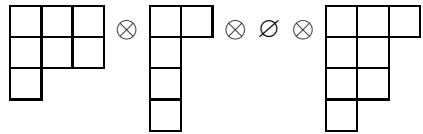
\[ \y{3,3,1}\,\otimes\,\y{2,1,1,1}\,\otimes\,\varnothing\,\otimes\,
   \y{3,2,2,1} \]
The next theorem computes $s_{\sigma}^{\otimes}P_{d^m}$ by adding certain
signed weighted ribbons to $\dg(\sigma)$ according to particular rules.
If $R$ is a ribbon added to the shape in position $i$ of the tensor
diagram, we let $\wt(R)=i$.

\begin{theorem}\label{thm:s-times-P1}
For any type $\sigma$ and block $d^m$,
\[ s_{\sigma}^{\otimes}P_{d^m}=\sum_{\tau} \sgn(R)\wt(R)s_{\tau}^{\otimes}, \]
where we sum over types $\tau$ that arise from $\sigma$ by adding
a $(dm/k)$-ribbon $R$ to $\dg(\sigma|_k)$ for some $k>0$ that divides $d$.
\end{theorem}
\begin{proof}
Combining~\eqref{eq:Pd} with the subsequent definition of $P_{d^m}$, we find %
\begin{equation}\label{eq:Pdm-in-p} 
 P_{d^m}=\sum_{k|d} k\sum_{j\geq 1} x_{k,j}^{dm/k}
          =\sum_{k|d} kp_{dm/k}(\xx_{k*})
          =\sum_{k|d} 1\otimes\cdots\otimes
 1\otimes kp_{dm/k}\otimes 1\otimes\cdots, 
\end{equation} 
where $kp_{dm/k}$ occurs in the $k$th tensor factor.
Multiplying $s_{\sigma}^{\otimes}$ by this expression, we get
\[ s_{\sigma}^{\otimes}P_{d^m}
  =\sum_{k|d} s_{\sigma|_1}\otimes s_{\sigma|_2}\otimes\cdots
 \otimes s_{\sigma|_k}\cdot kp_{dm/k}\otimes s_{\sigma|_{k+1}}\otimes\cdots. \]
For a fixed choice of $k$ dividing $d$,
the classical Pieri rule replaces the factor $s_{\sigma|_k}p_{dm/k}$ by
the sum of $\sgn(\nu/(\sigma|_k))s_{\nu}$ over all $\nu$ such that
$\nu/(\sigma|_k)$ is a $(dm/k)$-ribbon. We weight such a ribbon by $k$
to account for the extra factor of $k$. Adding over all choices of $k$
gives the formula in the theorem.
\end{proof}

\begin{example} 
Let $\sigma=(3^22^11^41^3)$.  We compute $s_{\sigma}^{\otimes}P_{3^2}$ 
using the following diagrams, where the boxes in the
newly added ribbons are shaded in gray.
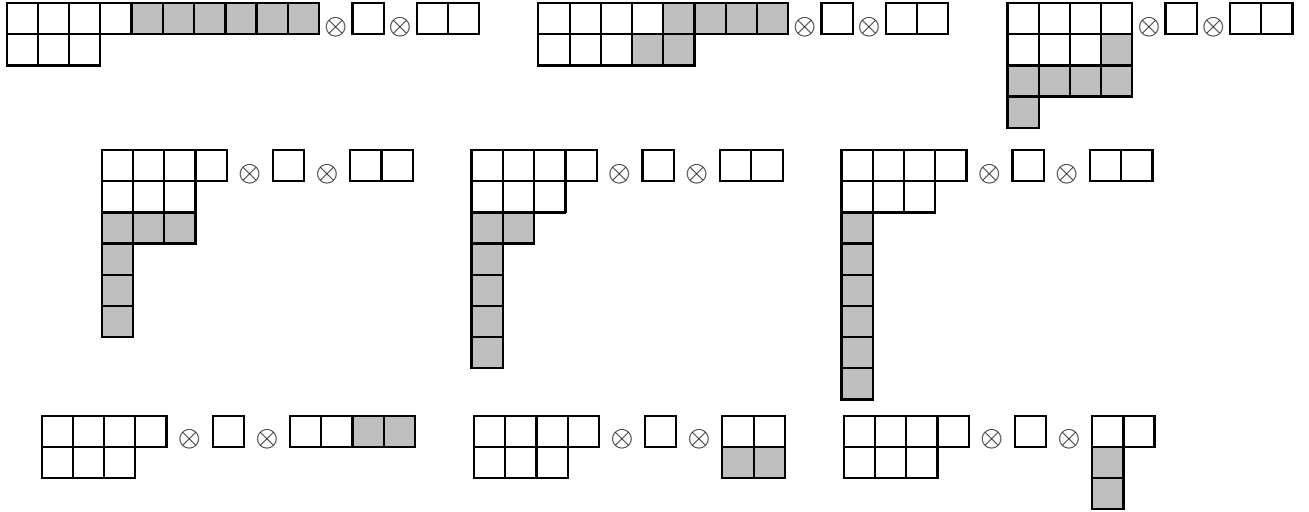
\ytableausetup{aligntableaux = top}
\[ 
\y{4,3}*[*(lightgray)]{4+6} \,\otimes \,
\y{1} \,\otimes \,
\y{2}\qquad
\y{4,3}*[*(lightgray)]{4+4,3+2} \,\otimes \,
\y{1}\, \otimes\,
\y{2}\qquad
\y{4,3}*[*(lightgray)]{0,3+1,4,1} \,\otimes \,
\y{1} \,\otimes \,
\y{2} \]
\[
\y{4,3}*[*(lightgray)]{0,0,3,1,1,1} \,\otimes \,
\y{1} \,\otimes \,
\y{2} \qquad
\y{4,3}*[*(lightgray)]{0,0,2,1,1,1,1} \,\otimes \,
\y{1} \,\otimes \,
\y{2}\qquad
\y{4,3}*[*(lightgray)]{0,0,1,1,1,1,1,1} \,\otimes \,
\y{1} \,\otimes \,
\y{2} 
\]
\[\y{4,3}\, \otimes \, \y{1}\, \otimes \, \y{2}*[*(lightgray)]{2+2} \qquad
\y{4,3}\, \otimes \, \y{1}\, \otimes \, \y{2}*[*(lightgray)]{0,2}\qquad
\y{4,3}\, \otimes \, \y{1}\, \otimes \, \y{2}*[*(lightgray)]{0,1,1}\qquad
\]
The answer is $s_{(3^22^11^{10,3})}^{\otimes}
-s_{(3^22^11^{8,5})}^{\otimes}
+s_{(3^22^11^{4,4,4,1})}^{\otimes}
-s_{(3^22^11^{4,3,3,1,1,1})}^{\otimes}
+s_{(3^22^11^{4,3,2,1,1,1,1})}^{\otimes}
-s_{(3^22^11^{4,3,1,1,1,1,1,1})}^{\otimes}
+3s_{(3^42^11^{4,3})}^{\otimes}
+3s_{(3^{2,2}2^11^{4,3})}^{\otimes}
-3s_{(3^{2,1,1}2^11^{4,3})}^{\otimes}$.
In contrast, when computing $s_{\sigma}^{\otimes}P_{2^3}$,
we keep the first six diagrams but replace the last three diagrams by these:
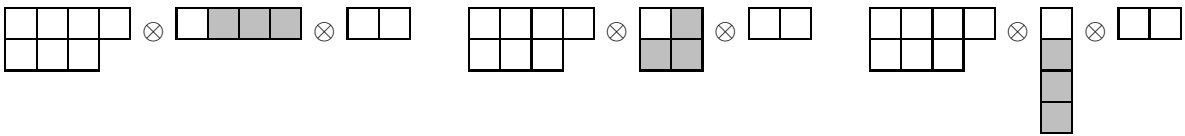
\[
  \y{4,3}\, \otimes \, \y{1}*[*(lightgray)]{1+3}\, \otimes\, \y{2}\qquad
    \y{4,3}\, \otimes \, \y{1}*[*(lightgray)]{1+1,2}\, \otimes\, \y{2}\qquad
    \y{4,3}\, \otimes \, \y{1}*[*(lightgray)]{0,1,1,1}\, \otimes\, \y{2}
\]
The new terms are $+2s_{(3^22^41^{4,3})}^{\otimes}
-2s_{(3^22^{2,2}1^{4,3})}^{\otimes}
+2s_{(3^22^{1,1,1,1}1^{4,3})}^{\otimes}$.
\end{example}

\subsection{Rule for $s_{\sigma}^{\otimes}P_{\delta}$ 
and $\mcM(P,s^{\otimes})$.}
\label{subsec:s-times-P2}

Let $\alpha=(\alpha_1,\ldots,\alpha_s)$ be a list of positive integers.
Iteration of~\cref{prop:slinky} leads to the classical
Schur expansion of $s_{\mu}p_{\alpha}$ in terms of rim hook tableaux,
which we now describe. A \emph{rim hook tableau (RHT)} of \emph{shape}
$\lambda/\mu$ and \emph{content} $\alpha$ is a sequence of partitions
$\mu=\nu^0,\nu^1,\nu^2,\ldots,\nu^s=\lambda$ such that $\nu^i/\nu^{i-1}$
is an $\alpha_i$-ribbon for $1\leq i\leq s$. We visualize this skew RHT
by drawing the skew shape $\lambda/\mu$ and filling the boxes in the ribbon
$\nu^i/\nu^{i-1}$ with the value $i$. The sign of the RHT is the
product of the signs of all the ribbons appearing in it.
The coefficient of $s_{\lambda}$ in $s_{\mu}p_{\alpha}$ is the signed sum 
of all RHT of shape $\lambda/\mu$ and content $\alpha$.
For example, here is one RHT that contributes $+1$ to the coefficient of
$s_{(4,4,4,4,1)}$ in $s_{(3,2)}p_{(4,2,3,3)}$.
\[ \tableau{ {}&{}&{}&2
          \\ {}&{}&1 &2
          \\ 1 & 1&1& 4
          \\ 3 & 3&4&4
          \\ 3 } \]

We get an analogous result for polysymmetric functions by iterating
\cref{thm:s-times-P1}. Let $\delta=(d_1^{m_1},\ldots,d_s^{m_s})$
be an ordered sequence of blocks.
A \emph{tensor rim hook tableau (TRHT)} of \emph{shape} $\tau/\sigma$
and \emph{content} $\delta$ is a sequence of types
$\sigma=\tau^0,\tau^1,\tau^2,\ldots,\tau^s=\tau$ such that,
for $1\leq i\leq s$, $\tau^i$ arises from $\tau^{i-1}$ by adding a 
$d_im_i/k_i$-ribbon $R_i$ to $\dg(\tau^{i-1}|_{k_i})$ for some $k_i$ 
dividing $d_i$. Let $\TRHT(\tau/\sigma,\delta)$ be the set of such objects.
Write $\TRHT(\tau,\delta)$ when $\sigma$ is the empty type.
The sign (resp. weight) of a TRHT is the product of the signs (resp. weights)
of all ribbons appearing in it. If the TRHT has $r_k$ ribbons in the shape
in tensor position $k$ for each $k$, then the weight of the TRHT is
$\prod_{k\geq 1} k^{r_k}$. As with RHT, we visualize a TRHT by filling
all cells in ribbon $R_i$ with the value $i$. This discussion proves
the following theorem.

\begin{theorem}\label{thm:s-times-P2}
For any type $\sigma$ and sequence $\delta=(d_1^{m_1},\ldots,d_s^{m_s})$,
\[ s_{\sigma}^{\otimes}P_{\delta}=\sum_{\tau} 
 \left[\sum_{T\in\TRHT(\tau/\sigma,\delta)} \sgn(T)\wt(T)\right]
  s_{\tau}^{\otimes}. \]
\end{theorem}

\begin{example}
The TRHT shown below contributes 
$(-1)^3\cdot 2\cdot 3^2\cdot 4=-72$ 
to the coefficient of $s^{\otimes}_{(1^{2,2,2}2^{2,2,2}3^{2,2}4^{3,2})}$ 
in $s^{\otimes}_{(1^{2,1}2^{1,1}4^{2,2})}P_{(4^2,3^2,6^1,3^1,4^1)}$.
\ytableausetup{aligntableaux = top}
\[ \yt{{}{},{}4,44} \otimes \yt{{}1,{}1,11}
   \otimes \yt{22,33} \otimes \yt{{}{}5,{}{}} \]
\end{example}

Starting with $s_0^{\otimes}=1$ and multiplying by $P_{\sigma}$,
we obtain the following transition matrix. 

\begin{corollary}\label{cor:M(P,s-tensor)}
For all types $\sigma,\tau\vdt n$, the coefficient of $s_{\tau}^{\otimes}$
in the $s^{\otimes}$-expansion of $P_{\sigma}$ is
$$\mcM(P,s^{\otimes})_{\tau,\sigma}=\sum_{T\in\TRHT(\tau,\sigma)}
 \sgn(T)\wt(T).$$
\end{corollary}

\ytableausetup{aligntableaux = top}
\begin{example} 
We compute the $s^\otimes$-expansion of $P_{(2^1, 1^2)}$. 
Creating the tensor rim hook tableaux according to the rules above, 
we get the following eight objects. \boks{0.4}
\[ \begin{array}{cccc}
             \yt{11,2,2} \otimes \vn \quad &
             \yt{1122} \otimes \vn \quad &
             \yt{12,12} \otimes \vn \quad &
             \yt{1,1,2,2} \otimes \vn \\[1.2cm]
              \yt{11,22} \otimes \vn &
              \yt{122,1} \otimes \vn &
              \yt{22} \otimes \yt{1} &
              \yt{2,2} \otimes \yt{1}
\end{array} \]
These give us the expansion 
\[
P_{(2^11^2)} = -s^\otimes_{(1^{2,1,1})} 
               +s^\otimes_{(1^4)} 
             + 2s^\otimes_{(1^{2,2})} 
              + s^\otimes_{(1^{1,1,1,1})} 
              - s^\otimes_{(1^{3,1})} 
             + 2s^\otimes_{(1^22^1)} 
              -2s^\otimes_{(1^{1,1}2^1)}. 
\]
\end{example}

\subsection{Rule for $s^{\otimes}_{\sigma}H_{d^r}$.}
\label{subsec:s-times-H1}
In order to understand the effect of multiplying $s^\otimes_\sigma$ 
by $H_{d^r}$, we express $H_{d^r}$ in terms of $h^\otimes$ 
and then use the plethystic Murnaghan--Nakayama Rule. 
We recall that \textit{plethysm} is a binary operation,
mapping an ordered pair $(f,g)$ of symmetric functions to an output 
$f[g]\in\Lambda$, which satisfies the \textit{Monomial Substitution Rule}: 
for any power-sum $p_n$ and $f=f(\xx)\in\Lambda$, $f[p_n]=f(\xx^n)$.
Plethysm appears in our discussion of $\PLambda$
since $H_{d^r}(\xx_{**}) = H_d(\xx_{**}^r)$. 
We shall only need the Monomial Substitution Rule here,
but readers interested in knowing more about plethysm 
may refer to~\cite{pleth-expose}. Note that $f[p_n]=p_n[f]$ for all $n$.

\begin{prop}\label{prop:H-expansions}
For nonnegative integers $d$ and $r$, the following expansions hold.
\begin{enumerate}[(a)]
\item $H_d = \sum\limits_{\lambda\vdash d} 
h_{m_1(\lambda)}(\xx_{1*})h_{m_2(\lambda)}(\xx_{2*})\cdots 
h_{m_k(\lambda)}(\xx_{k*})\cdots
= \sum\limits_{\lambda\vdash d} 
h_{m_1(\lambda)} \otimes h_{m_2(\lambda)} \otimes \ldots
 \otimes h_{m_k(\lambda)}\otimes\cdots$.
\item \begin{align*} H_{d^r} &= \sum\limits_{\lambda\vdash d} 
h_{m_1(\lambda)}(\xx^r_{1*})h_{m_2(\lambda)}(\xx^r_{2*})\cdots
h_{m_k(\lambda)}(\xx^r_{k*})\cdots
 \\ &= \sum\limits_{\lambda\vdash d} 
h_{m_1(\lambda)}[p_r] \otimes h_{m_2(\lambda)}[p_r] \otimes \cdots
\otimes h_{m_k(\lambda)}[p_r]\otimes\cdots. \end{align*}
\end{enumerate}
\end{prop}
\begin{proof}
To prove part~(a),
consider the summand on the right side indexed by the partition 
$\lambda = (1^{m_1(\lambda)}2^{m_2(\lambda)}\ldots k^{m_k(\lambda)}\ldots)$. 
We know that for each $i$, any monomial that appears in 
$h_{m_i(\lambda)}(\xx_{i*})$ is a product of $m_i(\lambda)$ variables 
chosen (with repetition allowed) from the variable set $\xx_{i*}$. Thus, 
any monomial in $h_{m_1(\lambda)}(\xx_{1*})h_{m_2(\lambda)}(\xx_{2*})\ldots$
is a product of $m_k(\lambda)$ variables from $\xx_{k*}$ (for each
$k$) and has degree $\sum\limits_{k\geq 1} km_k(\lambda)=|\lambda|=d$. 
This shows that each term in the sum on the right side of~(a) appears 
in the expansion of $H_d$. 
To show that these are the only possible terms, we observe that any monomial 
$f$ of degree $d$ in variables $\{x_{ij}\}_{i,j\geq 1}$ can be expressed as 
a product $f_1(\xx_{1*})f_2(\xx_{2*})\cdots$ where each $f_k$ is a monomial 
in the variables $\xx_{k*}$ of degree $d_k$. 
Define $\lambda = (1^{d_1}2^{d_2}\cdots k^{d_k}\cdots)$. 
Then $f$ appears as a monomial in the product 
$h_{m_1(\lambda)}(\xx_{1*})h_{m_2(\lambda)}(\xx_{2*})\cdots$ 
in the summand indexed by $\lambda$ on the right side of~(a).

Part~(b) follows from the definition of $H_{d^r}$, part~(a), 
and the Monomial Substitution Rule for plethysm.
\end{proof}

\begin{example}
The partitions of $4$ are $(1^4)$, $(1^22^1)$, $(1^13^1)$, $(2^2)$, and $(4^1)$.
So $H_4=h_4\otimes 1\otimes 1\otimes 1
       +h_2\otimes h_1\otimes 1\otimes 1
       +h_1\otimes 1\otimes h_1\otimes 1
       +1\otimes h_2\otimes 1\otimes 1
       +1\otimes 1\otimes 1\otimes h_1$.
\end{example}

To compute $s^\otimes_\sigma H_{d^r}$, we need to understand the 
combinatorial objects that appear in the Schur expansion of 
$s_\mu \cdot h_n[p_r]$. The formula appears in~\cite[pg. 29]{DLT} 
and a combinatorial interpretation in terms 
of $r$-decomposable partitions was given by Wildon in~\cite{wildon1}.
We give a formula based on the notion of $r^n$-polyribbons
following the description in Turek~\cite{turek}. The notation $r^n$ does not
signify exponentiation but is meant to evoke the $n$-fold iteration
of the operation of adding an $r$-ribbon.

Here is the formal definition. Let $\gamma/\rho$ be a $k$-ribbon.
The \emph{top row} of $\gamma/\rho$, denoted by $\top(\gamma/\rho)$,
is the least row containing a cell of $\gamma/\rho$. 
A skew shape $\lambda/\mu$ is called an \emph{$r^n$-polyribbon}
if there exist partitions $\gamma_{(0)}$, $\gamma_{(1)}$, $\ldots$,
$\gamma_{(n)}$ such that:
\begin{equation}\label{eq:polyribbon-decomp}
\mu = \gamma_{(0)} \subseteq \gamma_{(1)} \subseteq 
\cdots \subseteq \gamma_{(n)} = \lambda,
\end{equation}
$\gamma_{(i)}/\gamma_{(i-1)}$ is an $r$-ribbon for $1\leq i \leq n$, 
and $\top(\gamma_{(i)}/\gamma_{(i-1)}) \geq \top(\gamma_{(i+1)}/\gamma_{(i)})$ 
for $1\leq i \leq n-1$. If $\lambda/\mu$ is an $r^n$-polyribbon,
then (as is readily checked) only one list $\gamma_{(0)},\ldots,\gamma_{(n)}$
satisfies the conditions stated here. Thus, we may define the \emph{sign}
of this $r^n$-polyribbon, written $\sgn_r(\lambda/\mu)$, to be
$\prod_{i=1}^n \sgn(\gamma_{(i)}/\gamma_{(i-1)})$.
If $\lambda/\mu$ is not an $r^n$-polyribbon for any $n$, 
then we set $\sgn_r(\lambda/\mu) = 0$.

\begin{remark}
The condition on top rows is equivalent to saying that the 
northeasternmost box of each inserted ribbon lies weakly north 
and strictly east
of the northeasternmost box of the previously inserted ribbon.
\end{remark}

\begin{example}
For $\mu = (5,5,1)$ and $\lambda = (7,6,6,4)$, $\lambda/\mu$ is a skew
shape denoted by the gray cells in the figure below.
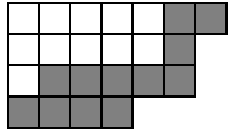
\[
\y{5,5,1}*[*(gray)]{5+2,5+1,1 + 5,4}
\]
The skew shape $\lambda/\mu$
is a $4^3$-polyribbon as it can be constructed by adding three 4-ribbons 
according to the aforementioned rules as shown here:
   
    \ytableausetup{aligntableaux = top}
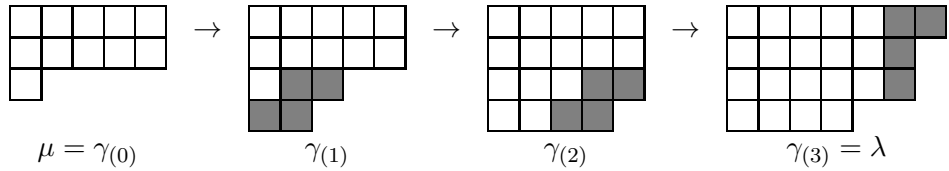
    \[\begin{array}{ccccccc}
        \y{5,5,1} & \to &\y{5,5,1}*[*(gray)]{0,0,1+2,2}&\to& \y{5,5,3,2}*[*(gray)]{0,0,3+2,2+2}& \to & \y{5,5,5,4}*[*(gray)]{5+2,5+1,5+1}\\
         \mu = \gamma_{(0)} & &\gamma_{(1)}& &\gamma_{(2)}& &\gamma_{(3)} = \lambda
    \end{array}\]
If we write $t_i$ for $\top\left(\gamma_{(i)}/\gamma_{(i-1)}\right)$, 
then $t_1 = 3$, $t_2 = 3$, and $t_3 = 1$. This polyribbon has sign
$\sgn_4(\lambda/\mu)=(-1)\cdot (-1)\cdot 1=1$.
\end{example}

\begin{remark}
The next examples illustrate some common pitfalls that may occur. 

(a) The shape $(1,1,1,1,1,1)$ is \emph{not} a $3^2$-polyribbon 
as the only way to construct it is as follows: \boks{0.3}
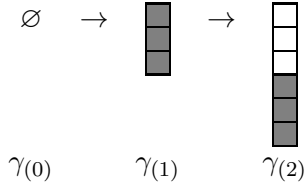
\[\begin{array}{ccccc}
        \varnothing &\to &\y{}*[*(gray)]{1,1,1}&\to& \y{1,1,1}*[*(gray)]{0,0,0,1,1,1}\\
         \gamma_{(0)} & &\gamma_{(1)}& &\gamma_{(2)}
    \end{array}\]
Here $\top(\gamma_{(1)}/\gamma_{(0)}) = 1$,
 which is smaller than $\top(\gamma_{(2)}/\gamma_{(1)}) = 4$.

(b) The list of component ribbons of an $r^n$-polyribbon is unique 
when nonnegative integers $r$ and $n$ are fixed. 
For instance, $(3,3)$ is a $2^3$-polyribbon
constructed via $\varnothing \to (1,1) \to (2,2) \to (3,3)$. 
On the other hand, $(3,3)$ is a $3^2$-polyribbon constructed 
via $\varnothing \to (2,1) \to (3,3)$; note that the alternate
construction $\varnothing\to (3)\to(3,3)$ is invalid.

(c) An $r^n$-polyribbon may not be connected, in the sense that 
the skew shape might be the union of two subsets of boxes with no shared edges.
For instance, $(6,1,1,1,1)/(3,1)$ is a disconnected $3^2$-polyribbon,
as one can see from this diagram: $\y{3,1}*[*(gray)]{3+3,0,1,1,1}$.

(d) We use the phrase ``adding an $r^n$-polyribbon to $\mu$ to 
give $\lambda$'' to mean $\lambda/\mu$ is an $r^n$-polyribbon.
If $\mu$ is given, we create a new $r^n$-polyribbon 
$\lambda/\mu$ by adding $n$ $r$-ribbons moving northeast along the border
of the growing shape.  If instead $\lambda/\mu$ is given at the outset,
we can test whether this shape is an $r^n$-polyribbon by trying to 
delete $n$ $r$-ribbons moving southwest along the border as the
shape $\lambda$ shrinks to $\mu$ through intermediate partition shapes.
For example, this test shows that $(2,2,2)$ is a $3^2$-polyribbon
but not a $2^3$-polyribbon.
\end{remark}

Here is the promised combinatorial description 
of the Schur expansion of $s_{\mu}\cdot h_n[p_r]$.

\begin{theorem}[\cite{wildon1}, Equation~(2)]\label{thm:r-decomp}
Let $\mu$ be a partition and $r,n$ be nonnegative integers. Then
\[s_\mu\cdot h_n[p_r] = s_\mu\cdot p_r[h_n] = \sum\limits_{\lambda} \sgn_r(\lambda/\mu) s_\lambda\]
where the sum is over all partitions $\lambda$ obtained 
by adding an $r^n$-polyribbon to $\mu$.
\end{theorem}

\begin{remark}
In the case $n=1$, $h_1[p_r]=p_r$, and the rule in the theorem
reduces to the Slinky Rule stated in \cref{prop:slinky}.
In the case $r=1$, $h_n[p_1]=h_n$, and the theorem reduces to the
classical Pieri rule. This says that $s_{\mu}h_n=\sum_{\nu} s_{\nu}$
where we sum over partitions $\nu$ such that $\nu/\mu$ is 
a horizontal $n$-strip, namely a collection of $n$ boxes in distinct columns.
\end{remark}

Applying \cref{thm:r-decomp} to the polysymmetric case leads
to the following theorem.
\begin{theorem}\label{thm:s-times-H1}
Let $\sigma$ be any type and $d^r$ be a block. Then 
\[ s_\sigma^\otimes H_{d^r} = 
\sum\limits_{\tau} \sgn_r^\otimes(\tau/\sigma) s^\otimes_\tau,\]
where we sum over all types $\tau$ obtained from $\sigma$ as follows:
for some partition $\lambda\vdash d$, $\tau|_k$ is obtained by 
adding an $r^{m_k(\lambda)}$-polyribbon to $\sigma|_k$ for all $k\geq 1$;
and $\sgn_r^\otimes(\tau/\sigma) 
= \prod\limits_{k\geq 1} \sgn_r((\tau|_k)/(\sigma|_k))$.
\end{theorem}
When $\tau$ is related to $\sigma$ as described in this theorem,
we say that \emph{$\tau/\sigma$ is a $d^r$-tensor polyribbon}.
\begin{proof}
By~\cref{prop:H-expansions}(b), 
\[ s_\sigma^{\otimes}H_{d^r} 
= \sum\limits_{\lambda\vdash d} s_{\sigma|_1}\cdot h_{m_1(\lambda)}[p_r] 
\otimes s_{\sigma|_2}\cdot h_{m_2(\lambda)}[p_r] \otimes \cdots
\otimes s_{\sigma|_k}\cdot h_{m_k(\lambda)}[p_r] \otimes \cdots.  \]
The $k$th factor in the tensor product expands into
$\sum\limits_{\nu_{(k)}} \sgn_r(\nu_{(k)}/(\sigma|_k)) s_{\nu_{(k)}}$ 
where the sum is over all partitions $\nu_{(k)}$ obtained by adding 
an $r^{m_k(\lambda)}$-polyribbon to $\sigma|_k$. 
Using the distributive property of tensor products over addition 
gives the signed sum of $s_{\tau}^{\otimes}$ for the types
$\tau$ described in the theorem.
\end{proof}

\begin{example}
Let $\sigma = (3^23^22^32^11^21^1)=(1^{2,1}2^{3,1}3^{2,2})$, 
which has the tensor diagram shown here:
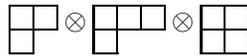
\[ \y{2,1}\otimes \y{3,1}\otimes \y{2,2} \] 
We describe one object in the expansion $s^\otimes_\sigma H_{14^3}$. 
First, we pick the partition $\lambda = (3,3,2,2,2,1,1)=(1^22^33^2)$ of 14. 
The theorem tells us to add a $3^2$-polyribbon to the first diagram,
a $3^3$-polyribbon to the second diagram, and 
a $3^2$-polyribbon to the third diagram in all possible ways.
One possible object is
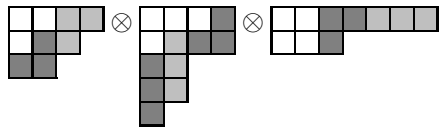
\[
    \y{2,1}*[*(gray)]{0,1+1,2}*[*(lightgray)]{2+2,2+1}\otimes \y{3,1}*[*(gray)]{0,0,1,1,1}*[*(lightgray)]{0,1+1,1+1,1+1}*[*(gray)]{3+1,2+2}\otimes \y{2,2}*[*(gray)]{2+2,2+1}*[*(lightgray)]{4+3}
    \]
Here the gray cells show the added polyribbons, and the shading shows
the constituent ribbons within each polyribbon.  The sign of this object is 
$(-1\cdot -1)\cdot (1 \cdot 1\cdot -1)\cdot (-1\cdot 1) = 1$,
and the corresponding term is $+s^\otimes_{(3^73^32^42^42^22^22^11^41^31^2)}$.
\end{example}

\subsection{Rule for $s_{\sigma}^{\otimes}H_{\delta}$ 
and $\mcM(H,s^{\otimes})$.}
\label{subsec:s-times-H2}

We can iterate~\cref{thm:s-times-H1}
to obtain the $s^{\otimes}$-expansions 
of $s^{\otimes}_{\sigma}H_{\delta}$ and $H_{\sigma}$. 
Let $\tau$ and $\sigma$ be types. 
Let $\delta=(d_1^{r_1},\ldots,d_s^{r_s})$ be an ordered sequence of blocks.
A \emph{tensor polyribbon tableau (TPRT)} $T$ of \emph{shape} $\tau/\sigma$
and \emph{content} $\delta$ is a sequence of types
$\sigma=\tau_{(0)},\tau_{(1)},\ldots,\tau_{(s)}=\tau$ such that,
for all $i$ between $1$ and $s$, $\tau_{(i)}/\tau_{(i-1)}$
is a $d_i^{r_i}$-tensor polyribbon.  Let $\TPRT(\tau/\sigma,\delta)$
be the set of such objects.  We visualize $T$ by drawing
the tensor diagram of $\tau$ and filling all cells in
$\dg(\tau_{(i)})\setminus\dg(\tau_{(i-1)})$ with the value $i$.
The \emph{sign} of $T$ is $\sgn(T)
=\prod_{i=1}^s \sgn_{r_i}^{\otimes}(\tau_{(i)}/\tau_{(i-1)})$.

\begin{theorem}\label{thm:s-times-H2}
Given a type $\sigma$ and a sequence of blocks
 $\delta = (d_1^{r_1}, \ldots, d_{s}^{r_s})$,
\[ s_\sigma^\otimes H_{\delta} = \sum_{\tau} 
 \left[\sum_{T\in\TPRT(\tau/\sigma,\delta)} \sgn(T)\right] s_{\tau}^{\otimes}.  \]
\end{theorem}
\begin{proof}
This follows by iterating \cref{thm:s-times-H1}
in the same way that \cref{thm:s-times-P2} is deduced
from \cref{thm:s-times-P1}.
\end{proof}

\begin{corollary}\label{cor:M(H,s-tensor)}
For all $\sigma,\tau\vdt n$, the coefficient of $s_{\tau}^{\otimes}$
in the $s^{\otimes}$-expansion of $H_{\sigma}$ is
\[ \mcM(H,s^{\otimes})_{\tau,\sigma}=
  \sum_{T\in\TPRT(\tau,\sigma)} \sgn(T). \]
\end{corollary}

\begin{example} 
We find the coefficient of $s^\otimes_{2^21^51^3}$ in the 
$s^\otimes$-expansion of $H_{3^23^2}$. 
Here, $d_1 = d_2 = 3$ and $r_1 = r_2 = 2$.
We first pick $\lambda\vdash 3$ and add a $2^{m_k(\lambda)}$-polyribbon
to an empty diagram in each position $k$. 
Then we pick $\mu\vdash 3$ and add a $2^{m_k(\mu)}$-polyribbon
to the current diagram in each position $k$. We make such choices in
all possible ways that lead to the target tensor diagram with $\dg(5,3)$ 
in position $1$ and $\dg(2)$ in position 2. Since position 3 is empty,
we cannot choose $\lambda$ or $\mu$ to be $(3^1)$.

Choosing $\lambda=(1^3)$ and $\mu=(1^12^1)$ leads to these two TPRTs,
both with sign $-1$:
\boks{0.4}
\ytableausetup{aligntableaux = top}
\[ \yt{{*(lightgray)1}{*(gray)1}{*(lightgray)1}22,
       {*(lightgray)1}{*(gray)1}{*(lightgray)1}}\otimes\yt{22}
    \otimes\varnothing\quad \quad 
   \yt{{*(lightgray)1}{*(gray)1}{*(gray)1}{*(lightgray)1}{*(lightgray)1},
       {*(lightgray)1}22}\otimes \yt{22} \otimes\varnothing \]
Choosing $\lambda=(1^12^1)$ and $\mu=(1^3)$ leads to these two TPRTs,
both with sign $-1$:
\boks{0.4}
\ytableausetup{aligntableaux = top}
\[ \yt{11{*(gray)2}{*(lightgray)2}{*(lightgray)2},
       {*(lightgray)2}{*(lightgray)2}{*(gray)2}} 
   \otimes \yt{11} \otimes\varnothing \quad \quad
  \yt{1{*(lightgray)2}{*(gray)2}{*(lightgray)2}{*(lightgray)2},
      1{*(lightgray)2}{*(gray)2}} \otimes \yt{11}\otimes\varnothing
\]
No other choice of $\lambda,\mu$ leads to the required tensor diagram.
Thus the coefficient of $s^\otimes_{2^21^51^3}$ in $H_{3^23^2}$ is $-4$.
\end{example}

\begin{remark}
Let $\sigma = (1^{1,1,\ldots,1})\vdt n$. 
The coefficient of $s_\tau^\otimes$ in the $s^{\otimes}$-expansion 
of $H_\sigma$ is
\[
      \mcM(H, s^\otimes)_{\tau, \sigma} = \begin{cases}
          f^\lambda & \text{if } \tau =(1^{\lambda}), \\
          0 & \text{otherwise,}
      \end{cases}
\]
where $f^\lambda$ is the number of standard Young tableaux of shape $\lambda$. 
This extends the analogous result for the symmetric function transition
matrix $\mcM(h,s)_{\lambda,1^n}$.
\end{remark}

\subsection{Rules for $s^{\otimes}_{\sigma}E^+_{d^m}$
                  and $s^{\otimes}_{\sigma}E_{d^m}$.}
\label{subsec:s-times-E1}

The rules for $E^+$ and $E$ follow from the rule for $H$. In this section, 
we make use of the involution $\omega$ on the algebra of symmetric 
functions. Under this map, $\omega(h_\lambda) = e_\lambda$,
$\omega(s_{\mu})=s_{\mu'}$, and 
$\omega(p_\lambda) = (-1)^{|\lambda|-\ell(\lambda)}p_\lambda$.
For more information about this involution, 
refer to Section~9.20 of \cite{loehr-comb}. 
In this section, we use the following result.
\begin{prop}[\cite{macd}, I.8, Ex 1(c)]\label{prop:omega-plethysm}
    Given nonnegative integers $r$ and $n$, 
    \[
    \omega(h_n[p_r]) = (-1)^{n(r-1)}e_n[p_r].
    \]
\end{prop}

Using the proof technique from~\cref{prop:H-expansions} and the
idea of square-free monomials,
we can find the $e^{\otimes}$-expansions of $E_d^+$ and $E_d$.

\begin{prop}\label{prop:E-expansions}
For nonnegative integers $d$ and $r$, 
the following expansions of $E^+$ and $E$ hold.
\begin{enumerate}[(a)]
\item $E^+_d = \sum\limits_{\lambda\vdash d} 
e_{m_1(\lambda)}(\xx_{1*})e_{m_2(\lambda)}(\xx_{2*})\cdots 
e_{m_k(\lambda)}(\xx_{k*})\cdots
= \sum\limits_{\lambda\vdash d} e_{m_1(\lambda)} \otimes e_{m_2(\lambda)} 
\otimes \cdots \otimes e_{m_k(\lambda)}\otimes\cdots$.

\item $E_d = \sum\limits_{\lambda\vdash d} 
(-1)^{\ell(\lambda)}e_{m_1(\lambda)}(\xx_{1*})\cdots
  e_{m_k(\lambda)}(\xx_{k*})\cdots
= \sum\limits_{\lambda\vdash d} (-1)^{\ell(\lambda)}
e_{m_1(\lambda)}\otimes \cdots\otimes e_{m_k(\lambda)}\otimes\cdots$.

\item $E^+_{d^r} = \sum\limits_{\lambda\vdash d} 
  e_{m_1(\lambda)}(\xx_{1*}^r)\cdots e_{m_k(\lambda)}(\xx_{k*}^r)\cdots
= \sum\limits_{\lambda\vdash d} 
e_{m_1(\lambda)}[p_r]\otimes \cdots\otimes e_{m_k(\lambda)}[p_r]\otimes\cdots$.

\item $E_{d^r} = \sum\limits_{\lambda\vdash d} 
(-1)^{\ell(\lambda)}e_{m_1(\lambda)}(\xx_{1*}^r)\cdots
  e_{m_k(\lambda)}(\xx_{k*}^r)\cdots
= \sum\limits_{\lambda\vdash d} (-1)^{\ell(\lambda)}
e_{m_1(\lambda)}[p_r]\otimes \cdots\otimes e_{m_k(\lambda)}[p_r]\otimes\cdots$.

\end{enumerate}
\end{prop}
\begin{proof}
We prove (a) and (b), and the rest follows from the Monomial Substitution 
Rule. We proceed as in the proof of~\cref{prop:H-expansions}, but in this case 
each variable appears at most once. This gives us the expression for $E_d^+$.
For $E_d$, the sign of a monomial $f$ is given by $(-1)^{\text{len}(f)}$, 
where $\text{len}(f)$ is the number of indeterminates in $f$. 
Each $e_{k}$ has exactly $k$ indeterminates and thus has the sign $(-1)^{k}$. 
This shows that the sign for the monomial 
$e_{m_1(\lambda)}(\xx_{1*})\cdots e_{m_k(\lambda)}(\xx_{k*})\cdots$ is
$(-1)^{m_1(\lambda)+\cdots+m_k(\lambda)+\cdots}= (-1)^{\ell(\lambda)}$.
\end{proof}

Before we present the analogue of~\cref{thm:r-decomp}
for multiplying a Schur function by $e_n[p_r]$,
we introduce a notion dual to that of an $r^n$-polyribbon. 
For any skew shape $\lambda/\mu$,
let $\bot(\lambda/\mu)$ denote the least index of a column 
that contains a cell of $\lambda/\mu$. 
A skew shape $\lambda/\mu$ is called an \textit{$(r^n)'$-polyribbon} or 
a \textit{dual $r^n$-polyribbon} if there 
exists a (necessarily unique) list of partitions 
$\gamma_{(0)},\gamma_{(1)},\ldots,\gamma_{(n)}$ such that
$\mu = \gamma_{(0)} \subseteq \gamma_{(1)} \subseteq \cdots 
\subseteq \gamma_{(n)} = \lambda$,
$\gamma_{(i)}/\gamma_{(i-1)}$ is an $r$-ribbon for $1\leq i \leq n$, 
and $\bot(\gamma_{(i)}/\gamma_{(i-1)}) \geq 
\bot(\gamma_{(i+1)}/\gamma_{(i)})$ for $1\leq i \leq n-1$. Define the 
\emph{sign} of an $(r^n)'$-polyribbon to be $\sgn'_r(\lambda/\mu) 
= \prod\limits_{i=1}^n \sgn(\gamma_{(i)}/\gamma_{(i-1)})$.

\begin{remark}
Equivalently, $\lambda/\mu$ is a dual $r^n$-polyribbon 
if we can go from $\dg(\mu)$
to $\dg(\lambda)$ by adding $n$ $r$-ribbons in succession,
where the southwesternmost box of each new
$r$-ribbon lies strictly south and weakly west of the southwesternmost box 
of the previously added $r$-ribbon.
\end{remark}

\begin{remark}\label{rem:dual-polyribbon}
If $\lambda/\mu$ is an $r^n$-polyribbon, 
then $\lambda'/\mu'$ is an $(r^n)'$-polyribbon, and conversely.
\end{remark}

\begin{example}
For $\mu = (3,1)$ and $\lambda = (4,4,2,2,2,2)$, $\lambda/\mu$ is the following skew shape:
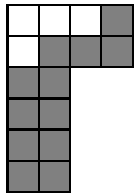
\[
\y{3,1}*[*(gray)]{3+1,1+3,2,2,2,2}
\]
The skew shape $\lambda/\mu$ is a dual $4^3$-polyribbon 
since it can be constructed as follows:
 \ytableausetup{aligntableaux = top}
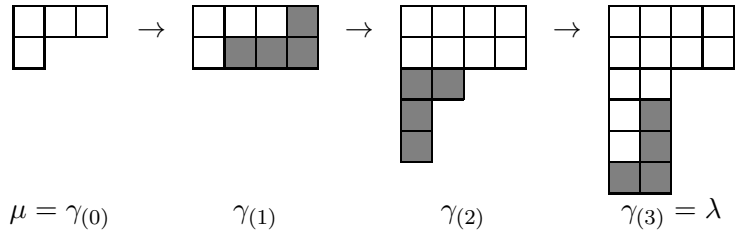
    \[\begin{array}{ccccccc}
        \y{3,1} & \to &\y{3,1}*[*(gray)]{3+1,1+3}&\to& \y{4,4}*[*(gray)]{0,0,2,1,1}& \to & \y{4,4,2,1,1}*[*(gray)]{0,0,0,1+1,1+1,2}\\
         \mu = \gamma_{(0)} & &\gamma_{(1)}& &\gamma_{(2)}& &\gamma_{(3)} = \lambda
    \end{array}\]
The values of $\bot\left(\gamma_{(i)}/\gamma_{(i-1)}\right)$ for $i=1,2,3$
are $2$, $1$, and $1$.  This polyribbon has sign
$\sgn'_4(\lambda/\mu)=(-1)\cdot 1 \cdot 1=-1$.
\end{example}

\begin{prop}\label{prop:ep-plethysm}
    Given a partition $\mu$ and nonnegative integers $n$ and $r$, 
    \[
  s_\mu \cdot e_n[p_r] = \sum\limits_{\lambda} \sgn'_r(\lambda/\mu) s_\lambda,
    \]
   where the sum is over all partitions
$\lambda$ obtained by adding a dual $r^n$-polyribbon to $\mu$.
\end{prop}
\begin{proof}
Recall from~\cref{thm:r-decomp} that
\[s_\mu\cdot h_n[p_r] = \sum\limits_{\nu} \sgn_r(\nu/\mu) s_\nu, \]
where the sum is over all partitions $\nu$ obtained by adding an
$r^n$-polyribbon to $\mu$. Acting on both sides by
$\omega$ and then using \cref{prop:omega-plethysm} gives 
\[
s_{\mu'}\cdot (-1)^{n(r-1)}e_n[p_r] 
= \sum\limits_{\nu} \sgn_{r}(\nu/\mu)s_{\nu'}.
\]
Replacing $\mu'$ by $\mu$ gives 
    \[
s_{\mu}\cdot e_n[p_r] = (-1)^{n(r-1)}\sum\limits_{\nu} 
\sgn_{r}(\nu/\mu')s_{\nu'},
\]
where the sum is over partitions $\nu$ obtained by adding an $r^n$-polyribbon
 to $\mu'$, the conjugate partition of $\mu$. 
Equivalently, by \cref{rem:dual-polyribbon},
$\nu'$ is obtained by adding the dual $r^n$-polyribbon $\nu'/\mu$ to $\mu$.
Defining $\lambda = \nu'$, it suffices to
show $\sgn_r'(\lambda/\mu)= (-1)^{n(r-1)}\sgn_r(\nu/\mu')$.  
If a skew shape $\alpha/\beta$ is a $r$-ribbon covering $\ell$ rows, 
then its sign is $(-1)^{\ell-1}$. 
The number of columns spanned by this ribbon is 
$r+1-\ell$ which determines the sign of $\alpha'/\beta'$, 
i.e., $\sgn(\alpha'/\beta') = (-1)^{r-\ell}$.
Let the $r^n$-polyribbon $\nu/\mu'$ have the decomposition 
$\gamma_{(0)},\gamma_{(1)},\ldots,\gamma_{(n)}$
as in~\cref{eq:polyribbon-decomp}, where each $\gamma_{(i)}/\gamma_{(i-1)}$ 
covers $\ell_i$ rows and $r+1-\ell_i$ columns. This gives us 
\begin{align*}
 \sgn_r'(\lambda/\mu) &= (-1)^{(r-\ell_1) + (r-\ell_2) + \ldots +(r-\ell_n)}
\\ &=(-1)^{nr}(-1)^{(\ell_1-1)+(\ell_2-1)+\cdots+(\ell_n-1)+n}
\\ &=(-1)^{n(r-1)}\sgn_r(\nu/\mu').  \qedhere
\end{align*}
\end{proof}

For types $\tau$ and $\sigma$,
we say that $\tau/\sigma$ is a \textit{dual $d^r$-tensor polyribbon} 
if, for some partition $\lambda$ of $d$, 
each $\tau|_k$ is obtained from $\sigma|_k$ by 
adding a dual $r^{m_k(\lambda)}$-polyribbon.
We call the partition $\lambda$ 
the \textit{associated partition of $\tau/\sigma$}.
In this situation,
define $\sgn_r^+(\tau/\sigma) 
= \prod\limits_{k=1}^\infty \sgn'_r((\tau|_k)/(\sigma|_k))$ 
and $\sgn^-_r(\tau/\sigma) 
= (-1)^{\ell(\lambda)}\prod\limits_{k=1}^\infty 
 \sgn'_r((\tau|_k)/(\sigma|_k))$,
where $\lambda$ is the associated partition of $\tau/\sigma$. 
The extra power $(-1)^{\ell(\lambda)}$ is the total number of
$r$-ribbons (within the various polyribbons) that are added
to $\sigma$ to reach $\tau$.

\begin{theorem}\label{thm:s-times-E1}
Let $\sigma$ be any type and $d^r$ be a block. Then 
\[s_\sigma^\otimes E^+_{d^r} = \sum\limits_{\tau} 
\sgn_r^+(\tau/\sigma) s_\tau^\otimes\quad\mbox{ and }\quad
 s_\sigma^\otimes E_{d^r} = \sum\limits_{\tau} 
\sgn^-_r(\tau/\sigma) s_\tau^\otimes,\]
where the sums range over types $\tau$ such that $\tau/\sigma$ 
is a dual $d^r$-tensor polyribbon.
\end{theorem}
\begin{proof}
We prove it for the case of $E_{d^r}^+$, and the same 
proof works for $E_{d^r}$ with an appropriate change of sign. 
From~\cref{prop:E-expansions}, we obtain 
$$s_\sigma^\otimes \cdot E^+_{d^r} 
= \sum\limits_{\lambda\vdash d } s_{\sigma|_1}\cdot e_{m_1(\lambda)}[p_r]  
\otimes \cdots \otimes s_{\sigma|_k}\cdot e_{m_k(\lambda)}[p_r]\otimes \cdots.$$
Applying \cref{prop:ep-plethysm} to the above expression,
the $k$th component of the tensor product expands to
$\sum\limits_{\gamma} \sgn'_r(\gamma/(\sigma|_k))s_\gamma$, 
where the sum is over partitions $\gamma$ that arise by adding 
a dual $r^{m_k(\lambda)}$-polyribbon to $\sigma|_k$. 
Using the distributive law gives us our result. 
\end{proof}

\subsection{Rules for $s_{\sigma}^{\otimes}E^+_{\delta}$,
$s_{\sigma}^{\otimes}E_{\delta}$, $\mcM(E^+,s^{\otimes})$, 
and $\mcM(E,s^{\otimes})$}
\label{subsec:s-times-E2}

To obtain the entries of the next transition matrices, 
we define a dual version of the tableaux in \cref{subsec:s-times-H2}. 
Let $\tau$ and $\sigma$ be types. 
Let $\delta=(d_1^{r_1},\ldots,d_s^{r_s})$ be an ordered sequence of blocks.
A \emph{dual tensor polyribbon tableau} (\emph{dual TPRT}) $T$ 
of \emph{shape} $\tau/\sigma$
and \emph{content} $\delta$ is a sequence of types
$\sigma=\tau_{(0)},\tau_{(1)},\ldots,\tau_{(s)}=\tau$ such that,
for all $i$ between $1$ and $s$, $\tau_{(i)}/\tau_{(i-1)}$
is a dual $d_i^{r_i}$-tensor polyribbon.  Let $\TPRT^\prime(\tau/\sigma,\delta)$
be the set of such objects.  We visualize $T$ by drawing
the tensor diagram of $\tau$ and filling all cells in
$\dg(\tau_{(i)})\setminus\dg(\tau_{(i-1)})$ with the value $i$.
Define the two corresponding signs associated with $T$ to be 
$\sgn^+(T) =\prod_{i=1}^s \sgn_{r_i}^+(\tau_{(i)}/\tau_{(i-1)})$ and 
$\sgn^-(T) =\prod_{i=1}^s \sgn_{r_i}^-(\tau_{(i)}/\tau_{(i-1)})$.

\begin{theorem}\label{thm:s-times-H2}
Given a type $\sigma$ and a sequence of blocks
 $\delta = (d_1^{r_1}, \ldots, d_{s}^{r_s})$,
\[ s_\sigma^\otimes E^+_{\delta} = \sum_{\tau} 
\left[\sum_{T\in\TPRT'(\tau/\sigma,\delta)} \sgn^+(T)\right] s_{\tau}^{\otimes} 
\quad\mbox{and}\quad s_\sigma^\otimes E_{\delta} = \sum_{\tau} 
\left[\sum_{T\in\TPRT'(\tau/\sigma,\delta)} \sgn^-(T)\right] s_{\tau}^{\otimes}.
\]
\end{theorem}
\begin{proof}
These follow by iterating \cref{thm:s-times-E1}
in the same way that \cref{thm:s-times-P2} is deduced
from \cref{thm:s-times-P1}.
\end{proof}

\begin{example}
In this example, we construct one object of shape 
$\tau = 1^{4^32^51^2}2^{2^51^5}3^{53^221}$ that appears 
in the $s^{\otimes}$-expansions of $s_{\sigma}^\otimes E_\delta$ 
and $s_{\sigma}^\otimes E^+_\delta$ for 
$\sigma = 1^{2}3^{1,1,1}$ and $\delta = (11^5, 5^6)$.
We first pick the partition $\lambda= 1^22^33^1 \vdash 11$. 
Starting with the tensor diagram of $\tau$,
we insert a dual $5^2$-polyribbon in the first diagram, 
a dual $5^3$-polyribbon in the second diagram, and 
a dual $5^1$-polyribbon in the third diagram. 
We label the cells in these polyribbons by 1. 
Next we pick the partition $\mu=1^23^1\vdash 5$. 
We continue by adding a dual $6^2$-polyribbon to the first diagram
and a dual $6^1$-polyribbon to the third diagram, with all new cells
labeled by $2$. Here is one possible object
$T\in \TPRT(\tau/\sigma, \delta)$ arising from these insertions:
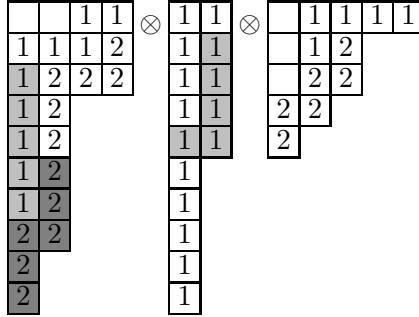
\ytableausetup{aligntableaux = top}
\[
\yt{{}{}11,1112,{*(lightgray)1}222,{*(lightgray)1}2,{*(lightgray)1}2,{*(lightgray)1}{*(gray)2},{*(lightgray)1}{*(gray)2},{*(gray)2}{*(gray)2},{*(gray)2},{*(gray)2}} \otimes \yt{11,1{*(lightgray)1},1{*(lightgray)1},1{*(lightgray)1},{*(lightgray)1}{*(lightgray)1},1,1,1,1,1} \otimes \yt{{}1111,{}12,{}22,22,2}
\]
We compute $\sgn^+(T)=1$ and $\sgn^-(T)=(-1)^{\ell(\lambda)+\ell(\mu)}\sgn^+(T)
 =(-1)^{6+3}=-1$.
\end{example}

\begin{corollary}\label{cor:M(E,s-tensor)}
For all types $\sigma,\tau\vdt n$, the coefficients of $s_{\tau}^{\otimes}$
in the $s^{\otimes}$-expansions of $E_{\sigma}^+$ and $E_{\sigma}$ are
\[ \mcM(E^+,s^{\otimes})_{\tau,\sigma}=
  \sum_{T\in\TPRT'(\tau,\sigma)} \sgn^+(T)\quad\mbox{and}\quad 
 \mcM(E,s^{\otimes})_{\tau,\sigma}=
  \sum_{T\in\TPRT'(\tau,\sigma)} \sgn^-(T). \]
\end{corollary}

\section{Expansions in the $p^{\otimes}$ Basis}
\label{sec:p-expand}

\subsection{Algebraic Development of $p^{\otimes}$-Expansions}
\label{subsec:frag-alg-p-tensor}

Given integer partitions 
 $\lambda=(1^{m_1(\lambda)}2^{m_2(\lambda)}\cdots)$ and
     $\mu=(1^{m_1(\mu)}2^{m_2(\mu)}\cdots)$,
define their \emph{union} to be
$\lambda\cup\mu=(1^{m_1(\lambda)+m_1(\mu)}2^{m_2(\lambda)+m_2(\mu)}\cdots)$,
which is the partition obtained by combining all the parts of $\lambda$
and $\mu$ (with multiplicities) into a new weakly decreasing list.
By definition of power-sums, we have $p_{\lambda}p_{\mu}=p_{\lambda\cup\mu}$.
More generally, given integer partitions $\lambda^{(1)},\ldots,\lambda^{(s)}$,
$\prod_{i=1}^s p_{\lambda^{(i)}}=p_{\lambda^{(1)}\cup\cdots\cup\lambda^{(s)}}$.

Similar results hold for types and the $p^{\otimes}$-basis of $\PLambda$.
For any types $\sigma$ and $\rho$, let $\sigma\cup\rho$ be the type
obtained by merging all the blocks in $\sigma$ and $\rho$ (with multiplicities)
into a new list of blocks. Equivalently, using the union operation on integer
partitions, we can define $\sigma\cup\rho$ by
$(\sigma\cup\rho)|_k=\sigma|_k\cup\rho|_k$ for all $k\geq 1$.
It follows from this definition that $p_{\sigma\cup\rho}^{\otimes}
 =p_{\sigma}^{\otimes}p_{\rho}^{\otimes}$. More generally,
for all types $\tau^{(1)},\ldots,\tau^{(s)}$, 
\begin{equation}\label{eq:ptens-product}
\prod_{i=1}^s p^{\otimes}_{\tau^{(i)}}
  =\prod_{i=1}^s \bigotimes_{k\geq 1} p_{\tau^{(i)}|_k}
  =\bigotimes_{k\geq 1}\ p_{\tau^{(1)}|_k\cup\cdots\cup\tau^{(s)}|_k}
  =p^{\otimes}_{\tau^{(1)}\cup\cdots\cup\tau^{(s)}}.
\end{equation}
Combining this formula with the distributive law, we get an algebraic
prescription for the $p^{\otimes}$-expansion of a product $G_1G_2\cdots G_s$ 
assuming we already know the $p^{\otimes}$-expansions of each $G_i$.
In particular, to get the transition matrices 
$\mcM(P,p^{\otimes})$, $\mcM(H,p^{\otimes})$, $\mcM(E^+,p^{\otimes})$, 
and $\mcM(E,p^{\otimes})$, it suffices to find the $p^{\otimes}$-expansions
of $P_{d^r}$, $H_{d^r}$, $E^+_{d^r}$, and $E_{d^r}$. 

Before presenting these expansions, we introduce some notation.
For each integer partition $\lambda$,
define $z_\lambda = \prod\limits_{i\geq 1} i^{m_i(\lambda)} m_i(\lambda)!$.
The factor $z_\lambda$ appears when finding $p$-expansions of certain
symmetric functions. In particular, $h_n = \sum\limits_{\lambda\vdash n} 
\dfrac{p_\lambda}{z_\lambda}$ and 
$e_n = \sum\limits_{\lambda\vdash n} (-1)^{n-\ell(\lambda)} 
\dfrac{p_\lambda}{z_\lambda}$ (see~\cite[\S9.19]{loehr-comb}).
The polysymmetric analog of $z_{\lambda}$ is defined by
$z^\otimes_\tau = \prod\limits_{k\geq 1} z_{\tau|_k}$ for a type $\tau$.

\begin{example}
 For $\tau =(3^23^22^32^22^21^41^2)$, we have associated
partitions ${\tau|_1} = (4,2) = 4^12^1$, ${\tau|_2} = (3,2,2) = 3^12^2$, 
and ${\tau|_3} = (2,2) = 2^2$.  We compute 
$z_{\tau}^\otimes = (4^1 1! 2^1 1!) \cdot (3^1 1! 2^2 2!)\cdot 
(2^2 2!) = 1536$. 
\end{example}

For a type $\tau =(d_1^{\,m_1}d_2^{\,m_2}\ldots d_s^{\,m_s})$ 
and an integer $r>0$, define the type 
$\tau^r=(d_1^{\,rm_1}d_2^{\,rm_2}\ldots d_s^{\,rm_s})$. 
Recall from~\S\ref{subsec:polysymm} that
$\sgn(\tau) =\prod_{i=1}^s (-1)^{m_i}=\prod_{k\geq 1} (-1)^{\area(\tau|_k)}$
and $\ell(\tau) = s=\sum_{k\geq 1} \ell(\tau|_k)$. 
The net exponent of $-1$ in $\sgn(\tau)$ is the number of blocks of $\tau$ 
with odd multiplicity, while the net exponent of $-1$ in 
$(-1)^{\ell(\tau)}\sgn(\tau)$
is the number of blocks of $\tau$ with even multiplicity.

\begin{prop}\label{prop:block-in-ptensor}
For positive integers $d$ and $r$, the following $p^\otimes$-expansions hold.
\\ (a) $P_{d^r} = \sum\limits_{k\mid d} k\, p^\otimes_{k^{rd/k}}$.
\\ (b) $H_{d^r} = \sum\limits_{\tau \vdt d} 
\dfrac{p^\otimes_{\tau^r}}{z^\otimes_\tau}$.
\\ (c) $E^+_{d^r} = \sum\limits_{\tau\vdt d} (-1)^{\ell(\tau)} 
\sgn(\tau)\dfrac{p^\otimes_{\tau^r}}{z^\otimes_\tau}$.
\\ (d) $E_{d^r} = \sum\limits_{\tau\vdt d} (-1)^{\ell(\tau)}
\dfrac{p^\otimes_{\tau^r}}{z^\otimes_\tau}$.
\end{prop}
\begin{proof}
Suppose we have found a required expansion when $r=1$,
say $F_d=\sum_{\tau} a_{\tau}p^{\otimes}_{\tau}$
where $F$ is $P$ or $H$ or $E^+$ or $E$ and $a_{\tau}\in\Q$.
The plethysm property $p_m[p_r] = p_{rm}$ (for positive integers $m,r$) 
extends to $p_{\lambda}[p_r] = p_{r\lambda}$ (for a partition $\lambda$
and integer $r$), where $r\lambda$ is $\lambda$ with all parts scaled by $r$.
Then the $p^{\otimes}$-expansion for general $r$ is
\begin{equation}\label{eq:Fd-to-Fdr}
F_{d^r} = \sum\limits_{\tau} a_\tau p_{\tau|_1}[p_r] 
\otimes p_{\tau|_2}[p_r] \otimes \ldots
= \sum\limits_{\tau} a_{\tau} p_{r\tau|_1} \otimes p_{r\tau|_2} \otimes \ldots = \sum\limits_{\tau} a_{\tau} p^\otimes_{\tau^r}.
\end{equation}

(a)~The sum $\sum\limits_{j\geq 1}x_{k,j}^{d/k}$ is the power-sum 
symmetric function $p_{d/k}(\xx_{k*})=p^\otimes_{k^{d/k}}$. 
Thus, \cref{eq:Pd} can be rephrased as 
$P_{d} = \sum\limits_{k\mid d} k\, p^\otimes_{k^{d/k}}$. 
Part~(a) now follows from~\eqref{eq:Fd-to-Fdr}.

(b)~By~\cref{prop:H-expansions}, 
$H_d = \sum\limits_{\lambda\vdash d} 
h_{m_1(\lambda)} \otimes h_{m_2(\lambda)} \otimes \cdots$. 
Using $h_n = \sum\limits_{\mu\vdash n} \dfrac{p_\mu}{z_\mu}$ 
on each factor gives 
\begin{equation}\label{eq:Hd-in-ptensor}
 H_d=
\sum\limits_{\lambda\vdash d} 
\sum\limits_{\mu^{(1)}\vdash m_1(\lambda)} 
\sum\limits_{\mu^{(2)}\vdash m_2(\lambda)} 
\cdots \sum\limits_{\mu^{(d)}\vdash m_d(\lambda)} 
\frac{p_{\mu^{(1)}}}{z_{\mu^{(1)}}}\otimes
\frac{p_{\mu^{(2)}}}{z_{\mu^{(2)}}}\otimes
\cdots\otimes \frac{p_{\mu^{(d)}}}{z_{\mu^{(d)}}}.
\end{equation}
The iterated sum here can be rewritten as a sum over types $\tau\vdt d$ via 
the bijection sending $(\lambda,\mu^{(1)},\mu^{(2)},\ldots,\mu^{(d)})$
to the type $\tau$ with $\tau|_k=\mu^{(k)}$ for all $k\geq 1$. We obtain
$H_d = \sum\limits_{\tau\vdt d} \dfrac{p^\otimes_\tau}{z_\tau^{\otimes}}$. 
Part~(b) now follows from~\eqref{eq:Fd-to-Fdr}.
(c) The proof for $E^+_d$ is like the proof for $H_d$, 
but with bookkeeping for signs. 
The $k$th tensor factor in~\eqref{eq:Hd-in-ptensor}
contributes the sign $(-1)^{\area(\mu^{(k)}) - \ell(\mu^{(k)})}$.
Converting to a sum over $\tau$ as described above, the $k$th sign factor
becomes $(-1)^{\area(\tau|_k)-\ell(\tau|_k)}$. Taking the product over $k\geq 1$
gives an overall sign of $\sgn(\tau)(-1)^{\ell(\tau)}$ for the coefficient
of $p^{\otimes}_{\tau}$.

(d) For $E_d$, each summand on the right side of~\eqref{eq:Hd-in-ptensor}
now has the sign $$(-1)^{\ell(\lambda)}\prod_{k\geq 1} (-1)^{m_k(\lambda)}
 \prod_{k\geq 1} (-1)^{\ell(\mu^{(k)})}.$$ But $\ell(\lambda)=\sum_{k\geq 1}
 m_k(\lambda)$, so that part of the sign disappears. We are left with
a sign of $(-1)^{\ell(\tau)}$ for the coefficient
of $p^{\otimes}_{\tau}$.
\end{proof}

\begin{example}
In this example, we illustrate~\cref{prop:block-in-ptensor}
for the types $2^3$ and $3^2$. We compute:
\[ \begin{array}{lcl}
P_{2^3} = p^\otimes_{1^6} + 2p^\otimes_{2^3}, & \qquad &
P_{3^2} = p^\otimes_{1^6} + 3p^\otimes_{3^2}, 
\\[6pt]
H_{2^3} = \dfrac{p^\otimes_{1^6}}{2} + \dfrac{p^\otimes_{1^31^3}}{2} 
+ p^\otimes_{2^3}, & \qquad &
H_{3^2} = \dfrac{p^\otimes_{1^6}}{3} + \dfrac{p^\otimes_{1^41^2}}{2} 
+ \dfrac{p^\otimes_{1^21^21^2}}{6} + p^\otimes_{2^21^2} + p^\otimes_{3^2},
\\[6pt]
E^+_{2^3} = -\dfrac{p^\otimes_{1^6}}{2} + \dfrac{p^\otimes_{1^31^3}}{2} 
+ p^\otimes_{2^3}, & \qquad &
E^+_{3^2} = \dfrac{p^\otimes_{1^6}}{3} - \dfrac{p^\otimes_{1^41^2}}{2} 
+ \dfrac{p^\otimes_{1^21^21^2}}{6} + p^\otimes_{2^21^2} + p^\otimes_{3^2}, 
\\[6pt]
E_{2^3} = -\dfrac{p^\otimes_{1^6}}{2} + \dfrac{p^\otimes_{1^31^3}}{2} 
- p^\otimes_{2^3}, & \qquad &
E_{3^2} = -\dfrac{p^\otimes_{1^6}}{3} + \dfrac{p^\otimes_{1^41^2}}{2} 
- \dfrac{p^\otimes_{1^21^21^2}}{6} + p^\otimes_{2^21^2} - p^\otimes_{3^2}.
\end{array} \]
For instance, we compute the coefficient of $p^{\otimes}_{1^41^2}$ 
in $E^+_{3^2}$ as follows. The type producing this term is $\tau=(1^21^1)$.
Here, $\ell(\tau) = (-1)^2 = 1$, $\sgn(\tau) = (-1)^{2+1} = -1$, 
and $z_{\tau}^{\otimes}=z_{(2,1)}=2$. So the required coefficient is $-1/2$.
 \end{example}

Combining~\cref{prop:block-in-ptensor} with the remark 
following~\eqref{eq:ptens-product} leads to algebraic formulas for
$p^{\otimes}$-expansions of various products of polysymmetric functions.
In the following subsections, we supplement these algebraic formulas
with combinatorial formulas that express
the final answers in terms of tableau-like structures.

\subsection{Rule for $p_{\sigma}^{\otimes}P_{\delta}$ 
and $\mcM(P,p^{\otimes})$.}
\label{subsec:p-times-P}

\begin{prop}\label{prop:p-times-P1}
For any type $\sigma$ and block $d^m$,
\[ p_{\sigma}^{\otimes}P_{d^m}
=\sum_{\tau} \wt(\sigma,\tau)p_{\tau}^{\otimes}, \]
where we sum over all types $\tau$ that arise from $\sigma$ by
choosing a positive divisor $k$ of $d$ and inserting one new part
of size $dm/k$ into $\sigma|_k$; and $\wt(\sigma,\tau)=k$ for each such $\tau$.
\end{prop}
\begin{proof}
Recall from~\eqref{eq:Pdm-in-p} that 
$P_{d^m}=\sum_{k|d} 1\otimes\cdots\otimes 1\otimes kp_{dm/k}\otimes 1\otimes
\cdots$, where $kp_{dm/k}$ occurs in position $k$.
Multiplying $p_{\sigma}^{\otimes}=p_{\sigma|_1}\otimes p_{\sigma|_2}\otimes
\cdots$ by this expression, we get
\[ p_{\sigma}^{\otimes}P_{d^m}
 =\sum_{k|d} p_{\sigma|_1}\otimes\cdots \otimes p_{\sigma|_k}kp_{dm/k}
 \otimes p_{\sigma|_{k+1}}\otimes\cdots. \]
Multiplying $p_{\sigma|_k}$ by $p_{dm/k}$ produces $p_{\tau|_k}$ where
$\tau$ is related to $\sigma$ as described in the proposition. 
The resulting term $p_{\tau}^{\otimes}$ in the expansion has coefficient $k$.
\end{proof}

Fix a type $\sigma=(1^{\sigma|_1}2^{\sigma|_2}\cdots )$ and an ordered
sequence of blocks $\delta=(d_1^{m_1},\ldots,d_s^{m_s})$.
Iteration of the rule in \cref{prop:p-times-P1} leads to
the $p^{\otimes}$-expansion of $p_{\sigma}^{\otimes}P_{\delta}$.
Starting with the tensor diagram of $\sigma$, we choose $k_i$ dividing $d_i$
(for $1\leq i\leq s$) and add a new part (weighted by $k_i$) of size
$d_im_i/k_i$ to the current partition diagram in tensor position $k_i$.
This produces the term $p^\otimes_\tau$ with the weight coefficient $\wt(\sigma, \tau) = k_i$. 
We get the required expansion by adding all such terms
generated by making all possible choices of divisors $(k_1,\ldots,k_s)$.

We now describe the answer in a different way, giving a combinatorial
formula for the net coefficient of each $p_{\tau}^{\otimes}$ in the output.
To do this, we define combinatorial structures 
(similar to TRHTs) that encode the required bookkeeping. 
We call these objects \emph{increasing constant-row $P$-tableaux} (ICRPTs).
Given $\sigma$ and $\delta$ as above, let $\tau=(1^{\tau|_1}2^{\tau|_2}\cdots)$ 
be a type such that for all $k,r$, $m_r(\tau|_k)\geq m_r(\sigma|_k)$.
Intuitively, this condition means that the tensor diagram for
$\tau$ arises from the tensor diagram for $\sigma$ by adding new parts
in various components. An ICRPT of \emph{shape} $\tau$ and \emph{extended
content} $(\sigma;\delta)$ is a filling $T$ of the cells in the tensor diagram
of $\tau$ with integers $0,1,\ldots,s$ satisfying these conditions:
\begin{itemize}
\item Each row of each $\tau|_k$ is constant (having the same value
in each cell).
\item For $1\leq i\leq s$, exactly one row in the tensor diagram of $\tau$
 contains the value $i$. If that row appears in $\tau|_k$ and has length $r$,
 then $rk=d_im_i$.
\item The cells containing $0$ in $T$ form a sub-tensor diagram that equals
 the tensor diagram of $\sigma$.
\item For each $r,k$, the values in the rows of $\tau|_k$ of length $r$
 weakly increase reading down the first column.
\end{itemize}
The \emph{weight} of the ICRPT $T$ is $\wt(T)=\prod_{k\geq 1} k^{n_k(T)}$,
where $n_k(T)$ is the number of rows in the diagram of $\tau|_k$
containing a nonzero value. Let $\ICRPT(\tau,(\sigma;\delta))$ be the
set of fillings $T$ satisfying these conditions. When $\sigma$ is empty,
we write $\ICRPT(\tau,\delta)$ for this set and call $\delta$ the
\emph{content} of $T$.

\begin{example}
 For $\sigma=(1^{3,1,1}2^{4,2}4^4)$, 
$\tau=(1^{4,3,1,1}2^{4,4,4,2}3^14^{4,2})$, and $\delta=(4^23^14^12^44^2)$,
the two objects in $\ICRPT(\tau,(\sigma;\delta))$ are shown here:
\ytableausetup{aligntableaux = top}
\[ \yt{3333,000,0,0}\otimes\yt{0000,1111,4444,00}\otimes\yt{2}
   \otimes\yt{0000,55} \qquad
   \yt{3333,000,0,0}\otimes\yt{0000,4444,5555,00}\otimes\yt{2}
   \otimes\yt{0000,11}  \]
Both objects have weight $1\cdot 2\cdot 2\cdot 3\cdot 4=48$ and thus the coefficient of $p^\otimes_\tau$ in $p^\otimes_\sigma P_\delta$ is 96.
In general, the weight of $T\in\ICRPT(\tau,(\sigma;\delta))$
depends only on $\tau$ and $\sigma$, not $\delta$.
\end{example}

\begin{theorem}\label{thm:p-times-P2}
For any type $\sigma$ and sequence $\delta=(d_1^{m_1},\ldots,d_s^{m_s})$,
\[ p_{\sigma}^{\otimes}P_{\delta}=\sum_{\tau}
 \left[\sum_{T\in\ICRPT(\tau,(\sigma;\delta))} \wt(T)\right]
  p_{\tau}^{\otimes}. \]
\end{theorem}
\begin{proof}
The entries in each ICRPT record the sequence of part
additions caused by starting at $p_{\sigma}^{\otimes}$ and successively
multiplying by $P_{d_1^{m_1}},\ldots,P_{d_s^{m_s}}$ in accordance with
\cref{prop:p-times-P1}. We start with the tensor diagram
of $\sigma$, which is filled with $0$s to indicate this is the initial
shape. For $i=1,2,\ldots,s$, the unique row containing value $i$
is the new row inserted into the tensor diagram due to the multiplication
by $P_{d_i^{m_i}}$. This row must appear in tensor position $k_i$, for some
$k_i$ dividing $d_i$, and must have length $r=d_im_i/k_i$. Each new row
is inserted in the proper position within the $k_i$th diagram so that
parts still appear in weakly decreasing order. If parts of length $r$
already exist in the $k_i$th diagram, the new part is placed just below them.
This is why values of $T$ must increase as we scan down through equal-length
parts in a given component of the tensor diagram. The net result of
all the part additions is a term $p^{\otimes}_{\tau}$.
Each new row added to the $k$th diagram multiplies this term by $k$,
so the net coefficient of this term is $\wt(T)$.
\end{proof}

\begin{corollary}\label{cor:M(P,p-tensor)}
For all types $\sigma,\tau\vdt n$, the coefficient of $p_{\tau}^{\otimes}$
in the $p^{\otimes}$-expansion of $P_{\sigma}$ is
$$\mcM(P,p^{\otimes})_{\tau,\sigma}=\sum_{T\in\ICRPT(\tau,\sigma)} \wt(T).$$
\end{corollary}

\begin{example}
\label{ex:ICRPT}
We find the $p^\otimes$-expansion of $p^\otimes_{2^21^3}P_{(2^2,4^1,2^2)}$. 
We compute one ICRPT step-by-step and present the rest in a figure. 
Here, $d_1^{m_1}= 2^2$, $d_2^{m_2} = 4^1$, and $d_3^{m_3} = 2^2$. 
Choose $k_ 1 = 2$, $k_2 = 2$, and $k_3 = 1$. 
First, since $k_1=2$, we place a row of length $d_1m_1/k_1 = 2$ 
with cells labeled $1$ in the second diagram.
Second, since $k_2=2$, we place another row of length $d_2m_2/k_2 = 2$ 
with cells labeled $2$ in the second diagram. 
Third, since $k_3=1$, we place a row of length $d_3m_3/k_3 = 4$ 
with cells labeled $3$ in the first diagram.
is added in the first tensor factor owing to the choice $k_3 = 1$. 
This gives the ICRPT
\[
\yt{3333,000} \otimes \yt{00,11,22} \otimes \varnothing \otimes \varnothing
\]
with weight $2\cdot 2\cdot 1 = 4$. \cref{fig:ICRPT} shows all ICRPTs
arising in~\cref{thm:p-times-P2} when $\sigma = 2^21^3$ 
and $\delta = (2^2,4^1,2^2)$.  Below each ICRPT, we show the 
tuple $(k_1,k_2,k_3)$ producing it and the weight of the ICRPT.
Combining all of this, we find the $p^{\otimes}$-expansion of
$p^\otimes_{1^32^2}P_{(2^2,4^1,2^2)}$ to be
\[ 
 1 p^\otimes_{1^{4,4,4,3}2^2} + 6 p^\otimes_{1^{4,4,3}2^{2,2}} 
+ 12p^\otimes_{1^{4,3}2^{2,2,2}} + 8p^\otimes_{1^32^{2,2,2,2}} 
+ 4 p^\otimes_{1^{4,4,3}2^24^1} + 16p^\otimes_{1^{4,3}2^{2,2}4^1}
+ 16p^{\otimes}_{1^3 2^{2,2,2} 4^1}.  \]
\end{example}
\begin{figure}[ht]
\begin{center}
\boks{0.4}
\[\begin{array}{ccc}
\yt{1111,2222,3333,000} \otimes \yt{00} \otimes \varnothing \otimes \varnothing & \yt{1111,2222,000} \otimes \yt{00,33} \otimes \varnothing \otimes \varnothing & \yt{2222,3333,000} \otimes \yt{00,11} \otimes \varnothing \otimes \varnothing \\
(1,1,1),\wt=1 & (1,1,2),\wt=2 & (2,1,1),\wt=2\\[0.4cm]

\yt{1111,3333,000} \otimes \yt{00,22} \otimes \varnothing \otimes \varnothing & \yt{1111,000} \otimes \yt{00,22,33} \otimes \varnothing \otimes \varnothing & \yt{2222,000} \otimes \yt{00,11,33} \otimes \varnothing \otimes \varnothing\\
(1,2,1),\wt=2 & (1,2,2),\wt=4 & (2,1,2),\wt=4\\[0.4cm]

\yt{3333,000} \otimes \yt{00,11,22} \otimes \varnothing \otimes \varnothing & \yt{000} \otimes \yt{00,11,22,33} \otimes \varnothing \otimes \varnothing & \yt{1111,3333,000} \otimes \yt{00} \otimes \varnothing \otimes \yt{2}\\
(2,2,1),\wt=4 & (2,2,2),\wt=8 & (1,4,1),\wt=4\\[0.4cm]

\yt{1111,000} \otimes \yt{00,33} \otimes \varnothing \otimes \yt{2} & \yt{3333,000} \otimes \yt{00,11} \otimes \varnothing \otimes \yt{2} & \yt{000} \otimes \yt{00,11,33} \otimes \varnothing \otimes \yt{2}\\
(1,4,2),\wt=8 & (2,4,1),\wt=8 & (2,4,2),\wt=16
\end{array}
\]
\caption{ICRPTs in~\cref{ex:ICRPT}.}
\label{fig:ICRPT}
\end{center}
\end{figure}

\subsection{Rule for $p_{\sigma}^{\otimes}H_{\delta}$ and $\mcM(H,p^{\otimes})$}
\label{subsec:Hp}

\begin{prop}\label{prop:p-times-H1}
For any type $\sigma$ and block $d^m$,
\[ p_{\sigma}^{\otimes}H_{d^m}
=\sum_{\tau\vdt d} \frac{1}{z_{\tau}^{\otimes}}p_{\sigma\cup\tau^m}^{\otimes}.\]
\end{prop}
\begin{proof}
The formula follows immediately from \cref{prop:block-in-ptensor}(b),
\eqref{eq:ptens-product}, and linearity.
\end{proof}

Here is a pictorial description of the rule in \cref{prop:p-times-H1}.
To compute the $p^{\otimes}$-expansion of $p_{\sigma}^{\otimes}H_{d_m}$,
start with the tensor diagram $\dg(\sigma)$. Choose any type $\tau\vdt d$.
For all $k\geq 1$, merge the partition diagrams $\dg(\sigma|_k)$
and $\dg(m\tau|_k)$ to get a new partition diagram in position $k$.
Weight the new tensor diagram by $1/z_{\tau}^{\otimes}
=\prod_{k\geq 1} z^{-1}_{\tau|_k}$. Add the resulting terms over all
choices of the type $\tau$.

Iteration of this rule leads to the $p^{\otimes}$-expansion
of $p_{\sigma}^{\otimes}H_{\delta}$, where $\sigma$ is a type
and $\delta=(d_1^{m_1},\ldots,d_s^{m_s})$ is a sequence of blocks.
Define an \emph{increasing constant-row $H$-tableau} (ICRHT)
of shape $\tau$ and extended content $(\sigma;\delta)$ to be
a filling $T$ of the cells in the tensor diagram $\dg(\tau)$ with
integers $0,1,\ldots,s$ satisfying these conditions:
\begin{itemize}
\item Each row of each diagram $\dg(\tau|_k)$ is constant.
\item The cells containing $0$ in $T$ form a sub-tensor diagram
 equal to $\dg(\sigma)$.
\item For $1\leq i\leq s$, the cells containing $i$ in $T$ form
a sub-tensor diagram equal to $\dg(m_i\rho^{(i)})$ for some type
 $\rho^{(i)}\vdt d_i$.
\item For each $r,k$, the values in the rows of $\tau|_k$ of length $r$
 weakly increase reading down the first column.
\end{itemize}

Let $\ICRHT(\tau,(\sigma;\delta))$ be the set of all such objects.
The \emph{weight} of an object $T$ in this set is
$\prod_{i=1}^s 1/z^{\otimes}_{\rho^{(i)}}$. 
Define $\sgn^+(T)=\prod_{i=1}^s (-1)^{\ell(\rho^{(i)})}\sgn(\rho^{(i)})$
and $\sgn^-(T)=\prod_{i=1}^s (-1)^{\ell(\rho^{(i)})}$. 
The exponent of $-1$ in $\sgn^-(T)$ is the number of rows with positive 
labels in the tensor diagram of $T$. To compute $\sgn^+(T)$ from the tensor 
diagram we do the following: for every label $i>0$, find the sub-tensor 
diagram formed by cells with label $i$, and divide the length of each row 
by $m_i$. Remove one cell from each row and call the total number of remaining 
cells $c_i$. Then $\sgn^+(T) = (-1)^{c_1 + c_2 + \ldots + c_s}$. 

\begin{theorem}\label{thm:p-times-H2}
For any type $\sigma$ and sequence $\delta=(d_1^{m_1},\ldots,d_s^{m_s})$,
\[ p_{\sigma}^{\otimes}H_{\delta}=\sum_{\tau}
 \left[\sum_{T\in\ICRHT(\tau,(\sigma;\delta))} \wt(T)\right]
  p_{\tau}^{\otimes}. \]
\end{theorem}
\begin{proof}
Start with $p_{\sigma}^{\otimes}$, modeled by the tensor diagram
$\dg(\sigma)$ with all cells containing $0$. For $i=1,2,\ldots,s$,
use \cref{prop:p-times-H1} to modify the current diagram to enact
multiplication by the next factor $H_{d_i^{m_i}}$. Do this by choosing
a type $\rho^{(i)}\vdt d_i$ and adding new parts given by $m_i\rho^{(i)}|_k$
to the $k$th diagram for all $k\geq 1$. Put $i$ in all cells in these new 
parts to record which factor created them. As before, new parts of 
size $r$ are placed immediately below existing parts of size $r$ in
each diagram. This explains the weakly increasing condition in the
definition of ICRHTs. The factor $\wt(T)$ accounts for all the weights
produced by each insertion step. Making these choices in all possible ways 
leads to the weighted set $\ICRHT(\tau,(\sigma;\delta))$ 
appearing in the theorem statement.
\end{proof}

\begin{corollary}\label{cor:M(H,p-tensor)}
For all types $\sigma,\tau\vdt n$, the coefficient of $p_{\tau}^{\otimes}$
in the $p^{\otimes}$-expansion of $H_{\sigma}$ is
$$\mcM(H,p^{\otimes})_{\tau,\sigma}=\sum_{T\in\ICRHT(\tau,\sigma)} \wt(T).$$
\end{corollary}

\begin{example}\label{ex:H-in-p} 
In this example, we compute the coefficient 
of $p^\otimes_\tau$ in the $p^{\otimes}$-expansion of $H_\sigma$, where
$\tau =(3^{2,1}2^{2,2,1}1^4)$ and $\sigma =(9^16^14^12^2)$. 
We construct the following six objects, 
each labeled by the tuple of types 
$(\rho^{(1)}\vdt 9,\rho^{(2)}\vdt 6,
  \rho^{(3)}\vdt 4,\rho^{(4)}\vdt 2)$ that produced it.
\begin{center}
$T_1 = \yt{4444}\otimes \yt{22,33,2}\otimes \yt{11,1}$ \qquad $T_2 = \yt{3333}\otimes \yt{22,44,2}\otimes \yt{11,1}$\\[0.1cm]
$((3^{2,1}), (2^{2,1}),(2^2),(1^2))$ 
\qquad $((3^{2,1}),(2^{2,1}),(1^4),(2^1))$\\[0.3cm]
$T_3 = \yt{4444}\otimes \yt{11,33,1}\otimes \yt{22,1}$ \qquad $T_4= \yt{3333}\otimes \yt{11,44,1}\otimes \yt{22,1}$\\[0.1cm]
$((3^12^{2,1}),(3^2),(2^2),(1^2))$ \qquad 
$((3^12^{2,1}),(3^2),(1^4),(2^1))$\\[0.3cm]
$T_5 = \yt{2222}\otimes \yt{33,44,2}\otimes \yt{11,1}$ \qquad $T_6= \yt{1111}\otimes \yt{33,44,1}\otimes \yt{22,1}$\\[0.1cm]
$((3^{2,1}),(2^11^4),(2^2),(2^1))$ \qquad 
$((3^12^11^4),(3^2),(2^2),(2^1))$
\end{center} 
The weight of the first ICRHT is 
$\wt(T_1)=z_{(2,1)}^{-1}z_{(2,1)}^{-1}z_{(2)}^{-1}z_{(2)}^{-1}
= \left(\frac{1}{2}\right)^4 = \frac{1}{16}$. Similarly, all six ICRHTs
shown here have weight $\frac{1}{16}$. So 
$\mcM(H,p^\otimes)_{\tau,\sigma} = \frac{3}{8}$.
\end{example}

\begin{remark} 
\boks{0.23}
\ytableausetup{aligntableaux = center}
In general, not all objects in $\ICRHT(\tau, \sigma)$ have the same weight. 
For example, let $\tau=(1^{1,1,1}2^{1,1,1})$ and $\sigma =(5^14^1)$. 
Two objects in $\ICRHT(\tau,\sigma)$ are
$T' = {\tiny \yt{1,1,1}\otimes \yt{1,2,2}}$ 
and $T'' ={\tiny \yt{1,2,2}\otimes \yt{1,1,2}}$,
arising from type choices $((1^{1,1,1}2^1),(2^{1,1}))$ for $T'$
                     and  $((1^12^{1,1}),(1^{1,1}2^{1}))$ for $T''$.
We compute $\wt(T')=z_{(1^3)}^{-1}z_{(1)}^{-1}z_{(1^2)}^{-1}=1/12$
       and $\wt(T'')=z_{(1)}^{-1}z_{(1^2)}^{-1}z_{(1^2)}^{-1}z_{(1)}^{-1}=1/4$.
\boks{0.4}
\end{remark}

\subsection{Rule for $p_{\sigma}^{\otimes}E^+_{\delta}$,
$p_{\sigma}^{\otimes}E_{\delta}$, $\mcM(E^+,p^{\otimes})$,
and $\mcM(E,p^{\otimes})$}
\label{subsec:Ep}

The next three results follow immediately by adapting the proofs
in the previous subsection, keeping in mind \cref{prop:block-in-ptensor}(c)
and (d).

\begin{prop}\label{prop:p-times-E1}
For any type $\sigma$ and block $d^m$,
\[ p_{\sigma}^{\otimes}E^+_{d^m}
=\sum_{\tau\vdt d} \frac{(-1)^{\ell(\tau)}\sgn(\tau)}{z_{\tau}^{\otimes}}
p_{\sigma\cup\tau^m}^{\otimes}
\quad\mbox{ and }\quad
 p_{\sigma}^{\otimes}E_{d^m} =
\sum_{\tau\vdt d} \frac{(-1)^{\ell(\tau)}}{z_{\tau}^{\otimes}}
p_{\sigma\cup\tau^m}^{\otimes}.\]
\end{prop}

\begin{theorem}\label{thm:p-times-E2}
For any type $\sigma$ and sequence $\delta=(d_1^{m_1},\ldots,d_s^{m_s})$,
\[ p_{\sigma}^{\otimes}E^+_{\delta}=\sum_{\tau}
 \left[\sum_{T\in\ICRHT(\tau,(\sigma;\delta))} \sgn^+(T)\wt(T)\right]
  p_{\tau}^{\otimes}; \]
\[ p_{\sigma}^{\otimes}E_{\delta}=\sum_{\tau}
 \left[\sum_{T\in\ICRHT(\tau,(\sigma;\delta))} \sgn^-(T)\wt(T)\right]
  p_{\tau}^{\otimes}. \]
\end{theorem}

\begin{corollary}\label{cor:M(E,p-tensor)}
For all types $\sigma,\tau\vdt n$, the coefficient of $p_{\tau}^{\otimes}$
in the $p^{\otimes}$-expansion of $E^+_{\sigma}$ is
$$\mcM(E^+,p^{\otimes})_{\tau,\sigma}=\sum_{T\in\ICRHT(\tau,\sigma)} 
 \sgn^+(T)\wt(T).$$
The coefficient of $p_{\tau}^{\otimes}$
in the $p^{\otimes}$-expansion of $E_{\sigma}$ is
$$\mcM(E,p^{\otimes})_{\tau,\sigma}=\sum_{T\in\ICRHT(\tau,\sigma)} 
 \sgn^-(T)\wt(T).$$
\end{corollary}

\begin{example}
We continue with \cref{ex:H-in-p} where 
$\tau =(3^{2,1}2^{2,2,1}1^4)$ and $\sigma =(9^16^14^12^2)$. 
For $i$ between $1$ and $6$, $\sgn^-(T_i) = (-1)^6$ since there
are 6 rows in $\dg(\tau)$, all filled with positive labels.
So the coefficient of $p^\otimes_\tau$ in the $p^{\otimes}$-expansion
of $E_\sigma$ is $\frac{3}{8}$. 
On the other hand, $\sgn^+(T_1) = \sgn^+(T_3)= (-1)^{10-6}=1$,
while $\sgn^+(T_2)=\sgn^+(T_4)=\sgn^+(T_5)=\sgn^+(T_6)=(-1)^{11-6}=-1$.
So the coefficient of $p^\otimes_\tau$ in the $p^{\otimes}$-expansion
of $E^+_\sigma$ is $-\frac{1}{8}$.
\end{example}

\section{Expansions in the $m^{\otimes}$ Basis}
\label{sec:m-expand}

\subsection{Rule for $m_{\sigma}^{\otimes}P_{\delta}$ and $\mcM(P,m^{\otimes})$}
\label{subsec:Pm}

Before stating our combinatorial rule for the $m^{\otimes}$-expansion of
$m_{\sigma}P_{\delta}$, we describe an analogous rule (cf.~\cite{eg-rem})
for the monomial expansion of $m_{\mu}p_{\alpha}$, where 
$\mu=(\mu_1,\ldots,\mu_{\ell})$ is an integer partition and 
$\alpha=(\alpha_1,\ldots,\alpha_s)$ is a sequence of positive integers.
We create $s$ horizontal bricks, namely, 
a brick containing $\alpha_1$ boxes labeled $1$,
a brick containing $\alpha_2$ boxes labeled $2$, $\ldots$, and
a brick containing $\alpha_s$ boxes labeled $s$.
We also create $\ell$ horizontal bricks of lengths
$\mu_1,\ldots,\mu_{\ell}$ with all boxes in these bricks labeled $0$.
For a given partition $\lambda$,
draw the diagram of $\lambda$ and place these bricks in
this diagram so that every box in the diagram is covered by exactly one
brick, and the brick labels strictly increase reading left
to right in each row. (Strict increase means that a row can contain
at most one brick labeled $0$.) Two bricks of the same length, with boxes 
labeled $0$, are considered indistinguishable. Call such a configuration
a \emph{$p$-brick tabloid} of \emph{shape} $\lambda$ and \emph{extended content}
$(\mu;\alpha)$.

\begin{prop}\label{prop:mp-in-m}
For any partitions $\lambda$, $\mu$ and list of positive integers $\alpha$,
the coefficient of $m_{\lambda}$ in $m_{\mu}p_{\alpha}$ is the number
of $p$-brick tabloids of shape $\lambda$ and extended content $(\mu;\alpha)$. 
\end{prop}
\begin{proof}
The coefficient of $m_{\lambda}$ in the $m$-expansion of 
$m_{\mu}p_{\alpha}$ equals the coefficient of the particular monomial
$\xx^{\lambda}=x_1^{\lambda_1}x_2^{\lambda_2}\cdots x_k^{\lambda_k}\cdots$ 
in the polynomial $m_{\mu}(\xx)p_{\alpha}(\xx)$. The $p$-brick tabloids 
described in the proposition record all the ways the monomial $\xx^{\lambda}$ 
can be 
generated by choosing particular monomials from each factor $m_{\mu}(\xx)$, 
$p_{\alpha_1}(\xx)$, $\ldots$, $p_{\alpha_s}(\xx)$ and multiplying those 
monomials together in accordance with the distributive law.

In more detail, the placement of all the bricks labeled $0$ in distinct rows
$i_1,i_2,\ldots,i_{\ell}$ records a monomial $x_{i_1}^{\mu_1}x_{i_2}^{\mu_2}
\cdots x_{i_{\ell}}^{\mu_{\ell}}$ coming from $m_{\mu}(\xx)$.
The placement of the brick of length $\alpha_1$ labeled $1$ in some row
$j_1$ records a monomial $x_{j_1}^{\alpha_1}$ coming from $p_{\alpha_1}(\xx)$.
The placement of the brick of length $\alpha_2$ labeled $2$ in some row
$j_2$ records a monomial $x_{j_2}^{\alpha_2}$ coming from $p_{\alpha_2}(\xx)$.
And so on. Since the $p$-brick tabloid covers each cell in row $k$ of the
diagram of $\lambda$ with exactly one brick, we see that the power of 
$x_k$ in the generated monomial is $\lambda_k$ for all $k$, as needed.
Brick labels increase from left to right in each row since we 
place the bricks in the diagram in the same order that the choices
of monomials are made from $m_{\mu}(\xx)$ (bricks labeled $0$), 
$p_{\alpha_1}(\xx)$ (brick labeled $1$), $\ldots$,
$p_{\alpha_s}(\xx)$ (brick labeled $s$).
\end{proof}

\begin{example}
Let $\mu = (3,3,1)$ and $\alpha = (2,4,2)$. We find the coefficient of 
$m_{(5,4,3,3)}$ in $m_{\mu}p_{\alpha}$ to be 6 by counting 
the following $p$-brick tabloids.
\begin{center}
\begin{tikzpicture}[scale = 0.4]
\draw (0,0) rectangle (3,-1);
\draw (0.5,-0.5) node {0};
\draw (1.5,-0.5) node {0};
\draw (2.5,-0.5) node {0};
\draw (3,0) rectangle (5,-1);
\draw (3.5,-0.5) node {1};
\draw (4.5,-0.5) node {1};
\draw (0,-1) rectangle (4,-2);
\draw (0.5,-1.5) node {2};
\draw (1.5,-1.5) node {2};
\draw (2.5,-1.5) node {2};
\draw (3.5,-1.5) node {2};
\draw (0,-2) rectangle (3,-3);
\draw (0.5,-2.5) node {0};
\draw (1.5,-2.5) node {0};
\draw (2.5,-2.5) node {0};
\draw (0,-3) rectangle (1,-4);
\draw (0.5,-3.5) node {0};
\draw (1,-3) rectangle (3,-4);
\draw (1.5,-3.5) node {3};
\draw (2.5,-3.5) node {3};
\end{tikzpicture}
\space{}\quad 
\begin{tikzpicture}[scale = 0.4]
\draw (0,0) rectangle (3,-1);
\draw (0.5,-0.5) node {0};
\draw (1.5,-0.5) node {0};
\draw (2.5,-0.5) node {0};
\draw (3,0) rectangle (5,-1);
\draw (3.5,-0.5) node {1};
\draw (4.5,-0.5) node {1};
\draw (0,-1) rectangle (4,-2);
\draw (0.5,-1.5) node {2};
\draw (1.5,-1.5) node {2};
\draw (2.5,-1.5) node {2};
\draw (3.5,-1.5) node {2};
\draw (0,-2) rectangle (1,-3);
\draw (0.5,-2.5) node {0};
\draw (1,-2) rectangle (3,-3);
\draw (1.5,-2.5) node {3};
\draw (2.5,-2.5) node {3};
\draw (0,-3) rectangle (3,-4);
\draw (0.5,-3.5) node {0};
\draw (1.5,-3.5) node {0};
\draw (2.5,-3.5) node {0};
\end{tikzpicture}
\space{}\quad 
\begin{tikzpicture}[scale = 0.4]
\draw (0,0) rectangle (3,-1);
\draw (0.5,-0.5) node {0};
\draw (1.5,-0.5) node {0};
\draw (2.5,-0.5) node {0};
\draw (3,0) rectangle (5,-1);
\draw (3.5,-0.5) node {3};
\draw (4.5,-0.5) node {3};
\draw (0,-1) rectangle (4,-2);
\draw (0.5,-1.5) node {2};
\draw (1.5,-1.5) node {2};
\draw (2.5,-1.5) node {2};
\draw (3.5,-1.5) node {2};
\draw (0,-2) rectangle (3,-3);
\draw (0.5,-2.5) node {0};
\draw (1.5,-2.5) node {0};
\draw (2.5,-2.5) node {0};
\draw (0,-3) rectangle (1,-4);
\draw (0.5,-3.5) node {0};
\draw (1,-3) rectangle (3,-4);
\draw (1.5,-3.5) node {1};
\draw (2.5,-3.5) node {1};
\end{tikzpicture}
\\[0.4cm]
\begin{tikzpicture}[scale = 0.4]
\draw (0,0) rectangle (3,-1);
\draw (0.5,-0.5) node {0};
\draw (1.5,-0.5) node {0};
\draw (2.5,-0.5) node {0};
\draw (3,0) rectangle (5,-1);
\draw (3.5,-0.5) node {3};
\draw (4.5,-0.5) node {3};
\draw (0,-1) rectangle (4,-2);
\draw (0.5,-1.5) node {2};
\draw (1.5,-1.5) node {2};
\draw (2.5,-1.5) node {2};
\draw (3.5,-1.5) node {2};
\draw (0,-2) rectangle (1,-3);
\draw (0.5,-2.5) node {0};
\draw (1,-2) rectangle (3,-3);
\draw (1.5,-2.5) node {1};
\draw (2.5,-2.5) node {1};
\draw (0,-3) rectangle (3,-4);
\draw (0.5,-3.5) node {0};
\draw (1.5,-3.5) node {0};
\draw (2.5,-3.5) node {0};
\end{tikzpicture}
\space{}\quad 
\begin{tikzpicture}[scale = 0.4]
\draw (0,0) rectangle (1,-1);
\draw (0.5,-0.5) node {0};
\draw (1,0) rectangle (3,-1);
\draw (1.5,-0.5) node {1};
\draw (2.5,-0.5) node {1};
\draw (3,0) rectangle (5,-1);
\draw (3.5,-0.5) node {3};
\draw (4.5,-0.5) node {3};
\draw (0,-1) rectangle (4,-2);
\draw (0.5,-1.5) node {2};
\draw (1.5,-1.5) node {2};
\draw (2.5,-1.5) node {2};
\draw (3.5,-1.5) node {2};
\draw (0,-2) rectangle (3,-3);
\draw (0.5,-2.5) node {0};
\draw (1.5,-2.5) node {0};
\draw (2.5,-2.5) node {0};
\draw (0,-3) rectangle (3,-4);
\draw (0.5,-3.5) node {0};
\draw (1.5,-3.5) node {0};
\draw (2.5,-3.5) node {0};
\end{tikzpicture}
\space{}\quad 
\begin{tikzpicture}[scale = 0.4]
\draw (0,0) rectangle (1,-1);
\draw (0.5,-0.5) node {0};
\draw (1,0) rectangle (5,-1);
\draw (1.5,-0.5) node {2};
\draw (2.5,-0.5) node {2};
\draw (3.5,-0.5) node {2};
\draw (4.5,-0.5) node {2};
\draw (0,-1) rectangle (2,-2);
\draw (0.5,-1.5) node {1};
\draw (1.5,-1.5) node {1};
\draw (2,-1) rectangle (4,-2);
\draw (2.5,-1.5) node {3};
\draw (3.5,-1.5) node {3};
\draw (0,-2) rectangle (3,-3);
\draw (0.5,-2.5) node {0};
\draw (1.5,-2.5) node {0};
\draw (2.5,-2.5) node {0};
\draw (0,-3) rectangle (3,-4);
\draw (0.5,-3.5) node {0};
\draw (1.5,-3.5) node {0};
\draw (2.5,-3.5) node {0};
\end{tikzpicture}
\space{}\end{center}
\end{example}

Turning to the polysymmetric case, fix types $\tau$ and $\sigma$,
and fix an ordered sequence of blocks $\delta=(d_1^{e_1},\ldots,d_s^{e_s})$.
We seek the coefficient of $m_{\tau}^{\otimes}$ in the
$m^{\otimes}$-expansion of $m_{\sigma}^{\otimes}P_{\delta}$.
We describe this coefficient as the weighted sum of 
\emph{$P$-tensor brick tabloids} constructed as follows.
We fill the tensor diagram of $\tau$ with certain horizontal bricks
labeled $0,1,2,\ldots,s$ so that every box is covered by exactly one brick.
The brick labels in each row of each component diagram must strictly
increase reading left to right.  In each tensor component $k$, we use
$\ell(\sigma|_k)$ bricks labeled $0$, with lengths given by the parts of 
the partition $\sigma|_k$. Next, fix $i$ between $1$ and $s$.
Recall from~\eqref{eq:Pdm-in-p} that 
$P_{d_i^{e_i}}=\sum_{k_i|d_i} k_i p_{d_ie_i/k_i}(\xx_{k_i*})$.
When building a particular $P$-tensor brick tabloid, we may use exactly one 
brick labeled $i$, chosen as follows: pick a positive divisor $k_i$ of $d_i$;
make a brick labeled $i$ containing $d_ie_i/k_i$ cells; and place that
brick in the $k_i$th component diagram of $\dg(\tau)$. 
Every positively-labeled brick placed in component diagram $k$ has
a \emph{weight} of $k$, while bricks labeled $0$ have weight $1$.

Any filling $T$ of $\dg(\tau)$ satisfying all rules stated here is called
a \emph{$P$-tensor brick tabloid (PTBT)} of \emph{shape} $\tau$ and
\emph{extended content} $(\sigma;\delta)$. 
Let $\PTBT(\tau,(\sigma;\delta))$ be the set of all such objects.
When $\sigma$ is empty, we write $\PTBT(\tau,\delta)$ and speak of
PTBT of shape $\tau$ and content $\delta$.
The weight of a PTBT $T$, written $\wt(T)$, 
is the product of the weights of all the bricks in it.
Equivalently, if component diagram $k$ in $T$ contains $n_k(T)$ bricks
with positive labels, then $\wt(T)=\prod_{k\geq 1} k^{n_k(T)}$.

\begin{theorem}
For any type $\sigma$ and sequence of blocks $\delta$,
\[ m_{\sigma}^{\otimes}P_{\delta}
  =\sum_{\tau}\left[\sum_{T\in\PTBT(\tau,(\sigma;\delta))} \wt(T)\right]
  m_{\tau}^{\otimes}. \] 
\end{theorem}
\begin{proof}
We expand $m_{\sigma}^{\otimes}(\xx_{**})P_{\delta}(\xx_{**})$ by
choosing one monomial from each factor, multiplying those monomials,
and adding over all possible choices of monomials. The weighted $P$-tensor brick
tabloids in $\PTBT(\tau,(\sigma;\delta))$ record all possible ways the
monomial $\xx^{\tau}=\xx_{1*}^{\tau|_1}\xx_{2*}^{\tau|_2}\cdots
 \xx_{k*}^{\tau|_k}\cdots$ can arise by such choices. 
The choice of a monomial from 
$m_{\sigma}^{\otimes}(\xx_{**})=\prod_{k\geq 1} m_{\sigma|_k}(\xx_{k*})$
is recorded by the placement of all the bricks labeled $0$.
For $1\leq i\leq s$, the choice of a monomial from $P_{d_i^{e_i}}(\xx_{**})$
is recorded by the placement of the brick labeled $i$ in some component
diagram $k_i$, including the appropriate weight $k_i$. The monomial choices
correspond bijectively to the objects in $\PTBT(\tau,(\sigma;\delta))$ as 
explained in the proof of \cref{prop:mp-in-m}.
\end{proof}

\begin{corollary}
For all types $\sigma,\tau\vdt n$, the coefficient of $m_{\tau}^{\otimes}$
in the $m^{\otimes}$-expansion of $P_{\sigma}$ is
\[ \mcM(P,m^{\otimes})_{\tau,\sigma} =\sum_{T\in\PTBT(\tau,\sigma)} \wt(T). \]
\end{corollary}

\begin{example}
We compute the $m^\otimes$-expansion of $P_{2^22^2}$ by drawing the following 
PTBTs. Each PTBT $T$ is labeled by the divisor pair $(k_1, k_2)$ 
that produced it and its weight, namely $\wt(T)=k_1k_2$.
\begin{center}
\begin{tikzpicture}[scale = 0.4]
\draw (0,0) rectangle (4,-1);
\draw (0.5,-0.5) node {1};
\draw (1.5,-0.5) node {1};
\draw (2.5,-0.5) node {1};
\draw (3.5,-0.5) node {1};
\draw (0,-1) rectangle (4,-2);
\draw (0.5,-1.5) node {2};
\draw (1.5,-1.5) node {2};
\draw (2.5,-1.5) node {2};
\draw (3.5,-1.5) node {2};
\end{tikzpicture}
\space{}\raisebox{7pt}{$\otimes \,\varnothing$}
\qquad \begin{tikzpicture}[scale = 0.4]
\draw (0,0) rectangle (4,-1);
\draw (0.5,-0.5) node {2};
\draw (1.5,-0.5) node {2};
\draw (2.5,-0.5) node {2};
\draw (3.5,-0.5) node {2};
\draw (0,-1) rectangle (4,-2);
\draw (0.5,-1.5) node {1};
\draw (1.5,-1.5) node {1};
\draw (2.5,-1.5) node {1};
\draw (3.5,-1.5) node {1};
\end{tikzpicture}
\space{}\raisebox{7pt}{$\otimes \,\varnothing$}
\qquad \begin{tikzpicture}[scale = 0.4]
\draw (0,0) rectangle (4,-1);
\draw (0.5,-0.5) node {1};
\draw (1.5,-0.5) node {1};
\draw (2.5,-0.5) node {1};
\draw (3.5,-0.5) node {1};
\draw (4,0) rectangle (8,-1);
\draw (4.5,-0.5) node {2};
\draw (5.5,-0.5) node {2};
\draw (6.5,-0.5) node {2};
\draw (7.5,-0.5) node {2};
\end{tikzpicture}
\space{}\raisebox{4pt}{$\otimes\, \varnothing$}
\qquad \begin{tikzpicture}[scale = 0.4]
\draw (0,0) rectangle (4,-1);
\draw (0.5,-0.5) node {1};
\draw (1.5,-0.5) node {1};
\draw (2.5,-0.5) node {1};
\draw (3.5,-0.5) node {1};
\end{tikzpicture}
\raisebox{4pt}{$\otimes\:$}\begin{tikzpicture}[scale = 0.4]
\draw (0,0) rectangle (2,-1);
\draw (0.5,-0.5) node {2};
\draw (1.5,-0.5) node {2};
\end{tikzpicture}
\space{}\\
\hspace{-1cm} $(1,1),\wt=1$ \qquad\ \ 
$(1,1),\wt=1$ \qquad\qquad 
$(1,1),\wt=1$ \qquad\qquad\qquad\ \ 
$(1,2),\wt=2$\\[0.3cm]

\begin{tikzpicture}[scale = 0.4]
\draw (0,0) rectangle (4,-1);
\draw (0.5,-0.5) node {2};
\draw (1.5,-0.5) node {2};
\draw (2.5,-0.5) node {2};
\draw (3.5,-0.5) node {2};
\end{tikzpicture}
\space{}\raisebox{4pt}{$\,\otimes\,$} \begin{tikzpicture}[scale = 0.4]
\draw (0,0) rectangle (2,-1);
\draw (0.5,-0.5) node {1};
\draw (1.5,-0.5) node {1};
\end{tikzpicture}
\space{}\qquad \raisebox{7pt}{$\varnothing\, \otimes$} \begin{tikzpicture}[scale = 0.4]
\draw (0,0) rectangle (2,-1);
\draw (0.5,-0.5) node {1};
\draw (1.5,-0.5) node {1};
\draw (0,-1) rectangle (2,-2);
\draw (0.5,-1.5) node {2};
\draw (1.5,-1.5) node {2};
\end{tikzpicture}
\space{}\qquad \raisebox{7pt}{$\varnothing\, \otimes\,$}\begin{tikzpicture}[scale = 0.4]
\draw (0,0) rectangle (2,-1);
\draw (0.5,-0.5) node {2};
\draw (1.5,-0.5) node {2};
\draw (0,-1) rectangle (2,-2);
\draw (0.5,-1.5) node {1};
\draw (1.5,-1.5) node {1};
\end{tikzpicture}
\space{}\qquad \raisebox{4pt}{$\varnothing\, \otimes\,$}\begin{tikzpicture}[scale = 0.4]
\draw (0,0) rectangle (2,-1);
\draw (0.5,-0.5) node {1};
\draw (1.5,-0.5) node {1};
\draw (2,0) rectangle (4,-1);
\draw (2.5,-0.5) node {2};
\draw (3.5,-0.5) node {2};
\end{tikzpicture}
\\[0.05cm] 
\quad $(2,1),\wt=2$ \qquad\qquad\qquad\quad 
$(2,2),\wt=4$ \quad\quad 
$(2,2),\wt=4$ \qquad 
$(2,2),\wt=4$
\end{center}
This gives 
$P_{2^22^2} = 2 m^\otimes_{1^41^4} + m^\otimes_{1^8} 
+ 4m^\otimes_{2^21^4} + 8m^\otimes_{2^22^2} + 4m^\otimes_{2^4}.$
\end{example}

\subsection{Rule for $m^{\otimes}_{\sigma}H_{\delta}$ and $\mcM(H,m^{\otimes})$}
\label{subsec:Hm}

In~\cite{polysymm}, the authors show that the coefficient of $m_\tau^\otimes$ 
in $H_\sigma$ is the number of arrangements of one type into another. 
They write $a_{\tau,\sigma}$ for what we call 
$\mcM(H,m^{\otimes})_{\tau,\sigma}$, so
$H_\sigma = \sum\limits_{\tau \vdt |\sigma|} a_{\tau,\sigma}m^\otimes_\tau$.
Here we develop alternate combinatorial formulas for these coefficients based 
on tensor versions of brick tabloids, by extending classical results for the 
symmetric case (cf.~\cite{eg-rem}) to the polysymmetric case. 

Let $\mu$ and $\lambda$ be partitions, and 
let $\alpha=(\alpha_1,\ldots,\alpha_s)$ be a sequence of positive integers. 
Define an \textit{$h$-brick tabloid of shape $\lambda$ and
 extended content $(\mu;\alpha)$} as follows. Construct $\alpha_i$ 
$1\times 1$ bricks labeled $i$ and $\ell(\mu)$ bricks labeled $0$
of lengths $\mu_1,\mu_2,\ldots$. An $h$-brick tabloid
is a non-overlapping cover of $\dg(\lambda)$ using these bricks 
such that each brick labeled 0 appears at most once in a row
while brick labels weakly increase along rows. 

\begin{example}\label{example:h-in-m}
The $h$-brick tabloids of shape $(4,4)$ with extended content $((2,1);(2,1,2))$ are
 \begin{center}
 \begin{tikzpicture}[scale = 0.4]
\draw (0,0) rectangle (2,-1);
\draw (0.5,-0.5) node {0};
\draw (1.5,-0.5) node {0};
\draw (2,0) rectangle (3,-1);
\draw (2.5,-0.5) node {1};
\draw (3,0) rectangle (4,-1);
\draw (3.5,-0.5) node {1};
\draw (0,-1) rectangle (1,-2);
\draw (0.5,-1.5) node {0};
\draw (1,-1) rectangle (2,-2);
\draw (1.5,-1.5) node {2};
\draw (2,-1) rectangle (3,-2);
\draw (2.5,-1.5) node {3};
\draw (3,-1) rectangle (4,-2);
\draw (3.5,-1.5) node {3};
\end{tikzpicture}
\space{}\quad 
\begin{tikzpicture}[scale = 0.4]
\draw (0,0) rectangle (2,-1);
\draw (0.5,-0.5) node {0};
\draw (1.5,-0.5) node {0};
\draw (2,0) rectangle (3,-1);
\draw (2.5,-0.5) node {1};
\draw (3,0) rectangle (4,-1);
\draw (3.5,-0.5) node {3};
\draw (0,-1) rectangle (1,-2);
\draw (0.5,-1.5) node {0};
\draw (1,-1) rectangle (2,-2);
\draw (1.5,-1.5) node {1};
\draw (2,-1) rectangle (3,-2);
\draw (2.5,-1.5) node {2};
\draw (3,-1) rectangle (4,-2);
\draw (3.5,-1.5) node {3};
\end{tikzpicture}
\space{}\quad \begin{tikzpicture}[scale = 0.4]
\draw (0,0) rectangle (2,-1);
\draw (0.5,-0.5) node {0};
\draw (1.5,-0.5) node {0};
\draw (2,0) rectangle (3,-1);
\draw (2.5,-0.5) node {3};
\draw (3,0) rectangle (4,-1);
\draw (3.5,-0.5) node {3};
\draw (0,-1) rectangle (1,-2);
\draw (0.5,-1.5) node {0};
\draw (1,-1) rectangle (2,-2);
\draw (1.5,-1.5) node {1};
\draw (2,-1) rectangle (3,-2);
\draw (2.5,-1.5) node {1};
\draw (3,-1) rectangle (4,-2);
\draw (3.5,-1.5) node {2};
\end{tikzpicture}
\space{}\quad
\begin{tikzpicture}[scale = 0.4]
\draw (0,0) rectangle (2,-1);
\draw (0.5,-0.5) node {0};
\draw (1.5,-0.5) node {0};
\draw (2,0) rectangle (3,-1);
\draw (2.5,-0.5) node {2};
\draw (3,0) rectangle (4,-1);
\draw (3.5,-0.5) node {3};
\draw (0,-1) rectangle (1,-2);
\draw (0.5,-1.5) node {0};
\draw (1,-1) rectangle (2,-2);
\draw (1.5,-1.5) node {1};
\draw (2,-1) rectangle (3,-2);
\draw (2.5,-1.5) node {1};
\draw (3,-1) rectangle (4,-2);
\draw (3.5,-1.5) node {3};
\end{tikzpicture}
\space{}\quad \begin{tikzpicture}[scale = 0.4]
\draw (0,0) rectangle (2,-1);
\draw (0.5,-0.5) node {0};
\draw (1.5,-0.5) node {0};
\draw (2,0) rectangle (3,-1);
\draw (2.5,-0.5) node {1};
\draw (3,0) rectangle (4,-1);
\draw (3.5,-0.5) node {2};
\draw (0,-1) rectangle (1,-2);
\draw (0.5,-1.5) node {0};
\draw (1,-1) rectangle (2,-2);
\draw (1.5,-1.5) node {1};
\draw (2,-1) rectangle (3,-2);
\draw (2.5,-1.5) node {3};
\draw (3,-1) rectangle (4,-2);
\draw (3.5,-1.5) node {3};
\end{tikzpicture}
\\[0.4cm]
\begin{tikzpicture}[scale = 0.4]
\draw (0,0) rectangle (1,-1);
\draw (0.5,-0.5) node {0};
\draw (1,0) rectangle (2,-1);
\draw (1.5,-0.5) node {2};
\draw (2,0) rectangle (3,-1);
\draw (2.5,-0.5) node {3};
\draw (3,0) rectangle (4,-1);
\draw (3.5,-0.5) node {3};
\draw (0,-1) rectangle (2,-2);
\draw (0.5,-1.5) node {0};
\draw (1.5,-1.5) node {0};
\draw (2,-1) rectangle (3,-2);
\draw (2.5,-1.5) node {1};
\draw (3,-1) rectangle (4,-2);
\draw (3.5,-1.5) node {1};
\end{tikzpicture}
\space{}\quad
\begin{tikzpicture}[scale = 0.4]
\draw (0,0) rectangle (1,-1);
\draw (0.5,-0.5) node {0};
\draw (1,0) rectangle (2,-1);
\draw (1.5,-0.5) node {1};
\draw (2,0) rectangle (3,-1);
\draw (2.5,-0.5) node {2};
\draw (3,0) rectangle (4,-1);
\draw (3.5,-0.5) node {3};
\draw (0,-1) rectangle (2,-2);
\draw (0.5,-1.5) node {0};
\draw (1.5,-1.5) node {0};
\draw (2,-1) rectangle (3,-2);
\draw (2.5,-1.5) node {1};
\draw (3,-1) rectangle (4,-2);
\draw (3.5,-1.5) node {3};
\end{tikzpicture}
\space{}\quad
\begin{tikzpicture}[scale = 0.4]
\draw (0,0) rectangle (1,-1);
\draw (0.5,-0.5) node {0};
\draw (1,0) rectangle (2,-1);
\draw (1.5,-0.5) node {1};
\draw (2,0) rectangle (3,-1);
\draw (2.5,-0.5) node {1};
\draw (3,0) rectangle (4,-1);
\draw (3.5,-0.5) node {2};
\draw (0,-1) rectangle (2,-2);
\draw (0.5,-1.5) node {0};
\draw (1.5,-1.5) node {0};
\draw (2,-1) rectangle (3,-2);
\draw (2.5,-1.5) node {3};
\draw (3,-1) rectangle (4,-2);
\draw (3.5,-1.5) node {3};
\end{tikzpicture}
\space{}\quad
\begin{tikzpicture}[scale = 0.4]
\draw (0,0) rectangle (1,-1);
\draw (0.5,-0.5) node {0};
\draw (1,0) rectangle (2,-1);
\draw (1.5,-0.5) node {1};
\draw (2,0) rectangle (3,-1);
\draw (2.5,-0.5) node {1};
\draw (3,0) rectangle (4,-1);
\draw (3.5,-0.5) node {3};
\draw (0,-1) rectangle (2,-2);
\draw (0.5,-1.5) node {0};
\draw (1.5,-1.5) node {0};
\draw (2,-1) rectangle (3,-2);
\draw (2.5,-1.5) node {2};
\draw (3,-1) rectangle (4,-2);
\draw (3.5,-1.5) node {3};
\end{tikzpicture}
\space{}\quad
\begin{tikzpicture}[scale = 0.4]
\draw (0,0) rectangle (1,-1);
\draw (0.5,-0.5) node {0};
\draw (1,0) rectangle (2,-1);
\draw (1.5,-0.5) node {1};
\draw (2,0) rectangle (3,-1);
\draw (2.5,-0.5) node {3};
\draw (3,0) rectangle (4,-1);
\draw (3.5,-0.5) node {3};
\draw (0,-1) rectangle (2,-2);
\draw (0.5,-1.5) node {0};
\draw (1.5,-1.5) node {0};
\draw (2,-1) rectangle (3,-2);
\draw (2.5,-1.5) node {1};
\draw (3,-1) rectangle (4,-2);
\draw (3.5,-1.5) node {2};
\end{tikzpicture}
\space{}\end{center}
There are 10 $h$-brick tabloids, and 10 is the coefficient of $m_{(4,4)}$ 
in the $m$-expansion of $m_{(2,1)}h_{(2,1,2)}$. This illustrates the
result proved next.
\end{example}

\begin{prop}\label{prop:h-brick}
Let $\lambda$, $\mu$ be partitions and $\alpha=(\alpha_1,\ldots,\alpha_s)$ 
be a sequence of positive integers. 
Then the coefficient of $m_\lambda$ in $m_\mu h_\alpha$ 
is the number of $h$-brick tabloids of shape $\lambda$ and extended content 
$(\mu;\alpha)$.
\end{prop}
\begin{proof}
As in the proof of \cref{prop:mp-in-m},
the coefficient of $m_{\lambda}$ in the $m$-expansion of 
$m_{\mu}h_{\alpha}$ equals the coefficient of 
$\xx^{\lambda}$ in $m_{\mu}(\xx)h_{\alpha}(\xx)$.
In turn, this coefficient is the number of ordered factorizations
of $\xx^{\lambda}$ of the form $f_0f_1\cdots f_s$,
where $f_0$ is a monomial in $m_{\mu}(\xx)$ and
$f_j$ is a monomial in $h_{\alpha_j}(\xx)$ for $j=1,2,\ldots,s$.

There is a bijection between the set of such factorizations
and the set of $h$-brick tabloids described in the proposition.
On one hand, given such an $h$-brick tabloid $T$, let $n_{ij}(T)$
be the number of cells in row $i$ of $T$ covered by a brick labeled $j$.
Define $f_j=\prod_{i\geq 1} x_i^{n_{ij}(T)}$ for $j=0,1,2,\ldots,s$.
By the rules for the brick sizes, $f_0$ is one of the monomials in
$m_{\mu}(\xx)$ and $f_j$ is a monomial of degree $\alpha_j$, which
is one of the terms in $h_{\alpha_j}(\xx)$. Since every cell in $\dg(\lambda)$
is covered by exactly one brick, $f_0f_1\cdots f_s=\xx^{\lambda}$ follows.

The inverse bijection acts as follows.
Given an ordered factorization $f_0f_1\cdots f_s$ of
$x^{\lambda}$, make the associated $h$-brick tabloid as follows.
Write $f_j=\prod_{i\geq 1} x_i^{r_{ij}}$ for $j=0,1,2,\ldots,s$.
Since brick labels weakly increase in each row, 
with at most one brick labeled $0$ in each row,
there is exactly one way to cover $\dg(\lambda)$ with bricks
such that the resulting tabloid has $r_{ij}$ cells in row $i$ covered 
by a brick labeled $j$ for all $i,j$.
\end{proof}

By putting $\mu = \varnothing$ and $\alpha=\nu$ (a partition)
in \cref{prop:h-brick}, we can find the 
coefficient of $m_{\lambda}$ in $h_\nu$ using objects of 
\emph{shape} $\lambda$ and \emph{content} $\nu$.
\begin{example}
The coefficient of $m_{(3,2)}$ in $h_{(2,2,1)}$ is 5,
which is the number of $h$-brick tabloids of shape $(3,2)$ and content
\boks{0.4} $(2,2,1)$ shown below.
    \[\begin{array}{ccccc}
             \yt{112,23}  & \yt{113,22} & \yt{223,11} & \yt{122,13} & \yt{123,12}\\[0.5cm]
             x_1^2\cdot  x_1x_2 \cdot x_2 & x_1^2\cdot x_2^2 \cdot x_1  & x_2^2\cdot x_1^2 \cdot x_1  & x_1x_2\cdot x_1^2 \cdot x_2  &x_1x_2\cdot x_1x_2 \cdot x_1 
    \end{array}\]
The ordered factorization under each $h$-brick tabloid 
is computed as in the proof: we have $x_i$ appearing in the $j$th factor 
as many times as the label $j$ appears in row $i$. 
For instance, for the leftmost $h$-brick tabloid, the first factor 
is $x_1^2$ as 1 appears twice in the first row. The second factor 
is $x_1x_2$ as 2 appears in the first and the second row. The third factor 
is $x_2$ because 3 appears once in row 2.
\end{example}

\begin{remark}
It is known that the coefficient of $m_\nu$ in $h_\lambda$ 
and the coefficient of $m_\lambda$ in $h_\nu$ are the same. 
This can be proved by a dual combinatorial construction illustrated in
the next example, where the coefficient of $m_{\lambda}$ in $h_{\nu}$
is found using objects of \emph{shape} $\nu$ and \emph{content} $\lambda$. 
\end{remark}

\begin{example}
    The coefficient of $m_{(3,2)}$ in the expansion of $h_{(2,2,1)}$ is 5,
which is the number of $h$-brick tabloids of shape $(2,2,1)$ and content 
$(3,2)$ shown below.
    \[\begin{array}{ccccc}
             \yt{11,12,2}  & \yt{11,22,1} & \yt{22,11,1} & \yt{12,11,2} & \yt{12,12,1}\\
             x_1^2\cdot x_1x_2 \cdot x_2 & x_1^2\cdot x_2^2 \cdot x_1  & x_2^2\cdot x_1^2 \cdot x_1  & x_1x_2\cdot x_1^2 \cdot x_2  &x_1x_2\cdot x_1x_2 \cdot x_1 
    \end{array}\]
In this case, we convert $h$-brick tabloids to ordered factorizations
as follows. For each brick labeled $i$ in row $j$, 
we include a copy of $x_i$ in the $j$th factor.
\end{example}
As seen in the last two examples, we have two bijections mapping 
$h$-brick tabloids to ordered factorizations. The first bijection forms
the $j$th factor by recording the rows containing the bricks labeled $j$.
The second bijection forms the $j$th factor by recording the brick labels
in row $j$. By composing these maps, we get a bijective proof
that $\mcM(h,m)_{\lambda,\nu}=\mcM(h,m)_{\nu,\lambda}$.

We now extend these results to the polysymmetric case. The objects here
are versions of $h$-brick tabloids for tensor product diagrams.
Let $\tau$ and $\sigma$ be types and $\delta = (d_1^{r_1}, \ldots, d_s^{r_s})$
be a sequence of blocks.
We define an \textit{$H$-tensor brick tabloid (HTBT) of shape $\tau$ and 
extended content $(\sigma;\delta)$} 
as a filling of $\dg(\tau)$ built as follows. 
We first choose partitions $\lambda^{(i)}$ of $d_i$ for $i=1,2,\ldots,s$.
For each $k\geq 1$,
we fill the $k$th component of $\dg(\tau)$ using these rules:
\begin{itemize}
\item Make $m_k(\lambda^{(i)})$ bricks labeled $i$, each of length $r_i$ 
and height $1$. Make $\ell(\sigma|_k)$ bricks labeled $0$, each
of height 1 and with lengths given by the parts of $\sigma|_k$. 
\item Cover $\dg(\tau|_k)$ with these bricks so that labels 
weakly increase in each row, and each row has at most one brick labeled 0.
\end{itemize}

Denote this set of objects by $\HTBT(\tau,(\sigma;\delta))$.
This definition constructs objects similar to $h$-brick tabloids 
but with bricks scaled horizontally according to the multiplicity $r_i$
of the block $d_i^{r_i}$. The degree $d_i$ of the block determines 
the number of such bricks we make. More specifically, 
if the $k$th tensor diagram has $m_{k,i}$ bricks labeled $i$, 
then $\sum_{k\geq 1} km_{k,i}=d_i$ for $i=1,2,\ldots,s$,
where $m_{k,i}=m_k(\lambda^{(i)})$.

\begin{theorem}\label{thm:H-in-m}
Let $\tau$ and $\sigma$ be types and $\delta=(d_1^{r_1}, \ldots, d_s^{r_s})$
be a sequence of blocks. Then the coefficient of $m^\otimes_\tau$ 
in the $m^\otimes$-expansion of $m^\otimes_\sigma H_\delta$ 
is $|\HTBT(\tau,(\sigma;\delta))|$.
\end{theorem}
\begin{proof}
Recall from \cref{prop:H-expansions}(b) that 
$H_{d^r} = \sum\limits_{\lambda\vdash d} 
\prod_{k\geq 1} h_{m_k(\lambda)}(\xx^r_{k*})$, so
\begin{equation}\label{eq:H-in-m-proof}
 m_{\sigma}^{\otimes}H_{\delta}
  =\sum_{\lambda^{(1)}\vdash d_1}\cdots
   \sum_{\lambda^{(s)}\vdash d_s} 
   \prod_{k\geq 1}\left[m_{\sigma|_k}(\xx_{k*})
   \prod_{i=1}^s h_{m_k(\lambda^{(i)})}(\xx_{k*}^{r_i})\right]. 
\end{equation}
The $k$th component of the tensor diagram $\dg(\tau)$ is the partition
$\tau|_k$. We fill this partition with bricks (using the rules above)
to record all possible ways of getting the monomial $\xx_{k*}^{\tau|_k}$
as part of the expression in~\eqref{eq:H-in-m-proof}.
For a given choice of $\lambda^{(1)},\ldots,\lambda^{(s)}$ indexing the 
summands in~\eqref{eq:H-in-m-proof} and for a given $k$, the part
of the expression involving the variables $\xx_{k*}$ is
\[ m_{\sigma|_k}(\xx_{k*}) h_{m_k(\lambda^{(1)})}(\xx^{r_1}_{k*})
 h_{m_k(\lambda^{(2)})}(\xx^{r_2}_{k*}) 
\ldots h_{m_k(\lambda^{(s)})}(\xx^{r_s}_{k*}).\] 
The result then follows as in the proof of \cref{prop:h-brick},
noting that raising the variables $\xx_{k*}$ to the power $r_i$ 
can be modeled by horizontally scaling $1\times 1$ bricks to become
bricks of length $r_i$.
\end{proof}

\begin{example}
Let $\tau = 3^23^22^41^31^31$, $\sigma = 2^21^21$ and $\delta = (8, 3^2,3^2)$. 
We compute the coefficient of $m^\otimes_\tau$ in $m^\otimes_\sigma H_\delta$ 
to be 24 as follows.
\begin{enumerate}
    \item Corresponding to the choice of partitions $(2,2,1,1,1,1)\vdash 8$, 
$(3)\vdash 3$ and $(3)\vdash 3$, we get 8 objects in 
$\HTBT(\tau,(\sigma;\delta))$. We list 4 objects below, 
and the remaining 4 are obtained by swapping 
the \raisebox{-3pt}{\scriptsize \begin{tikzpicture}[scale = 0.3]
\draw (0,0) rectangle (2,-1);
\draw (0.5,-0.5) node {2};
\draw (1.5,-0.5) node {2};
\end{tikzpicture}} and  
\raisebox{-3pt}{\scriptsize \begin{tikzpicture}[scale = 0.3]
\draw (0,0) rectangle (2,-1);
\draw (0.5,-0.5) node {3};
\draw (1.5,-0.5) node {3};
\end{tikzpicture}} in the third component diagram.

    \begin{center}
        \begin{tiny}
		\scalebox{0.9}{\begin{tikzpicture}[scale = 0.4]
\draw (0,0) rectangle (2,-1);
\draw (0.5,-0.5) node {0};
\draw (1.5,-0.5) node {0};
\draw (2,0) rectangle (3,-1);
\draw (2.5,-0.5) node {1};
\draw (0,-1) rectangle (1,-2);
\draw (0.5,-1.5) node {0};
\draw (1,-1) rectangle (2,-2);
\draw (1.5,-1.5) node {1};
\draw (2,-1) rectangle (3,-2);
\draw (2.5,-1.5) node {1};
\draw (0,-2) rectangle (1,-3);
\draw (0.5,-2.5) node {1};
\end{tikzpicture}
\space{}\raisebox{15pt}{$\otimes\:\:$}\raisebox{10pt}{\begin{tikzpicture}[scale = 0.4]
\draw (0,0) rectangle (2,-1);
\draw (0.5,-0.5) node {0};
\draw (1.5,-0.5) node {0};
\draw (2,0) rectangle (3,-1);
\draw (2.5,-0.5) node {1};
\draw (3,0) rectangle (4,-1);
\draw (3.5,-0.5) node {1};
\end{tikzpicture}
} \raisebox{15pt}{$\otimes\:\:$}\raisebox{7pt}{\begin{tikzpicture}[scale = 0.4]
\draw (0,0) rectangle (2,-1);
\draw (0.5,-0.5) node {2};
\draw (1.5,-0.5) node {2};
\draw (0,-1) rectangle (2,-2);
\draw (0.5,-1.5) node {3};
\draw (1.5,-1.5) node {3};
\end{tikzpicture}
}}
		\qquad\scalebox{0.9}{ \begin{tikzpicture}[scale = 0.4]
\draw (0,0) rectangle (2,-1);
\draw (0.5,-0.5) node {0};
\draw (1.5,-0.5) node {0};
\draw (2,0) rectangle (3,-1);
\draw (2.5,-0.5) node {1};
\draw (0,-1) rectangle (1,-2);
\draw (0.5,-1.5) node {1};
\draw (1,-1) rectangle (2,-2);
\draw (1.5,-1.5) node {1};
\draw (2,-1) rectangle (3,-2);
\draw (2.5,-1.5) node {1};
\draw (0,-2) rectangle (1,-3);
\draw (0.5,-2.5) node {0};
\end{tikzpicture}
\space{}\raisebox{15pt}{$\otimes\:\:$}\raisebox{10pt}{\begin{tikzpicture}[scale = 0.4]
\draw (0,0) rectangle (2,-1);
\draw (0.5,-0.5) node {0};
\draw (1.5,-0.5) node {0};
\draw (2,0) rectangle (3,-1);
\draw (2.5,-0.5) node {1};
\draw (3,0) rectangle (4,-1);
\draw (3.5,-0.5) node {1};
\end{tikzpicture}
} \raisebox{15pt}{$\otimes\:\:$}\raisebox{7pt}{\begin{tikzpicture}[scale = 0.4]
\draw (0,0) rectangle (2,-1);
\draw (0.5,-0.5) node {2};
\draw (1.5,-0.5) node {2};
\draw (0,-1) rectangle (2,-2);
\draw (0.5,-1.5) node {3};
\draw (1.5,-1.5) node {3};
\end{tikzpicture}
}}
		\qquad\scalebox{0.9}{ \begin{tikzpicture}[scale = 0.4]
\draw (0,0) rectangle (1,-1);
\draw (0.5,-0.5) node {0};
\draw (1,0) rectangle (2,-1);
\draw (1.5,-0.5) node {1};
\draw (2,0) rectangle (3,-1);
\draw (2.5,-0.5) node {1};
\draw (0,-1) rectangle (2,-2);
\draw (0.5,-1.5) node {0};
\draw (1.5,-1.5) node {0};
\draw (2,-1) rectangle (3,-2);
\draw (2.5,-1.5) node {1};
\draw (0,-2) rectangle (1,-3);
\draw (0.5,-2.5) node {1};
\end{tikzpicture}
\space{}\raisebox{15pt}{$\otimes\:\:$}\raisebox{10pt}{\begin{tikzpicture}[scale = 0.4]
\draw (0,0) rectangle (2,-1);
\draw (0.5,-0.5) node {0};
\draw (1.5,-0.5) node {0};
\draw (2,0) rectangle (3,-1);
\draw (2.5,-0.5) node {1};
\draw (3,0) rectangle (4,-1);
\draw (3.5,-0.5) node {1};
\end{tikzpicture}
} \raisebox{15pt}{$\otimes\:\:$}\raisebox{7pt}{\begin{tikzpicture}[scale = 0.4]
\draw (0,0) rectangle (2,-1);
\draw (0.5,-0.5) node {2};
\draw (1.5,-0.5) node {2};
\draw (0,-1) rectangle (2,-2);
\draw (0.5,-1.5) node {3};
\draw (1.5,-1.5) node {3};
\end{tikzpicture}
}}
		\\[0.1cm]\scalebox{0.9}{ \begin{tikzpicture}[scale = 0.4]
\draw (0,0) rectangle (1,-1);
\draw (0.5,-0.5) node {1};
\draw (1,0) rectangle (2,-1);
\draw (1.5,-0.5) node {1};
\draw (2,0) rectangle (3,-1);
\draw (2.5,-0.5) node {1};
\draw (0,-1) rectangle (2,-2);
\draw (0.5,-1.5) node {0};
\draw (1.5,-1.5) node {0};
\draw (2,-1) rectangle (3,-2);
\draw (2.5,-1.5) node {1};
\draw (0,-2) rectangle (1,-3);
\draw (0.5,-2.5) node {0};
\end{tikzpicture}
\space{}\raisebox{15pt}{$\otimes\:\:$}\raisebox{10pt}{\begin{tikzpicture}[scale = 0.4]
\draw (0,0) rectangle (2,-1);
\draw (0.5,-0.5) node {0};
\draw (1.5,-0.5) node {0};
\draw (2,0) rectangle (3,-1);
\draw (2.5,-0.5) node {1};
\draw (3,0) rectangle (4,-1);
\draw (3.5,-0.5) node {1};
\end{tikzpicture}
} \raisebox{15pt}{$\otimes\:\:$}\raisebox{7pt}{\begin{tikzpicture}[scale = 0.4]
\draw (0,0) rectangle (2,-1);
\draw (0.5,-0.5) node {2};
\draw (1.5,-0.5) node {2};
\draw (0,-1) rectangle (2,-2);
\draw (0.5,-1.5) node {3};
\draw (1.5,-1.5) node {3};
\end{tikzpicture}
}}
        \end{tiny}
    \end{center}

    \item Now, we make a choice of partitions $(3,3,1,1)\vdash 8$, $(2,1)\vdash 3$ and $(3)\vdash 3$ which again gives us 8 objects. We list four objects and the rest can be obtained by swapping \raisebox{-3pt}{\scriptsize\begin{tikzpicture}[scale = 0.3]
\draw (0,0) rectangle (1,-1);
\draw (0.5,-0.5) node {1};
\draw (1,0) rectangle (2,-1);
\draw (1.5,-0.5) node {1};
\end{tikzpicture}
} and \raisebox{-3pt}{\scriptsize\begin{tikzpicture}[scale = 0.3]
\draw (0,0) rectangle (2,-1);
\draw (0.5,-0.5) node {3};
\draw (1.5,-0.5) node {3};
\end{tikzpicture}
} in the third component diagram.
    \begin{center}
    \begin{tiny}
    \scalebox{0.9}{\begin{tikzpicture}[scale = 0.4]
\draw (0,0) rectangle (2,-1);
\draw (0.5,-0.5) node {0};
\draw (1.5,-0.5) node {0};
\draw (2,0) rectangle (3,-1);
\draw (2.5,-0.5) node {1};
\draw (0,-1) rectangle (1,-2);
\draw (0.5,-1.5) node {0};
\draw (1,-1) rectangle (3,-2);
\draw (1.5,-1.5) node {2};
\draw (2.5,-1.5) node {2};
\draw (0,-2) rectangle (1,-3);
\draw (0.5,-2.5) node {1};
\end{tikzpicture}
\space{}\raisebox{15pt}{$\otimes\:\:$}\raisebox{10pt}{\begin{tikzpicture}[scale = 0.4]
\draw (0,0) rectangle (2,-1);
\draw (0.5,-0.5) node {0};
\draw (1.5,-0.5) node {0};
\draw (2,0) rectangle (4,-1);
\draw (2.5,-0.5) node {2};
\draw (3.5,-0.5) node {2};
\end{tikzpicture}
} \raisebox{15pt}{$\otimes\:\:$}\raisebox{7pt}{\begin{tikzpicture}[scale = 0.4]
\draw (0,0) rectangle (2,-1);
\draw (0.5,-0.5) node {3};
\draw (1.5,-0.5) node {3};
\draw (0,-1) rectangle (1,-2);
\draw (0.5,-1.5) node {1};
\draw (1,-1) rectangle (2,-2);
\draw (1.5,-1.5) node {1};
\end{tikzpicture}
}}
		\qquad \scalebox{0.9}{\begin{tikzpicture}[scale = 0.4]
\draw (0,0) rectangle (2,-1);
\draw (0.5,-0.5) node {0};
\draw (1.5,-0.5) node {0};
\draw (2,0) rectangle (3,-1);
\draw (2.5,-0.5) node {1};
\draw (0,-1) rectangle (1,-2);
\draw (0.5,-1.5) node {1};
\draw (1,-1) rectangle (3,-2);
\draw (1.5,-1.5) node {2};
\draw (2.5,-1.5) node {2};
\draw (0,-2) rectangle (1,-3);
\draw (0.5,-2.5) node {0};
\end{tikzpicture}
\space{}\raisebox{15pt}{$\otimes\:\:$}\raisebox{10pt}{\begin{tikzpicture}[scale = 0.4]
\draw (0,0) rectangle (2,-1);
\draw (0.5,-0.5) node {0};
\draw (1.5,-0.5) node {0};
\draw (2,0) rectangle (4,-1);
\draw (2.5,-0.5) node {2};
\draw (3.5,-0.5) node {2};
\end{tikzpicture}
} \raisebox{15pt}{$\otimes\:\:$}\raisebox{7pt}{\begin{tikzpicture}[scale = 0.4]
\draw (0,0) rectangle (2,-1);
\draw (0.5,-0.5) node {3};
\draw (1.5,-0.5) node {3};
\draw (0,-1) rectangle (1,-2);
\draw (0.5,-1.5) node {1};
\draw (1,-1) rectangle (2,-2);
\draw (1.5,-1.5) node {1};
\end{tikzpicture}
}}
		\qquad \scalebox{0.9}{\begin{tikzpicture}[scale = 0.4]
\draw (0,0) rectangle (1,-1);
\draw (0.5,-0.5) node {0};
\draw (1,0) rectangle (3,-1);
\draw (1.5,-0.5) node {2};
\draw (2.5,-0.5) node {2};
\draw (0,-1) rectangle (2,-2);
\draw (0.5,-1.5) node {0};
\draw (1.5,-1.5) node {0};
\draw (2,-1) rectangle (3,-2);
\draw (2.5,-1.5) node {1};
\draw (0,-2) rectangle (1,-3);
\draw (0.5,-2.5) node {1};
\end{tikzpicture}
\space{}\raisebox{15pt}{$\otimes\:\:$}\raisebox{10pt}{\begin{tikzpicture}[scale = 0.4]
\draw (0,0) rectangle (2,-1);
\draw (0.5,-0.5) node {0};
\draw (1.5,-0.5) node {0};
\draw (2,0) rectangle (4,-1);
\draw (2.5,-0.5) node {2};
\draw (3.5,-0.5) node {2};
\end{tikzpicture}
} \raisebox{15pt}{$\otimes\:\:$}\raisebox{7pt}{\begin{tikzpicture}[scale = 0.4]
\draw (0,0) rectangle (2,-1);
\draw (0.5,-0.5) node {3};
\draw (1.5,-0.5) node {3};
\draw (0,-1) rectangle (1,-2);
\draw (0.5,-1.5) node {1};
\draw (1,-1) rectangle (2,-2);
\draw (1.5,-1.5) node {1};
\end{tikzpicture}
}}
		\\[0.1cm] \scalebox{0.9}{\begin{tikzpicture}[scale = 0.4]
\draw (0,0) rectangle (1,-1);
\draw (0.5,-0.5) node {1};
\draw (1,0) rectangle (3,-1);
\draw (1.5,-0.5) node {2};
\draw (2.5,-0.5) node {2};
\draw (0,-1) rectangle (2,-2);
\draw (0.5,-1.5) node {0};
\draw (1.5,-1.5) node {0};
\draw (2,-1) rectangle (3,-2);
\draw (2.5,-1.5) node {1};
\draw (0,-2) rectangle (1,-3);
\draw (0.5,-2.5) node {0};
\end{tikzpicture}
\space{}\raisebox{15pt}{$\otimes\:\:$}\raisebox{10pt}{\begin{tikzpicture}[scale = 0.4]
\draw (0,0) rectangle (2,-1);
\draw (0.5,-0.5) node {0};
\draw (1.5,-0.5) node {0};
\draw (2,0) rectangle (4,-1);
\draw (2.5,-0.5) node {2};
\draw (3.5,-0.5) node {2};
\end{tikzpicture}
} \raisebox{15pt}{$\otimes\:\:$}\raisebox{7pt}{\begin{tikzpicture}[scale = 0.4]
\draw (0,0) rectangle (2,-1);
\draw (0.5,-0.5) node {3};
\draw (1.5,-0.5) node {3};
\draw (0,-1) rectangle (1,-2);
\draw (0.5,-1.5) node {1};
\draw (1,-1) rectangle (2,-2);
\draw (1.5,-1.5) node {1};
\end{tikzpicture}
}}
   
    \end{tiny}
    \end{center}
    \item For the choice of partitions $(3,3,1,1)\vdash 8$, $(3)\vdash 3$ 
and $(2,1)\vdash 3$, we construct 8 objects as in part~(2) 
where the labels 2 and 3 are swapped.
\item It is routine to check that the choices of partitions in (1), (2),
and (3) are the only possibilities leading to collections of bricks
that can fill $\dg(\tau)$ following the rules for HTBTs.
\end{enumerate}
\end{example}

\begin{corollary}
For all types $\tau,\sigma\vdt n$, the coefficient of 
$m^\otimes_\tau$ in the $m^{\otimes}$-expansion of $H_\sigma$ is 
\[
a_{\tau\sigma} = \mcM(H, m^\otimes)_{\tau, \sigma} = |\HTBT(\tau, \sigma)|.
\]
\end{corollary}

\subsection{Rules for $m^{\otimes}_{\sigma}E^+_{\delta}$,
  $m^{\otimes}_{\sigma}E_{\delta}$, $\mcM(E^+,m^{\otimes})$, 
  and $\mcM(E,m^{\otimes})$}
\label{subsec:Em}

In this section, we start by finding the
$m$-expansion of the symmetric polynomial $m_{\mu}e_{\alpha}$. 
We then use similar ideas to obtain the $m^\otimes$-expansions
of the polysymmetric functions $E_{\sigma}$ and $E^+_{\sigma}$.

Define an \emph{$e$-brick tabloid} to be an $h$-brick tabloid with the 
added condition that in each row, all bricks have distinct labels.

\begin{example}
The $e$-brick tabloids of shape $(4,4)$ with extended content $((2,1);(2,1,2))$
are
\begin{center}
\begin{tikzpicture}[scale = 0.4]
\draw (0,0) rectangle (1,-1);
\draw (0.5,-0.5) node {0};
\draw (1,0) rectangle (2,-1);
\draw (1.5,-0.5) node {1};
\draw (2,0) rectangle (3,-1);
\draw (2.5,-0.5) node {2};
\draw (3,0) rectangle (4,-1);
\draw (3.5,-0.5) node {3};
\draw (0,-1) rectangle (2,-2);
\draw (0.5,-1.5) node {0};
\draw (1.5,-1.5) node {0};
\draw (2,-1) rectangle (3,-2);
\draw (2.5,-1.5) node {1};
\draw (3,-1) rectangle (4,-2);
\draw (3.5,-1.5) node {3};
\end{tikzpicture} \quad \begin{tikzpicture}[scale = 0.4]
\draw (0,0) rectangle (2,-1);
\draw (0.5,-0.5) node {0};
\draw (1.5,-0.5) node {0};
\draw (2,0) rectangle (3,-1);
\draw (2.5,-0.5) node {1};
\draw (3,0) rectangle (4,-1);
\draw (3.5,-0.5) node {3};
\draw (0,-1) rectangle (1,-2);
\draw (0.5,-1.5) node {0};
\draw (1,-1) rectangle (2,-2);
\draw (1.5,-1.5) node {1};
\draw (2,-1) rectangle (3,-2);
\draw (2.5,-1.5) node {2};
\draw (3,-1) rectangle (4,-2);
\draw (3.5,-1.5) node {3};
\end{tikzpicture}
\end{center}
which are 2 of the 10 $h$-brick tabloids from \cref{example:h-in-m}.
\end{example}
\begin{prop}
Let $\lambda, \mu$ be partitions and $\alpha=(\alpha_1,\ldots,\alpha_s)$ 
be a sequence of positive integers. 
The coefficient of $m_\lambda$ in $m_\mu e_\alpha$ 
is the number of $e$-brick tabloids of shape $\lambda$ 
and extended content $(\mu;\alpha)$.
\end{prop}
\begin{proof}
We need to find the coefficient of $\xx^\lambda$ in $m_\mu(\xx)e_\alpha(\xx)$. 
This is the number of ordered square-free factorizations
of $\xx^\lambda$, which have the form 
$\xx^{\lambda}=f_0 f_1 \ldots f_s$ where $f_0$ 
is a monomial appearing in $m_\mu(\xx)$ and $f_j$ is a monomial 
appearing in $e_{\alpha_j}(\xx)$ for $j = 1, 2, \ldots, s$.
We proceed similarly to the proof of \cref{prop:h-brick}.  For each $j>0$,
the condition that each row has at most one brick labeled $j$
ensures that $f_j$ is a square-free monomial of degree $\alpha_j$ 
and thus appears in $e_{\alpha_j}(\xx)$. 
The factor $f_0 =\prod\limits_{i\geq 1} x_i^{r_i}$ is recorded in the
brick tabloid by putting a brick of length $r_i$ with label 0 in row $i$. 
For $j\geq 1$, if $f_j = x_{i_1}x_{i_2}\ldots x_{i_{\alpha_j}}$, then we 
put one brick labeled $j$ in each row $i_1$, $i_2$, $\ldots$, $i_{\alpha_j}$. 
This gives us the $e$-brick tabloid recording the given square-free
factorization of $\xx^{\lambda}$.
\end{proof}

To get the analogue of \cref{thm:H-in-m}, 
we define \emph{$E$-tensor brick tabloids} 
(ETBTs) of shape $\tau$ and extended content $(\sigma;\delta)$, 
where $\tau$ and $\sigma$ are types and 
$\delta=(d_1^{r_1}, \ldots, d_k^{r_k})$ is a sequence of blocks. 
To build such an ETBT, say $T$,
first choose partitions $\lambda^{(i)}$ of $d_i$. For $k\geq 1$, 
the $k$th component of $\dg(\tau)$ is filled as follows.
\begin{itemize}
\item Make $m_k(\lambda^{(i)})$ bricks of length $r_i$ and height $1$, 
each with label $i$. Make $\ell(\sigma|_k)$ bricks of height 1 and label 0
with lengths corresponding to the parts of $\sigma|_k$. 
\item Cover $\dg(\tau|_k)$ with these bricks subject to the 
condition that brick labels increase strictly in each row.
\end{itemize}
Define the \emph{sign} of the $E$-tensor brick tabloid thus constructed 
to be $\sgn(T) = \prod\limits_{i=1}^k (-1)^{\ell(\lambda^{(i)})}$. 
Denote the set of such objects by $\ETBT(\tau,(\sigma;\delta))$. 
The power of $-1$ in $\sgn(T)$ is the total number of bricks in $T$
with a positive label.

\begin{theorem}\label{thm:E-in-m}
Let $\tau$ and $\sigma$ be types and $\delta$ be a sequence of blocks.
\\ (a) The coefficient of $m^\otimes_\tau$ in the $m^{\otimes}$-expansion
of $m^\otimes_\sigma E^+_\delta$ is $\sum\limits_{T\in \ETBT(\tau,(\sigma;\delta))} 1 = |\ETBT(\tau,(\sigma;\delta))|$.
\\ (b) The coefficient of $m^\otimes_\tau$ in the $m^{\otimes}$-expansion
of $m^\otimes_\sigma E_\delta$ is $\sum\limits_{T\in \ETBT(\tau,(\sigma;\delta))} \sgn(T)$.
\end{theorem}
\begin{proof}
We adapt the proof of \cref{thm:H-in-m}. 
For~(a), Equation~\eqref{eq:H-in-m-proof} becomes
\begin{equation}\label{eq:E-in-m-proof}
 m_{\sigma}^{\otimes}E^+_{\delta}
  =\sum_{\lambda^{(1)}\vdash d_1}\cdots
   \sum_{\lambda^{(s)}\vdash d_s} 
   \prod_{k\geq 1}\left[m_{\sigma|_k}(\xx_{k*})
   \prod_{i=1}^s e_{m_k(\lambda^{(i)})}(\xx_{k*}^{r_i})\right]. 
\end{equation}
The part of this expression involving the variables $\xx_{k*}$ is
\[ m_{\sigma|_k}(\xx_{k*}) e_{m_k(\lambda^{(1)})}(\xx^{r_1}_{k*})
 e_{m_k(\lambda^{(2)})}(\xx^{r_2}_{k*}) 
\ldots e_{m_k(\lambda^{(s)})}(\xx^{r_s}_{k*}).\] 
Choosing monomials from these factors corresponds to filling $\dg(\tau|_k)$ 
with bricks according to the rules in the definition of ETBTs.
In particular, brick labels strictly increase in each row since
the monomials in $e_m(\xx_{k*})$ are square-free.

Part~(b) is proved similarly, but now the right side
of~\eqref{eq:E-in-m-proof} includes the sign factor
$\prod_{i=1}^s (-1)^{\ell(\lambda^{(i)})}$ for the summand
indexed by $\lambda^{(1)},\ldots,\lambda^{(s)}$. 
This sign equals $\sgn(T)$ for any ETBT $T$ built from this
choice of the partitions $\lambda^{(i)}$.
\end{proof}

\begin{example}
Let $\sigma  = 1^{(2,1)}$, 
$\tau = 2^{2,1,1}1^{5,2,1}$ and $\delta = (5^1,3^2,2^1)$. 
Via the objects below, we find that
the coefficient of $m^\otimes_\tau$ in the $m^{\otimes}$-expansion of 
$m^\otimes_\sigma E^+_\delta$ is $7$, while the coefficient 
of $m^\otimes_\sigma E_\delta$ is $-7$. 
\begin{enumerate}
\item We first choose the partitions $(2,2,1)\vdash 5$, $(2,1)\vdash 3$ 
and $(1,1)\vdash 2$. Then we construct the four ETBTs shown below. 
Note that certain configurations that give valid
HTBTs are not possible in the setting of ETBTs.
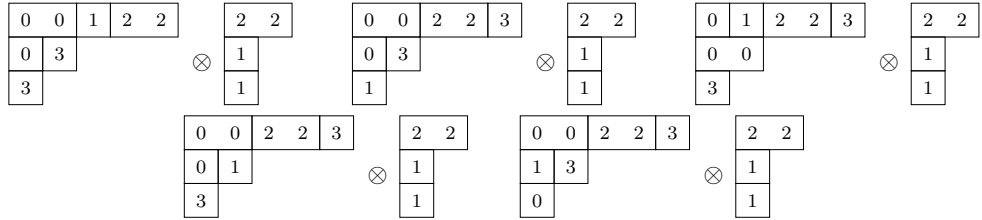
\begin{center}
\begin{scriptsize}
\scalebox{0.9}{\begin{tikzpicture}[scale = 0.5]
\draw (0,0) rectangle (2,-1);
\draw (0.5,-0.5) node {0};
\draw (1.5,-0.5) node {0};
\draw (2,0) rectangle (3,-1);
\draw (2.5,-0.5) node {1};
\draw (3,0) rectangle (5,-1);
\draw (3.5,-0.5) node {2};
\draw (4.5,-0.5) node {2};
\draw (0,-1) rectangle (1,-2);
\draw (0.5,-1.5) node {0};
\draw (1,-1) rectangle (2,-2);
\draw (1.5,-1.5) node {3};
\draw (0,-2) rectangle (1,-3);
\draw (0.5,-2.5) node {3};
\end{tikzpicture}
\space{}\raisebox{15pt}{{\normalsize $\otimes\:\:$}}{\begin{tikzpicture}[scale = 0.5]
\draw (0,0) rectangle (2,-1);
\draw (0.5,-0.5) node {2};
\draw (1.5,-0.5) node {2};
\draw (0,-1) rectangle (1,-2);
\draw (0.5,-1.5) node {1};
\draw (0,-2) rectangle (1,-3);
\draw (0.5,-2.5) node {1};
\end{tikzpicture}
}}
 \qquad \scalebox{0.9}{\begin{tikzpicture}[scale = 0.5]
\draw (0,0) rectangle (2,-1);
\draw (0.5,-0.5) node {0};
\draw (1.5,-0.5) node {0};
\draw (2,0) rectangle (4,-1);
\draw (2.5,-0.5) node {2};
\draw (3.5,-0.5) node {2};
\draw (4,0) rectangle (5,-1);
\draw (4.5,-0.5) node {3};
\draw (0,-1) rectangle (1,-2);
\draw (0.5,-1.5) node {0};
\draw (1,-1) rectangle (2,-2);
\draw (1.5,-1.5) node {3};
\draw (0,-2) rectangle (1,-3);
\draw (0.5,-2.5) node {1};
\end{tikzpicture}
\space{}\raisebox{15pt}{{\normalsize $\otimes\:\:$}}{\begin{tikzpicture}[scale = 0.5]
\draw (0,0) rectangle (2,-1);
\draw (0.5,-0.5) node {2};
\draw (1.5,-0.5) node {2};
\draw (0,-1) rectangle (1,-2);
\draw (0.5,-1.5) node {1};
\draw (0,-2) rectangle (1,-3);
\draw (0.5,-2.5) node {1};
\end{tikzpicture}
}}
 \qquad \scalebox{0.9}{\begin{tikzpicture}[scale = 0.5]
\draw (0,0) rectangle (1,-1);
\draw (0.5,-0.5) node {0};
\draw (1,0) rectangle (2,-1);
\draw (1.5,-0.5) node {1};
\draw (2,0) rectangle (4,-1);
\draw (2.5,-0.5) node {2};
\draw (3.5,-0.5) node {2};
\draw (4,0) rectangle (5,-1);
\draw (4.5,-0.5) node {3};
\draw (0,-1) rectangle (2,-2);
\draw (0.5,-1.5) node {0};
\draw (1.5,-1.5) node {0};
\draw (0,-2) rectangle (1,-3);
\draw (0.5,-2.5) node {3};
\end{tikzpicture}
\space{}\raisebox{15pt}{{\normalsize $\otimes\:\:$}}{\begin{tikzpicture}[scale = 0.5]
\draw (0,0) rectangle (2,-1);
\draw (0.5,-0.5) node {2};
\draw (1.5,-0.5) node {2};
\draw (0,-1) rectangle (1,-2);
\draw (0.5,-1.5) node {1};
\draw (0,-2) rectangle (1,-3);
\draw (0.5,-2.5) node {1};
\end{tikzpicture}
}} \\[0.1cm] \scalebox{0.9}{\begin{tikzpicture}[scale = 0.5]
\draw (0,0) rectangle (2,-1);
\draw (0.5,-0.5) node {0};
\draw (1.5,-0.5) node {0};
\draw (2,0) rectangle (4,-1);
\draw (2.5,-0.5) node {2};
\draw (3.5,-0.5) node {2};
\draw (4,0) rectangle (5,-1);
\draw (4.5,-0.5) node {3};
\draw (0,-1) rectangle (1,-2);
\draw (0.5,-1.5) node {0};
\draw (1,-1) rectangle (2,-2);
\draw (1.5,-1.5) node {1};
\draw (0,-2) rectangle (1,-3);
\draw (0.5,-2.5) node {3};
\end{tikzpicture}
\space{}\raisebox{15pt}{{\normalsize $\otimes\:\:$}}{\begin{tikzpicture}[scale = 0.5]
\draw (0,0) rectangle (2,-1);
\draw (0.5,-0.5) node {2};
\draw (1.5,-0.5) node {2};
\draw (0,-1) rectangle (1,-2);
\draw (0.5,-1.5) node {1};
\draw (0,-2) rectangle (1,-3);
\draw (0.5,-2.5) node {1};
\end{tikzpicture}
}}\qquad
\scalebox{0.9}{\begin{tikzpicture}[scale = 0.5]
\draw (0,0) rectangle (2,-1);
\draw (0.5,-0.5) node {0};
\draw (1.5,-0.5) node {0};
\draw (2,0) rectangle (4,-1);
\draw (2.5,-0.5) node {2};
\draw (3.5,-0.5) node {2};
\draw (4,0) rectangle (5,-1);
\draw (4.5,-0.5) node {3};
\draw (0,-1) rectangle (1,-2);
\draw (0.5,-1.5) node {1};
\draw (1,-1) rectangle (2,-2);
\draw (1.5,-1.5) node {3};
\draw (0,-2) rectangle (1,-3);
\draw (0.5,-2.5) node {0};
\end{tikzpicture}
\space{}\raisebox{15pt}{{\normalsize $\otimes\:\:$}}{\begin{tikzpicture}[scale = 0.5]
\draw (0,0) rectangle (2,-1);
\draw (0.5,-0.5) node {2};
\draw (1.5,-0.5) node {2};
\draw (0,-1) rectangle (1,-2);
\draw (0.5,-1.5) node {1};
\draw (0,-2) rectangle (1,-3);
\draw (0.5,-2.5) node {1};
\end{tikzpicture}
}}
\end{scriptsize}
\end{center}
All these ETBTs have the same sign, namely $(-1)^{3+2+2} = -1$.
\item We now choose a different set of partitions $(2,1,1,1)\vdash 5$, $(2,1)\vdash 3$ and $(2)\vdash 2$. This gives us the two ETBTs shown below.
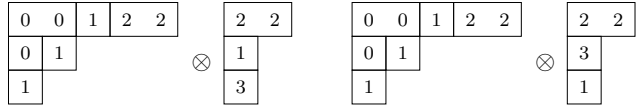
\begin{center}
\begin{scriptsize}
\scalebox{0.9}{\begin{tikzpicture}[scale = 0.5]
\draw (0,0) rectangle (2,-1);
\draw (0.5,-0.5) node {0};
\draw (1.5,-0.5) node {0};
\draw (2,0) rectangle (3,-1);
\draw (2.5,-0.5) node {1};
\draw (3,0) rectangle (5,-1);
\draw (3.5,-0.5) node {2};
\draw (4.5,-0.5) node {2};
\draw (0,-1) rectangle (1,-2);
\draw (0.5,-1.5) node {0};
\draw (1,-1) rectangle (2,-2);
\draw (1.5,-1.5) node {1};
\draw (0,-2) rectangle (1,-3);
\draw (0.5,-2.5) node {1};
\end{tikzpicture}
\space{}\raisebox{15pt}{{\normalsize $\otimes\:\:$}}{\begin{tikzpicture}[scale = 0.5]
\draw (0,0) rectangle (2,-1);
\draw (0.5,-0.5) node {2};
\draw (1.5,-0.5) node {2};
\draw (0,-1) rectangle (1,-2);
\draw (0.5,-1.5) node {1};
\draw (0,-2) rectangle (1,-3);
\draw (0.5,-2.5) node {3};
\end{tikzpicture}
}}
 \qquad \scalebox{0.9}{\begin{tikzpicture}[scale = 0.5]
\draw (0,0) rectangle (2,-1);
\draw (0.5,-0.5) node {0};
\draw (1.5,-0.5) node {0};
\draw (2,0) rectangle (3,-1);
\draw (2.5,-0.5) node {1};
\draw (3,0) rectangle (5,-1);
\draw (3.5,-0.5) node {2};
\draw (4.5,-0.5) node {2};
\draw (0,-1) rectangle (1,-2);
\draw (0.5,-1.5) node {0};
\draw (1,-1) rectangle (2,-2);
\draw (1.5,-1.5) node {1};
\draw (0,-2) rectangle (1,-3);
\draw (0.5,-2.5) node {1};
\end{tikzpicture}
\space{}\raisebox{15pt}{{\normalsize $\otimes\:\:$}}{\begin{tikzpicture}[scale = 0.5]
\draw (0,0) rectangle (2,-1);
\draw (0.5,-0.5) node {2};
\draw (1.5,-0.5) node {2};
\draw (0,-1) rectangle (1,-2);
\draw (0.5,-1.5) node {3};
\draw (0,-2) rectangle (1,-3);
\draw (0.5,-2.5) node {1};
\end{tikzpicture}
}}
\end{scriptsize}
\end{center}
all with the sign $(-1)^{4 + 2 + 1} = -1$.
\item It is routine to check that no other choices of partitions
lead to brick collections that can fill $\dg(\tau)$ following the 
rules for ETBTs.
\end{enumerate}
\end{example}

\begin{corollary}
 (a) For all $\tau,\sigma\vdt n$, the coefficient of 
$m^\otimes_\tau$ in the $m^{\otimes}$-expansion of $E^+_\sigma$ is 
$$\mcM(E^+, m^\otimes)_{\tau, \sigma} = |\ETBT(\tau, \sigma)|.$$
\\ (b) For all $\tau,\sigma\vdt n$, the coefficient of
$m^\otimes_\tau$ in the $m^{\otimes}$-expansion of $E_\sigma$ is 
$$\mcM(E, m^\otimes)_{\tau, \sigma} = 
\sum\limits_{T\in \ETBT(\tau,(\sigma;\delta))} \sgn(T).$$ 
\end{corollary}

\section{Appendix: Sample Transition Matrices}
\label{sec:data}

Below we give the transition matrices computed in this paper for bases
of $\PLambda^4$. For example, the column marked $1^{22}$ in
$\mcM(P,s^{\otimes})$ tells us that
\[ P_{1^{22}}=1s^{\otimes}_{1^4}-1s^{\otimes}_{1^{31}}
   +2s^{\otimes}_{1^{22}}-1s^{\otimes}_{1^{211}}
   +1s^{\otimes}_{1^{1111}}. \]

\begin{footnotesize} 
{\let\quad\thinspace 
\begin{align*}
&\mcM(P,s^\otimes) &\mcM(H,s^\otimes)\\
&\bbmatrix{
~ & 1^{4} & 1^{31} & 1^{22} & 1^{211} & 1^{1111} & 2^{1}1^{2} & 2^{1}1^{11} & 3^{1}1^{1} & 2^{2} & 2^{11} & 4^{1} \cr
1^{4} & 1 & 1 & 1 & 1 & 1 & 1 & 1 & 1 & 1 & 1 & 1 \cr
1^{31} & -1 & 0 & -1 & 1 & 3 & -1 & 1 & 0 & -1 & -1 & -1 \cr
1^{22} & 0 & -1 & 2 & 0 & 2 & 2 & 0 & -1 & 0 & 2 & 0 \cr
1^{211} & 1 & 0 & -1 & -1 & 3 & -1 & -1 & 0 & 1 & -1 & 1 \cr
1^{1111} & -1 & 1 & 1 & -1 & 1 & 1 & -1 & 1 & -1 & 1 & -1 \cr
2^{1}1^{2} & 0 & 0 & 0 & 0 & 0 & 2 & 2 & 0 & 0 & 4 & 0 \cr
2^{1}1^{11} & 0 & 0 & 0 & 0 & 0 & -2 & 2 & 0 & 0 & -4 & 0 \cr
3^{1}1^{1} & 0 & 0 & 0 & 0 & 0 & 0 & 0 & 3 & 0 & 0 & 0 \cr
2^{2} & 0 & 0 & 0 & 0 & 0 & 0 & 0 & 0 & 2 & 4 & 2 \cr
2^{11} & 0 & 0 & 0 & 0 & 0 & 0 & 0 & 0 & -2 & 4 & -2 \cr
4^{1} & 0 & 0 & 0 & 0 & 0 & 0 & 0 & 0 & 0 & 0 & 4 \cr
}
 &\qquad
\bbmatrix{ %
~ & 1^4 & 1^{31} & 1^{22} & 1^{211} & 1^{1111}
 & 2^11^2 & 2^11^{11} & 3^11^1 & 2^2 & 2^{11} & 4^1 \cr
1^4 & 1 & 1 & 1 & 1 & 1 & 1 & 1 & 1 & 1 & 1 & 1 \cr
1^{31} & -1 & 0 & -1 & 1 & 3 & 0 & 2 & 1 & -1 & 1 & 0 \cr
1^{22} & 0 & -1 & 2 & 0 & 2 & 1 & 1 & 0 & 1 & 1 & 0 \cr
1^{211} & 1 & 0 & -1 & -1 & 3 & -1 & 1 & 0 & 0 & 0 & 0 \cr
1^{1111} & -1 & 1 & 1 & -1 & 1 & 0 & 0 & 0 & 0 & 0 & 0 \cr
2^11^2 & 0 & 0 & 0 & 0 & 0 & 1 & 1 & 1 & 0 & 2 & 1 \cr
2^11^{11} & 0 & 0 & 0 & 0 & 0 & -1 & 1 & 1 & 0 & 0 & 0 \cr
3^11^1 & 0 & 0 & 0 & 0 & 0 & 0 & 0 & 1 & 0 & 0 & 1 \cr
2^2 & 0 & 0 & 0 & 0 & 0 & 0 & 0 & 0 & 1 & 1 & 1 \cr
2^{11} & 0 & 0 & 0 & 0 & 0 & 0 & 0 & 0 & -1 & 1 & 0 \cr
4^1 & 0 & 0 & 0 & 0 & 0 & 0 & 0 & 0 & 0 & 0 & 1 \cr
} \end{align*} }
\end{footnotesize}

\begin{footnotesize} 
{\let\quad\thinspace 
\begin{align*}
&\mcM(E^+,s^\otimes) &\mcM(E,s^\otimes)\\
&\bbmatrix{
~ & 1^{4} & 1^{31} & 1^{22} & 1^{211} & 1^{1111} & 2^{1}1^{2} & 2^{1}1^{11} & 3^{1}1^{1} & 2^{2} & 2^{11} & 4^{1} \cr
1^{4} & 1 & 1 & 1 & 1 & 1 & 0 & 0 & 0 & 0 & 0 & 0 \cr
1^{31} & -1 & 0 & -1 & 1 & 3 & 1 & 1 & 0 & 0 & 0 & 0 \cr
1^{22} & 0 & -1 & 2 & 0 & 2 & -1 & 1 & 0 & 1 & 1 & 0 \cr
1^{211} & 1 & 0 & -1 & -1 & 3 & 0 & 2 & 1 & -1 & 1 & 0 \cr
1^{1111} & -1 & 1 & 1 & -1 & 1 & -1 & 1 & 1 & 1 & 1 & 1 \cr
2^{1}1^{2} & 0 & 0 & 0 & 0 & 0 & 1 & 1 & 1 & 0 & 0 & 0 \cr
2^{1}1^{11} & 0 & 0 & 0 & 0 & 0 & -1 & 1 & 1 & 0 & 2 & 1 \cr
3^{1}1^{1} & 0 & 0 & 0 & 0 & 0 & 0 & 0 & 1 & 0 & 0 & 1 \cr
2^{2} & 0 & 0 & 0 & 0 & 0 & 0 & 0 & 0 & 1 & 1 & 0 \cr
2^{11} & 0 & 0 & 0 & 0 & 0 & 0 & 0 & 0 & -1 & 1 & 1 \cr
4^{1} & 0 & 0 & 0 & 0 & 0 & 0 & 0 & 0 & 0 & 0 & 1 \cr
}
 &\qquad
\bbmatrix{
~ & 1^{4} & 1^{31} & 1^{22} & 1^{211} & 1^{1111} & 2^{1}1^{2} & 2^{1}1^{11} & 3^{1}1^{1} & 2^{2} & 2^{11} & 4^{1} \cr
1^{4} & -1 & 1 & 1 & -1 & 1 & 0 & 0 & 0 & 0 & 0 & 0 \cr
1^{31} & 1 & 0 & -1 & -1 & 3 & -1 & 1 & 0 & 0 & 0 & 0 \cr
1^{22} & 0 & -1 & 2 & 0 & 2 & 1 & 1 & 0 & 1 & 1 & 0 \cr
1^{211} & -1 & 0 & -1 & 1 & 3 & 0 & 2 & 1 & -1 & 1 & 0 \cr
1^{1111} & 1 & 1 & 1 & 1 & 1 & 1 & 1 & 1 & 1 & 1 & 1 \cr
2^{1}1^{2} & 0 & 0 & 0 & 0 & 0 & 1 & -1 & -1 & 0 & 0 & 0 \cr
2^{1}1^{11} & 0 & 0 & 0 & 0 & 0 & -1 & -1 & -1 & 0 & -2 & -1 \cr
3^{1}1^{1} & 0 & 0 & 0 & 0 & 0 & 0 & 0 & 1 & 0 & 0 & 1 \cr
2^{2} & 0 & 0 & 0 & 0 & 0 & 0 & 0 & 0 & -1 & 1 & 0 \cr
2^{11} & 0 & 0 & 0 & 0 & 0 & 0 & 0 & 0 & 1 & 1 & 1 \cr
4^{1} & 0 & 0 & 0 & 0 & 0 & 0 & 0 & 0 & 0 & 0 & -1 \cr
}
 \end{align*} }
\end{footnotesize}

\begin{footnotesize} 
{\let\quad\thinspace 
\begin{align*}
&\mcM(P,p^\otimes) &\mcM(H,p^\otimes)\\
&\bbmatrix{
~ & 1^{4} & 1^{31} & 1^{22} & 1^{211} & 1^{1111} & 2^{1}1^{2} & 2^{1}1^{11} & 3^{1}1^{1} & 2^{2} & 2^{11} & 4^{1} \cr
1^{4} & 1 & 0 & 0 & 0 & 0 & 0 & 0 & 0 & 1 & 0 & 1 \cr
1^{31} & 0 & 1 & 0 & 0 & 0 & 0 & 0 & 1 & 0 & 0 & 0 \cr
1^{22} & 0 & 0 & 1 & 0 & 0 & 1 & 0 & 0 & 0 & 1 & 0 \cr
1^{211} & 0 & 0 & 0 & 1 & 0 & 0 & 1 & 0 & 0 & 0 & 0 \cr
1^{1111} & 0 & 0 & 0 & 0 & 1 & 0 & 0 & 0 & 0 & 0 & 0 \cr
2^{1}1^{2} & 0 & 0 & 0 & 0 & 0 & 2 & 0 & 0 & 0 & 4 & 0 \cr
2^{1}1^{11} & 0 & 0 & 0 & 0 & 0 & 0 & 2 & 0 & 0 & 0 & 0 \cr
3^{1}1^{1} & 0 & 0 & 0 & 0 & 0 & 0 & 0 & 3 & 0 & 0 & 0 \cr
2^{2} & 0 & 0 & 0 & 0 & 0 & 0 & 0 & 0 & 2 & 0 & 2 \cr
2^{11} & 0 & 0 & 0 & 0 & 0 & 0 & 0 & 0 & 0 & 4 & 0 \cr
4^{1} & 0 & 0 & 0 & 0 & 0 & 0 & 0 & 0 & 0 & 0 & 4 \cr
}
&\qquad
\bbmatrix{
~ & 1^{4} & 1^{31} & 1^{22} & 1^{211} & 1^{1111} & 2^{1}1^{2} & 2^{1}1^{11} & 3^{1}1^{1} & 2^{2} & 2^{11} & 4^{1} \cr
1^{4} & 1 & 0 & 0 & 0 & 0 & 0 & 0 & 0 & \frac{1}{2} & 0 & \frac{1}{4} \cr
1^{31} & 0 & 1 & 0 & 0 & 0 & 0 & 0 & \frac{1}{3} & 0 & 0 & \frac{1}{3} \cr
1^{22} & 0 & 0 & 1 & 0 & 0 & \frac{1}{2} & 0 & 0 & \frac{1}{2} & \frac{1}{4} & \frac{1}{8} \cr
1^{211} & 0 & 0 & 0 & 1 & 0 & \frac{1}{2} & \frac{1}{2} & \frac{1}{2} & 0 & \frac{1}{2} & \frac{1}{4} \cr
1^{1111} & 0 & 0 & 0 & 0 & 1 & 0 & \frac{1}{2} & \frac{1}{6} & 0 & \frac{1}{4} & \frac{1}{24} \cr
2^{1}1^{2} & 0 & 0 & 0 & 0 & 0 & 1 & 0 & 0 & 0 & 1 & \frac{1}{2} \cr
2^{1}1^{11} & 0 & 0 & 0 & 0 & 0 & 0 & 1 & 1 & 0 & 1 & \frac{1}{2} \cr
3^{1}1^{1} & 0 & 0 & 0 & 0 & 0 & 0 & 0 & 1 & 0 & 0 & 1 \cr
2^{2} & 0 & 0 & 0 & 0 & 0 & 0 & 0 & 0 & 1 & 0 & \frac{1}{2} \cr
2^{11} & 0 & 0 & 0 & 0 & 0 & 0 & 0 & 0 & 0 & 1 & \frac{1}{2} \cr
4^{1} & 0 & 0 & 0 & 0 & 0 & 0 & 0 & 0 & 0 & 0 & 1 \cr
}
\end{align*} }
\end{footnotesize}

\begin{footnotesize} 
{\let\quad\thinspace 
\begin{align*}
&\mcM(E^+,p^\otimes) &\mcM(E,p^\otimes)\\
&\bbmatrix{
~ & 1^{4} & 1^{31} & 1^{22} & 1^{211} & 1^{1111} & 2^{1}1^{2} & 2^{1}1^{11} & 3^{1}1^{1} & 2^{2} & 2^{11} & 4^{1} \cr
1^{4} & 1 & 0 & 0 & 0 & 0 & 0 & 0 & 0 & -\frac{1}{2} & 0 & -\frac{1}{4} \cr
1^{31} & 0 & 1 & 0 & 0 & 0 & 0 & 0 & \frac{1}{3} & 0 & 0 & \frac{1}{3} \cr
1^{22} & 0 & 0 & 1 & 0 & 0 & -\frac{1}{2} & 0 & 0 & \frac{1}{2} & \frac{1}{4} & \frac{1}{8} \cr
1^{211} & 0 & 0 & 0 & 1 & 0 & \frac{1}{2} & -\frac{1}{2} & -\frac{1}{2} & 0 & -\frac{1}{2} & -\frac{1}{4} \cr
1^{1111} & 0 & 0 & 0 & 0 & 1 & 0 & \frac{1}{2} & \frac{1}{6} & 0 & \frac{1}{4} & \frac{1}{24} \cr
2^{1}1^{2} & 0 & 0 & 0 & 0 & 0 & 1 & 0 & 0 & 0 & -1 & -\frac{1}{2} \cr
2^{1}1^{11} & 0 & 0 & 0 & 0 & 0 & 0 & 1 & 1 & 0 & 1 & \frac{1}{2} \cr
3^{1}1^{1} & 0 & 0 & 0 & 0 & 0 & 0 & 0 & 1 & 0 & 0 & 1 \cr
2^{2} & 0 & 0 & 0 & 0 & 0 & 0 & 0 & 0 & 1 & 0 & -\frac{1}{2} \cr
2^{11} & 0 & 0 & 0 & 0 & 0 & 0 & 0 & 0 & 0 & 1 & \frac{1}{2} \cr
4^{1} & 0 & 0 & 0 & 0 & 0 & 0 & 0 & 0 & 0 & 0 & 1 \cr
}
&\qquad
\bbmatrix{
~ & 1^{4} & 1^{31} & 1^{22} & 1^{211} & 1^{1111} & 2^{1}1^{2} & 2^{1}1^{11} & 3^{1}1^{1} & 2^{2} & 2^{11} & 4^{1} \cr
1^{4} & -1 & 0 & 0 & 0 & 0 & 0 & 0 & 0 & -\frac{1}{2} & 0 & -\frac{1}{4} \cr
1^{31} & 0 & 1 & 0 & 0 & 0 & 0 & 0 & \frac{1}{3} & 0 & 0 & \frac{1}{3} \cr
1^{22} & 0 & 0 & 1 & 0 & 0 & \frac{1}{2} & 0 & 0 & \frac{1}{2} & \frac{1}{4} & \frac{1}{8} \cr
1^{211} & 0 & 0 & 0 & -1 & 0 & -\frac{1}{2} & -\frac{1}{2} & -\frac{1}{2} & 0 & -\frac{1}{2} & -\frac{1}{4} \cr
1^{1111} & 0 & 0 & 0 & 0 & 1 & 0 & \frac{1}{2} & \frac{1}{6} & 0 & \frac{1}{4} & \frac{1}{24} \cr
2^{1}1^{2} & 0 & 0 & 0 & 0 & 0 & 1 & 0 & 0 & 0 & 1 & \frac{1}{2} \cr
2^{1}1^{11} & 0 & 0 & 0 & 0 & 0 & 0 & -1 & -1 & 0 & -1 & -\frac{1}{2} \cr
3^{1}1^{1} & 0 & 0 & 0 & 0 & 0 & 0 & 0 & 1 & 0 & 0 & 1 \cr
2^{2} & 0 & 0 & 0 & 0 & 0 & 0 & 0 & 0 & -1 & 0 & -\frac{1}{2} \cr
2^{11} & 0 & 0 & 0 & 0 & 0 & 0 & 0 & 0 & 0 & 1 & \frac{1}{2} \cr
4^{1} & 0 & 0 & 0 & 0 & 0 & 0 & 0 & 0 & 0 & 0 & -1 \cr
}
\end{align*} }
\end{footnotesize}

\begin{footnotesize} 
{\let\quad\thinspace 
\begin{align*}
&\mcM(P,m^\otimes) &\mcM(H,m^\otimes)\\
&\bbmatrix{
~ & 1^{4} & 1^{31} & 1^{22} & 1^{211} & 1^{1111} & 2^{1}1^{2} & 2^{1}1^{11} & 3^{1}1^{1} & 2^{2} & 2^{11} & 4^{1} \cr
1^{4} & 1 & 1 & 1 & 1 & 1 & 1 & 1 & 1 & 1 & 1 & 1 \cr
1^{31} & 0 & 1 & 0 & 2 & 4 & 0 & 2 & 1 & 0 & 0 & 0 \cr
1^{22} & 0 & 0 & 2 & 2 & 6 & 2 & 2 & 0 & 0 & 2 & 0 \cr
1^{211} & 0 & 0 & 0 & 2 & 12 & 0 & 2 & 0 & 0 & 0 & 0 \cr
1^{1111} & 0 & 0 & 0 & 0 & 24 & 0 & 0 & 0 & 0 & 0 & 0 \cr
2^{1}1^{2} & 0 & 0 & 0 & 0 & 0 & 2 & 2 & 0 & 0 & 4 & 0 \cr
2^{1}1^{11} & 0 & 0 & 0 & 0 & 0 & 0 & 4 & 0 & 0 & 0 & 0 \cr
3^{1}1^{1} & 0 & 0 & 0 & 0 & 0 & 0 & 0 & 3 & 0 & 0 & 0 \cr
2^{2} & 0 & 0 & 0 & 0 & 0 & 0 & 0 & 0 & 2 & 4 & 2 \cr
2^{11} & 0 & 0 & 0 & 0 & 0 & 0 & 0 & 0 & 0 & 8 & 0 \cr
4^{1} & 0 & 0 & 0 & 0 & 0 & 0 & 0 & 0 & 0 & 0 & 4 \cr
}
&\qquad
\bbmatrix{
~ & 1^{4} & 1^{31} & 1^{22} & 1^{211} & 1^{1111} & 2^{1}1^{2} & 2^{1}1^{11} & 3^{1}1^{1} & 2^{2} & 2^{11} & 4^{1} \cr
1^{4} & 1 & 1 & 1 & 1 & 1 & 1 & 1 & 1 & 1 & 1 & 1 \cr
1^{31} & 0 & 1 & 0 & 2 & 4 & 1 & 3 & 2 & 0 & 2 & 1 \cr
1^{22} & 0 & 0 & 2 & 2 & 6 & 2 & 4 & 2 & 1 & 3 & 1 \cr
1^{211} & 0 & 0 & 0 & 2 & 12 & 1 & 7 & 3 & 0 & 4 & 1 \cr
1^{1111} & 0 & 0 & 0 & 0 & 24 & 0 & 12 & 4 & 0 & 6 & 1 \cr
2^{1}1^{2} & 0 & 0 & 0 & 0 & 0 & 1 & 1 & 1 & 0 & 2 & 1 \cr
2^{1}1^{11} & 0 & 0 & 0 & 0 & 0 & 0 & 2 & 2 & 0 & 2 & 1 \cr
3^{1}1^{1} & 0 & 0 & 0 & 0 & 0 & 0 & 0 & 1 & 0 & 0 & 1 \cr
2^{2} & 0 & 0 & 0 & 0 & 0 & 0 & 0 & 0 & 1 & 1 & 1 \cr
2^{11} & 0 & 0 & 0 & 0 & 0 & 0 & 0 & 0 & 0 & 2 & 1 \cr
4^{1} & 0 & 0 & 0 & 0 & 0 & 0 & 0 & 0 & 0 & 0 & 1 \cr
}
\end{align*} }
\end{footnotesize}

\begin{footnotesize} 
{\let\quad\thinspace 
\begin{align*}
&\mcM(E^+,m^\otimes) &\mcM(E,m^\otimes)\\
&\bbmatrix{
~ & 1^{4} & 1^{31} & 1^{22} & 1^{211} & 1^{1111} & 2^{1}1^{2} & 2^{1}1^{11} & 3^{1}1^{1} & 2^{2} & 2^{11} & 4^{1} \cr
1^{4} & 1 & 1 & 1 & 1 & 1 & 0 & 0 & 0 & 0 & 0 & 0 \cr
1^{31} & 0 & 1 & 0 & 2 & 4 & 1 & 1 & 0 & 0 & 0 & 0 \cr
1^{22} & 0 & 0 & 2 & 2 & 6 & 0 & 2 & 0 & 1 & 1 & 0 \cr
1^{211} & 0 & 0 & 0 & 2 & 12 & 1 & 5 & 1 & 0 & 2 & 0 \cr
1^{1111} & 0 & 0 & 0 & 0 & 24 & 0 & 12 & 4 & 0 & 6 & 1 \cr
2^{1}1^{2} & 0 & 0 & 0 & 0 & 0 & 1 & 1 & 1 & 0 & 0 & 0 \cr
2^{1}1^{11} & 0 & 0 & 0 & 0 & 0 & 0 & 2 & 2 & 0 & 2 & 1 \cr
3^{1}1^{1} & 0 & 0 & 0 & 0 & 0 & 0 & 0 & 1 & 0 & 0 & 1 \cr
2^{2} & 0 & 0 & 0 & 0 & 0 & 0 & 0 & 0 & 1 & 1 & 0 \cr
2^{11} & 0 & 0 & 0 & 0 & 0 & 0 & 0 & 0 & 0 & 2 & 1 \cr
4^{1} & 0 & 0 & 0 & 0 & 0 & 0 & 0 & 0 & 0 & 0 & 1 \cr
}
&\qquad
\bbmatrix{
~ & 1^{4} & 1^{31} & 1^{22} & 1^{211} & 1^{1111} & 2^{1}1^{2} & 2^{1}1^{11} & 3^{1}1^{1} & 2^{2} & 2^{11} & 4^{1} \cr
1^{4} & -1 & 1 & 1 & -1 & 1 & 0 & 0 & 0 & 0 & 0 & 0 \cr
1^{31} & 0 & 1 & 0 & -2 & 4 & -1 & 1 & 0 & 0 & 0 & 0 \cr
1^{22} & 0 & 0 & 2 & -2 & 6 & 0 & 2 & 0 & 1 & 1 & 0 \cr
1^{211} & 0 & 0 & 0 & -2 & 12 & -1 & 5 & 1 & 0 & 2 & 0 \cr
1^{1111} & 0 & 0 & 0 & 0 & 24 & 0 & 12 & 4 & 0 & 6 & 1 \cr
2^{1}1^{2} & 0 & 0 & 0 & 0 & 0 & 1 & -1 & -1 & 0 & 0 & 0 \cr
2^{1}1^{11} & 0 & 0 & 0 & 0 & 0 & 0 & -2 & -2 & 0 & -2 & -1 \cr
3^{1}1^{1} & 0 & 0 & 0 & 0 & 0 & 0 & 0 & 1 & 0 & 0 & 1 \cr
2^{2} & 0 & 0 & 0 & 0 & 0 & 0 & 0 & 0 & -1 & 1 & 0 \cr
2^{11} & 0 & 0 & 0 & 0 & 0 & 0 & 0 & 0 & 0 & 2 & 1 \cr
4^{1} & 0 & 0 & 0 & 0 & 0 & 0 & 0 & 0 & 0 & 0 & -1 \cr
}
\end{align*} }
\end{footnotesize}

 

\begin{thebibliography}{20} 
\bibitem{remmel-trans} Desiree Beck, Jeffrey Remmel, and Tamsen Whitehead,
 ``The combinatorics of transition matrices between the bases of symmetric
 functions and the $B_n$ analogues,'' \emph{Proceedings of FPSAC,
 Discrete Math.} \textbf{153} (1996), 3--27.

\bibitem{DLT}
J. Desarmenien, B. Leclerc, and J. Y. Thibon, 
``Hall--Littlewood functions and Kostka--Foulkes
polynomials in representation theory,''
\emph{Sem. Lothar. Combin,} \textbf{32} (1994), paper B32c.

\bibitem{polysymm}
Asvin G and Andrew O'Desky, ``Configuration spaces, graded spaces,
 and polysymmetric functions,'' preprint available online at
 \texttt{arxiv.org} (arXiv:2207.04529v1), July 2022.

\bibitem{loehr-comb} 
Nicholas Loehr, \emph{Combinatorics} (second edition), CRC Press (2017).

\bibitem{pleth-expose}
Nicholas Loehr and Jeffrey Remmel,
``A computational and combinatorial expos\'e of plethystic calculus,''
\emph{J. Algebraic Combin.} \textbf{33} (2011), 163--198.

\bibitem{macd} Ian Macdonald, \emph{Symmetric Functions and Hall Polynomials}
(second ed.), Oxford University Press (1995).

\bibitem{eg-rem} \"Omer E\u{g}ecio\u{g}lu and Jeffrey B. Remmel,
``Brick tabloids and the connection matrices between bases of 
symmetric functions,'' \emph{Discrete Applied Mathematics} \textbf{34}
 (1991), 107--120.

\bibitem{stanvol2}
Richard Stanley, \emph{Enumerative Combinatorics} (vol. 2),
 Cambridge University Press, 1999.

\bibitem{wildon1}
Mark Wildon, ``A combinatorial proof of a plethystic
 Murnaghan--Nakayama rule,'' \emph{SIAM J. Discrete Math.} \textbf{30}
 (2016), 1526--1533.


\bibitem{turek}
Pavel Turek, ``Proof of the plethystic Murnaghan---Nakayama rule using
 Loehr's labeled abacus,'' \emph{Ann. Comb.} (2024),
\texttt{https://doi.org/10.1007/s00026-024-00698-y}.

\end{thebibliography}
\end{document}